\documentclass[11pt]{amsart}

\usepackage{amsmath,amssymb,amsthm,amsfonts,amsrefs}
\usepackage[margin=1in,marginparwidth=0.8in, marginparsep=0.1in]{geometry}
\usepackage{graphicx}
\usepackage[all]{xy}

\usepackage{enumerate}
\usepackage{tikz}
\usepackage{tikz-cd}
    \usetikzlibrary{arrows}

\usepackage{xcolor,float}
\usepackage{mathrsfs}
\usepackage{caption}

\usepackage{comment}

\usepackage{mathtools}
\usepackage[makeroom]{cancel}
\usepackage{stmaryrd}
\usepackage[bottom]{footmisc}
\usepackage{enumitem}
\usepackage{todonotes}
\usepackage[hidelinks]{hyperref}
\usepackage[title]{appendix}
\usepackage{float} 

\usepackage{fancyhdr}
\usepackage{mathrsfs}
    \newtheorem{definition}{Definition}[section]
    \newtheorem{theorem}[definition]{Theorem}
    \newtheorem{proposition}[definition]{Proposition}
    \newtheorem{lemma}[definition]{Lemma}
    \newtheorem{corollary}[definition]{Corollary}

\theoremstyle{definition}
    \newtheorem{example}[definition]{Example}
    \newtheorem{remark}[definition]{Remark}

\newtheorem{question}[definition]{Question}

\makeatletter
\newcommand*\bigcdot{ {\mathpalette\bigcdot@{.5}} }
\newcommand*\bigcdot@[2]{\mathbin{\vcenter{\hbox{\scalebox{#2}{$\m@th#1\bullet$}}}}}
\makeatother

\newcommand\RR{ \mathbb{R} }
\newcommand\SPH{ \mathbb{S} }
\newcommand\ZZ{ \mathbb{Z} }

\newcommand\cA{ \mathcal{A} }

\newcommand\cD{ \mathcal{D} }

\newcommand\cL{ \mathcal{L} }
\newcommand\cM{ \mathcal{M} }

\newcommand\cV{ \mathcal{V} }

\newcommand\cY{ \mathcal{Y} }

\newcommand\sA{ \mathscr{A} }
\newcommand\sB{ \mathscr{B} }
\newcommand\sC{ \mathscr{C} }

\newcommand\pre{ {\operatorname{pre}} }
\newcommand\Rp{ { \mathbb{R}_{>0} } }

\newcommand\Tr{\operatorname{Tr}}

\newcommand\Aut{\operatorname{Aut}}

\newcommand\End{\operatorname{End}}
\newcommand\inEnd{\operatorname{\underline{End}}}
\newcommand\Fun{ {\operatorname{Fun}} }
\newcommand\id{\operatorname{id}}
\newcommand\Id{\operatorname{Id}}
\newcommand\Ind{ {\operatorname{Ind} } }
\newcommand\Hom{\operatorname{Hom}}
\newcommand\inHom{\operatorname{\underline{Hom}} }
\newcommand\lp[1]{  {\mathcal{L}{#1}}  }

\newcommand\PrLcs{  {  \operatorname{Pr}^{\operatorname{L} }_{\omega,st} } }
\newcommand\PrLst{  {  \operatorname{Pr}^{\operatorname{L} }_{st} } }
\newcommand\Prop{\operatorname{Prop}}

\newcommand\colim{\operatorname{colim}}

\newcommand\fib{\operatorname{fib}}

\newcommand\Mod{\operatorname{Mod}}
\newcommand\Sp{ {\operatorname{Sp} } }
\newcommand\st{   {\operatorname{st}} }

\newcommand\Loc{ {\operatorname{Loc}} }
\newcommand\mhom{\operatorname{\mu hom}}
\newcommand\msh{\operatorname{\mu sh}}
\newcommand\ms{\operatorname{SS}}
\newcommand\msif{\operatorname{SS}^{\infty}}
\newcommand\msnz{ \dot{\operatorname{SS} }}
\newcommand\Op{\operatorname{Op}}

\newcommand\Sh{\operatorname{Sh}}

\newcommand\sHom{\mathscr{H}om}

\newcommand\wrap{\mathfrak{W}}

\newcommand\Coh{\operatorname{Coh}}

\newcommand\Spec{\operatorname{Spec}}

\newcommand\HH{ {\operatorname{HH} } }

\newcommand\PrLV[1][\cV]{  {  \operatorname{Pr}^{\operatorname{L} }_{\cV,st} } }

\newcommand\SD[1]{ {\operatorname{D_{#1}  }  } }

\newcommand\VD[1]{ {\operatorname{D}_{#1}  } }
\newcommand\LD[1]{ {\operatorname{D}'_{#1}  } }
\newcommand\dT{ {\dot{T} } }

\newcommand{\SC}{\mathcal{C}}

\newcommand{\bC}{\mathbb{C}}

\newcommand{\bN}{\mathbb{N}}

\newcommand{\bR}{\mathbb{R}}

\newcommand{\Arr}{\operatorname{\mathrm{Arr}} }
\newcommand{\Iduc}{\operatorname{\Id_{1_\SC}} }

\begin{document}


\title[Relative Calabi--Yau structure on microlocalization]{\textbf{Relative Calabi--Yau structure on microlocalization}\\\vspace{3mm} {\textbf{\footnotesize{-- An approach by contact isotopies --}}}}
\date{}
\author{Christopher Kuo}
\address{Department of Mathematics, University of Southern California}
\email{chrislpkuo@berkeley.edu}
\author{Wenyuan Li}
\address{Department of Mathematics, University of Southern California.}
\email{wenyuan.li@usc.edu}
\maketitle

\begin{abstract}
    For an oriented manifold $M$ and a compact subanalytic Legendrian $\Lambda \subseteq S^*M$, we construct a canonical strong smooth relative Calabi--Yau structure on the microlocalization at infinity and its left adjoint $m_\Lambda^l: \msh_\Lambda(\Lambda) \rightleftharpoons \Sh_\Lambda(M)_0 : m_\Lambda$ between compactly supported sheaves on $M$ with singular support on $\Lambda$ and microsheaves on $\Lambda$. We also construct a canonical strong Calabi--Yau structure on microsheaves $\msh_\Lambda(\Lambda)$. Our approach does not require local properness and hence does not depend on arborealization. We thus obtain a canonical smooth relative Calabi--Yau structure on the Orlov functor for wrapped Fukaya categories of cotangent bundles with Weinstein stops, such that the wrap-once functor is the inverse dualizing bimodule.
\end{abstract}

\tableofcontents

\section{Introduction}
    
\subsection{Context and background}
    This paper is the third in a series of study, following \cite{Kuo-Li-spherical, Kuo-Li-duality}. Our objective is to further investigate non-commutative geometric structure of the category of sheaves arising from the symplectic geometry of $(T^*M, \Lambda)$, where $\Lambda \subseteq S^*M$ is a subanalytic Legendrian subset in the ideal contact boundary $S^*M$ of the symplectic manifold $T^*M$. In this paper, we construct the relative Calabi--Yau structure, which menifests as a non-commutative analogue of Poincar\'e--Lefschetz duality.

    Following Kashiwara--Schapira \cite{KS}, given a real analytic manifold $M$, one can define a stable $\infty$-category $\Sh_\Lambda(M)$ of constructible sheaves on $M$ with subanalytic Legendrian singular support $\Lambda \subseteq S^*M$. One can also define another stable $\infty$-category $\msh_\Lambda(\Lambda)$ of microlocal sheaves on the Legendrian $\Lambda \subseteq S^*M$ \cite{Gui,Nadler-pants}, and there is a microlocalization functor with a left adjoint
    $$m_\Lambda^l: \msh_\Lambda(\Lambda) \rightleftharpoons \Sh_\Lambda(M) : m_\Lambda.$$

    The category of (microlocal) sheaves is found to be closely related to a number of central topics in mathematical physics, since Nadler--Zaslow discovered the relation between constructible sheaves and Fukaya categories on the cotangent bundle \cite{NadZas,Nad}.
    Recently, Ganatra--Pardon--Shende \cite{Ganatra-Pardon-Shende3} showed that the partially wrapped Fukaya category is equivalent to the category of compact objects in the unbounded dg category of sheaves. In particular, for contangent bundles with Legendrian stops,
    $$\mathrm{Perf}\,\mathcal{W}(T^*M, \Lambda)^\text{op} \simeq \Sh^c_\Lambda(M).$$
    The homological mirror symmetry conjecture \cite{KonHMS,AurouxAnti} predicts an equivalence of the Fukaya categories and categories of coherent sheaves on the mirror complex varieties. 
    On the other hand, in geometric representation theory, the Betti geometric Langlands program \cite{Ben-Zvi-Nadler-Betti} conjectures an equivalence of constructible sheaves on $\mathrm{Bun}_G(X)$ and quasi-coherent sheaves on $\mathrm{Loc}_{G^\vee}(X)$ with nilpotent microsupports.

    From the perspective of non-commutative geometry, categories of (microlocal) sheaves are expected to carry a relative Calabi--Yau structure, a notion introduced by Brav--Dyckerhoff \cite{Brav-Dyckerhoff1}. This is a non-commutative analogue of orientations that induce the Poincar\'{e}--Lefschetz duality and fiber sequence. Shende--Takeda have shown the existence of the Calabi--Yau structure using arborealization and explicit local computations on arboreal singularities, where they crucially require local properness of the sheaf of categories \cite{Shende-Takeda}. However, the global symplectic geometric and sheaf theoretic interpretation of the structure remains unknown, and it is also unclear whether their construction depends on the choice of arboreal skeleta. 
    
    In this paper, we will show a canonical relative Calabi--Yau structure and provide global geometric interpretations on the functors and transformations involved in the structure.

\subsection{Main results and corollaries}
    We will consider a compactly generated rigid symmetric monoidal stable $\infty$-category $\cV$ as our coefficient (for example, $\cV$ can be modules over any field or modules over any rigid $E_\infty$-ring spectrum).
    
    We now state the main result in the paper on relative Calabi--Yau structures, which will be elaborated further in Theorems \ref{thm:weak-cy-intro} and \ref{thm:strong-cy-intro}. As we have explained, this is a non-commutative version of orientations that induce the Poincar\'e--Lefschetz duality \cite{Brav-Dyckerhoff1}. We will consider the category of compactly supported sheaves on $M$ with microsupport on $\Lambda$, denoted by $\Sh_\Lambda(M)_0$, and the category of microsheaves on $\Lambda$, denoted by $\msh_\Lambda(\Lambda)$.

\begin{theorem}
    Let $M$ be an oriented real analytic manifold and $\Lambda \subseteq S^*M$ a compact subanalytic Legendrian. There exists a strong relative smooth Calabi--Yau structure on the adjunction pair of the microlocalization at infinity and the left adjoint
    $$m_\Lambda^l: \msh_\Lambda(\Lambda) \rightleftharpoons \Sh_\Lambda(M)_0 : m_\Lambda.$$
    There is also a strong smooth Calabi--Yau structure on $\msh_\Lambda(\Lambda)$. The Calabi--Yau structures are canonical under Legendrian deformations.
\end{theorem}
\begin{remark}
    When $M = N \times \bR$ and $\Lambda \subseteq J^1(N) \subseteq S^*(N \times \bR)$, our result does not require the Legendrian to be horizontally displaceable in the sense of \cite{EESduality}. In fact, there are horizontally non-displaceable Legendrians in $J^1(N)$ such that $\Sh_\Lambda(N \times \bR)_0$ is non-trivial. However, for certain horizontally non-displaceable Legendrians like the zero section $N \subset J^1(N)$, it is true that $\Sh_\Lambda(N \times \bR)_0$ is trivial, and the relative Calabi--Yau structure reduces to the absolute Calabi--Yau structure on microsheaves. See \cite{LiEstimate} for detailed discussions.
\end{remark}

Let $\Sh_\Lambda^b(M)_0$ and $\msh_\Lambda^b(\Lambda)$ be the full subcategories of (micro)sheaves with perfect stalks, which are equivalent to the proper modules (or pseudo-perfect modules) of the categories of compact objects. Then we automatically get the following result on proper Calabi--Yau structures.

\begin{corollary}
    Let $M$ be an oriented real analytic manifold and $\Lambda \subseteq S^*M$ a compact subanalytic Legendrian. There exists a strong relative proper Calabi--Yau structure on the microlocalization at infinity functor
    $$m_\Lambda: \Sh_\Lambda^b(M)_0 \to \msh_\Lambda^b(\Lambda).$$
    There is also a strong proper Calabi--Yau structure on $\msh_\Lambda^b(\Lambda)$. The Calabi--Yau structures are canonical under Legendrians deformations.
\end{corollary}

    Using the result of Ganatra--Pardon--Shende \cite[Theorem 1.1, 1.4 \& Proposition 7.24]{Ganatra-Pardon-Shende3}, we can deduce a strong relative Calabi--Yau structure on the Orlov functor on partially wrapped Fukaya categories of the cotangent bundle with Weinstein stops.

\begin{corollary}
    Let $M$ be a closed oriented real analytic manifold and $F \subseteq S^*M$ a Weinstein hypersurface. There exists a canonical strong relative smooth Calabi--Yau structure on the Orlov functor
    $$\cup_F: \mathcal{W}(F) \to \mathcal{W}(T^*M, F),$$
    and there is a canonical strong smooth Calabi--Yau structure on $\mathcal{W}(F)$.
\end{corollary}
\begin{remark}
    Using the Legendrian surgery formula of Asplund--Ekholm \cite{AsplundEkholm} built on previous results \cite{EkholmLekili,EkholmHolo}, we know that, when $M$ is an open manifold, our result also deduces a strong relative smooth Calabi--Yau structure on Legendrian contact dg algebras of a singular Legendrian $\Lambda \subseteq S^*M$ (with the Liouville filling $T^*M$ viewed as a subcritical Weinstein manifold) and the attaching spheres $S$ of the top dimensional Legendrian handles
    $$\mathcal{A}(S) \to \mathcal{A}(\Lambda, S).$$
    We conjecture this to be compatible with the weak relative Calabi--Yau structure considered by Dimitroglou Rizell--Legout \cite{DRizellLegout} and the absolute Calabi--Yau structure on the attaching spheres by Legout \cite{Legout}.
    
    In particular, when $\Lambda$ is smooth, $\mathcal{A}(S) \simeq C_{-*}(\Omega_*\Lambda)$ and $\mathcal{A}(\Lambda, S) \simeq \mathcal{A}_{C_{-*}(\Omega_*\Lambda)}(\Lambda)$, though one will still need to compare the Calabi--Yau structure on $\cA(S)$ coming from symplectic geometry and the one on $C_{-*}(\Omega_*\Lambda)$ coming from string topology (this is not done in this paper).
\end{remark}
\begin{remark}
    For Legendrian links $\Lambda \subseteq \bR^3 \subseteq S^*\bR^2$, using the result that augmentations are sheaves \cite{AugSheaf}, we can deduce a proper Calabi--Yau structure on the $\mathcal{A}_\infty$-category of augmentations
    $$\mathcal{A}ug_+(\Lambda) \to \Loc^1(\Lambda).$$
    We conjecture this to be compatible with the relative Calabi--Yau structure considered by Ma--Sabloff \cite{MaSabloff} and independently by Chen \cite{Chen}.
    
    For general singular Legendrians in $S^*M$, using the Legendrian surgery formula, we can also deduce that there is a proper Calabi--Yau structure on finite dimensional representations (in particular augmentations)
    $$\Mod^\text{fd}(\mathcal{A}(\Lambda, S)) \to \Mod^\text{fd}(\mathcal{A}(S)).$$
    However, we remark that though objects are the same, we do not know from Legendrian surgery formula whether morphisms in the dg category $\Mod^1(\mathcal{A}_{\Bbbk[\pi_1(\Lambda)]}(\Lambda))$ agrees with morphisms in the $\cA_\infty$-category $\cA ug_+(\Lambda)$ defined in \cite{AugSheaf}\footnote{On the level of $\cA_\infty$-structures, only the comparison of the endomorphism of a single object is expected to follow from existing literature \cite[Proposition 14, Remark 27 \& Theorem 51]{EkholmLekili}; see also Dimitroglou Rizell's lecture note \href{https://conference.math.muni.cz/srni/files/archiv/2022/plenary\%20lectures/}{\textit{DG-Algebraic Aspects of Contact Invariants 3}} Section 5 for the version without loop space coefficients. However, for general morphisms between different objects, we do not know a written-down proof in the literature.}.
\end{remark}

We now explain our main result, which breaks down into two concrete parts as follows.

First, the weak smooth relative Calabi--Yau structure refers to a duality and fiber sequence that relates the identity bimodule and the inverse dualizing bimodule induced from a fundamental class in the relative Hochschild homology.

We construct the weak relative Calabi--Yau structure by showing the following functorial enhancement of the Sato--Sabloff fiber sequence (Theorem \ref{thm:sato-sabloff}) and the Sabloff--Serre duality (Theorem \ref{thm:sabloff-serre}) \cite[Section 4]{Kuo-Li-spherical}. Let $S_\Lambda^+ : \Sh_\Lambda(M)_0 \to \Sh_\Lambda(M)_0$ be the wrap-once functor, defined by a small positive Hamiltonian push-off of $\Lambda$ followed by the colimit of all positive Hamiltonian push-offs supported away from $\Lambda$ (Definition \ref{def:wrap-once}) \cite{Kuo-Li-spherical,Kuo-Li-duality}. We show that the wrap-once functor is the inverse dualizing bimodule, therefore proving a generalized version of the conjecture on inverse Serre functor by Seidel \cite{SeidelSH=HH}.

\begin{theorem}\label{thm:weak-cy-intro}
    Let $M$ be an $n$-dimensional oriented real analytic manifold and $\Lambda \subseteq S^*M$ a compact subanalytic Legendrian. Then there is a commutative diagram of colimit preserving functors on $\Sh_\Lambda(M)_0$ where the horizontal arrows induce fiber sequences
    \[\xymatrix{
    \Id_{\Sh_\Lambda(M)_0}^! \ar[r] \ar[d]^{\rotatebox{90}{$\sim$}} & m_\Lambda^l \Id_{\msh_\Lambda(\Lambda)}^! m_\Lambda \ar[r] \ar[d]^{\rotatebox{90}{$\sim$}} & (S_\Lambda^+)^![1] \ar[d]^{\rotatebox{90}{$\sim$}} \\
    S_\Lambda^+[-n] \ar[r] & m_\Lambda^l \Id_{\msh_\Lambda(\Lambda)} m_\Lambda[1-n] \ar[r] & \Id_{\Sh_\Lambda(M)_0}[1-n].
    }\]
    Here $\Id_{\Sh_\Lambda(M)_0}$ and $\Id_{\msh_\Lambda(\Lambda)}$ are the identity bimodules while $\Id_{\Sh_\Lambda(M)_0}^!$ and $\Id_{\msh_\Lambda(\Lambda)}^!$ are their inverse dualizing bimodules.
\end{theorem}
\begin{remark}
    The horizontal fiber sequence has been obtained in our previous paper as part of spherical adjunction \cite{Kuo-Li-spherical}. However, the vertical isomorphism is a new result, which upgrades the Serre duality statement in \cite{Kuo-Li-spherical}. Our spherical adjunction and Serre duality statement \cite{Kuo-Li-spherical} requires $\Lambda \subseteq S^*M$ to be either swappable or fully stopped, but the relative Calabi--Yau property requires nothing other than compactness. This is because spherical adjunction asks the cotwist $S_\Lambda^+$ to be an equivalence, which is an extra requirement beyond relative Calabi--Yau. Spherical adjunction plus Serre duality and is only related to weak relative Calabi--Yau in the case of smooth and proper categories \cite[Theorem 1.18]{KPSsphericalCY}, which does not hold for $\Sh_\Lambda(M)_0$ in general.
\end{remark}

Second, the strong relative Calabi--Yau structure is an enhancement of the weak structure which means a lifting of the fundamental class from the relative Hochschild homology to the relative negative cyclic homology. 

We lift the weak Calabi--Yau structure to a strong one by understanding the circle action on the Hochschild homology of $\msh_\Lambda(\Lambda)$ and $\Sh_\Lambda(M)$ through sheaf theoretic operations. The key result is the construction of $S^1$-equivariant inclusions of constant orbits (also known as the acceleration morphism in Floer theory) in which the fundamental class lies.

\begin{theorem}\label{thm:strong-cy-intro}
    Let $M$ be an $n$-dimensional oriented real analytic manifold, $\Lambda \subseteq S^*M$ be a compact subanalytic Legendrian, and $\widehat\Lambda \subseteq T^*M$ be the conification of $\Lambda$ containing all compact strata in the zero section. Then there is a commutative diagram
    \[\xymatrix{
    \Gamma(\widehat\Lambda, 1_{\widehat\Lambda}) \ar[r] \ar[d] & \Gamma(\Lambda, 1_\Lambda) \ar[r] \ar[d] & \Gamma(\widehat\Lambda, 1_{\widehat\Lambda \setminus \Lambda})[1] \ar[d] \\
    \HH_*(\msh_\Lambda(\Lambda), \Sh_\Lambda(M)_0)[-n] \ar[r] & \HH_*(\msh_\Lambda(\Lambda))[1-n] \ar[r] & \HH_*(\Sh_\Lambda(M)_0)[1-n].
    }\]
    Here, all the morphisms are $S^1$-equivariant and the $S^1$-actions on the top are the trivial actions.
\end{theorem}
\begin{remark}
    Using the cyclic Deligne conjecture recently established by Brav--Rozenblyum \cite{BravRozenblyum}, we know that the (relative) Hochschild homologies and cohomologies are isomorphic and have the structure of framed $E_2$-algebras. In fact, we expect that the above commutative diagram is a also diagram of framed $E_2$-algebras.
\end{remark}

\begin{remark}
    We remark on the relationship between the acceleration morphisms and global (co)sections of (co)sheafified Hochschild homology that Shende--Takeda constructed \cite{Shende-Takeda}. For arboreal singularities, we expect that they are compatible. However, for general subanalytic isotropics, we think they will not be. Indeed, we think the global (co)section of the (co)sheafified Hochschild homology is not invariant under non-characteristic deformations of Legendrians. See Remark \ref{rem:shende-takeda}.
\end{remark}

\subsection{Some examples and applications}
We explain some examples and applications of our result of relative Calabi--Yau structures. 

First, our result can recover known constructions of relative Calabi--Yau structures in topology:
\begin{enumerate}
    \item the pair $\Loc(\partial M) \rightleftharpoons Loc(M)$ where $(M, \partial M)$ is an oriented manifold with boundary \cite{Brav-Dyckerhoff1};
    \item the pair $\bigoplus_{i=1}^k \Bbbk[\mu_i^{\pm 1}] \oplus \Bbbk[\lambda^{\pm 1}] \rightleftharpoons \mathscr{A}(D^2, x_1, \dots, x_k)$ where $\mathscr{A}(D^2; x_1, \dots, x_k)$ is the derived relative multiplicative preprojective algebra \cite{BezrukavnikovKapranov,BozecCalaqueScherotzke};
    \item the pair $\Loc(S^*_NM) \oplus \Loc(\partial M) \rightleftharpoons \mathscr{A}(M, N)$ where $M$ is an oriented manifold, $N \subset M$ is a submanifold and $\mathscr{A}(M, N)$ is the link dg category in Berest--Eshmatov--Yeung \cite{BerestEshmatovYeung,YeungComplete}.
\end{enumerate}

Second, our result is related to constructions of relative Calabi--Yau structures in algebraic geometry. Under homological mirror symmetry for toric varieties (the coherent-constructible correspondence) \cite{FLTZCCC,KuwaCCC,GammageShende}, we expect the relative Calabi--Yau structure on $\Sh_{\Lambda_\Sigma}(T^n) \rightleftharpoons \msh_{\Lambda_\Sigma}(\Lambda_\Sigma)$ to agree with the relative Calabi--Yau structure on $\Coh(X_\Sigma) \rightleftharpoons \Coh(\partial X_\Sigma)$. However, the comparison will be the topic for future studies.

Furthermore, building on the result of Brav--Dyckerhoff that Calabi--Yau categories induce shifted symplectic structures on the derived moduli stacks, and relative Calabi--Yau functors induce shifted Lagrangian structures between derived moduli stacks \cite{Brav-Dyckerhoff2}, we can show existence of symplectic structures and Lagrangian structures on the moduli spaces of constructible sheaves: 
\begin{enumerate}
    \item[(4)] the pair $\Loc(\Lambda) \rightleftharpoons \Sh_\Lambda(\Sigma)_0$ where $\Sigma$ is a surface with puncture and $\Lambda \subset S^*\Sigma$ is a Legendrian knot with vanishing Maslov class \cite{Shende-Treumann-Zaslow,Shende-Takeda}, which induces a symplectic structure on the moduli space of sheaves or the augmentation variety \cite{CGNSW};
    \item[(5)] the pair $\Loc(\Lambda) \rightleftharpoons \Sh_\Lambda(M)_0$ where $M$ is a 3-manifold and $\Lambda \subset S^*\Sigma$ is a Legendrian surface with vanishing Maslov class, which induces a Lagrangian structure on the map of moduli spaces of sheaves or the augmentation variety to the moduli of local systems \cite{AganagicEkholmNgVafa,Treumann-Zaslow}.
\end{enumerate}
We mention that Case (4) includes examples of the moduli space of framed local systems, the moduli space of Stokes data \cite{Shende-Treumann-Williams-Zaslow} and the symplectic leaves of Richardson and braid varieties \cite{CGGS}; Case (5) includes examples of the augmentation varieties of knot conormal tori \cite{AganagicEkholmNgVafa}, the Aganagic--Vafa Lagrangian and more generally the chromatic Lagrangian branes \cite{Treumann-Zaslow,SchraderShenZaslow}.

\subsection{Related works and discussion}
Categories with Calabi--Yau structures are expected to define a 2D topological field theory \cite{Kontsevich-Soibelman-Ainfty,Costello,LurieCobordism,Kontsevich-Takeda-Vlassopoulos1}. Fukaya categories correspond to the 2D topological A theory, whereas the mirror coherent sheaf categories correspond to the 2D topological B theory. This leads to numerous studies of Calabi--Yau properties arising from symplectic geometry.

Using arborealization for Legendrian stops in cotangent bundles \cite{Nadler-nonchar} and positively polarizable Weinstein manifolds with stops \cite{Arborealization}, it is known that strong (relative) Calabi--Yau structures exist:
\begin{enumerate}
    \item Shende--Takeda \cite{Shende-Takeda} constructed strong smooth/proper (relative) Calabi--Yau structures on the compact/proper objects in the microlocal sheaf categories for the arboreal skeleton.
\end{enumerate}
Other than the arborealization approach, there are also some known results on absolute Calabi--Yau structures in symplectic geometry using Floer theory:
\begin{enumerate}
    \item[(2)] For the wrapped Fukaya categories of non-degenerate Liouville manifolds $\mathcal{W}(X)$, Ganatra \cite{Ganatra1,Ganatra2} showed the existence of strong smooth Calabi--Yau structures on $\mathcal{W}(X)$.
    \item[(3)] For compact Fukaya categories of closed symplectic manifolds $\mathcal{F}(X)$, Abouzaid--Fukaya--Oh--Ohta--Ono \cite{AFOOO}, following Fukaya \cite{Fukaya-cyclic}, will construct a cyclic $\mathcal{A}_\infty$-structure on $\mathcal{F}(X)$, which, in characteristic 0, is weakly equivalent to a strong proper Calabi--Yau structure. Ganatra \cite{Ganatra2} showed the existence of strong proper Calabi--Yau structures over any field.
    \item[(4)] For the Legendrian contact dg algebra $\mathcal{A}(\Lambda)$, Legout \cite{Legout} showed the existence of weak smooth Calabi--Yau structures on $\mathcal{A}(\Lambda)$ for a Legendrian sphere $\Lambda$.
    \item[(5)] For the Tamarkin categories of compact domains $U \subseteq T^*M$ with contact boundary, Kuo--Shende--Zhang \cite{KuoShendeZhang} showed the existence of strong proper Calabi--Yau structures on $\mathcal{T}(U)$.
\end{enumerate}
On the other hand, less is known on the relative Calabi--Yau structures in symplectic geometry:
\begin{enumerate}
    \item[(6)] For the Fukaya--Seidel categories of Lefschetz fibrations $\mathcal{F}(X, \pi)$ and compact Fukaya categories of the fiber $\mathcal{F}(F)$, Seidel \cite{SeidelFukI} showed the existence of a weak relative proper Calabi--Yau structure on the cap functor $\cap_F: \mathcal{FS}(X, \pi) \to \mathcal{F}(F)$.
    \item[(7)] For the partially wrapped Fukaya categories of Weinstein manifolds with stops $\mathcal{W}(X, F)$ and wrapped Fukaya categories of the stop $\mathcal{W}(F)$, no results are known using Floer theory.
    \item[(8)] For Legendrian contact dg algebras over chains over based loop spaces $\mathcal{A}_{C_{-*}(\Omega_*\Lambda)}(\Lambda)$, the upcoming work of Ma--Sabloff \cite{MaSabloff} and Chen \cite{Chen}, following Ekholm--Etnyre--Sabloff \cite{Sabduality,EESduality}, will show a weak relative proper Calabi--Yau structure on category of augmentations $\mathcal{A}ug_+(\Lambda) \to \Loc_1(\Lambda)$, and an upcoming work of Dimitroglou Rizell--Legout \cite{DRizellLegout} will show a weak relative smooth Calabi--Yau structure on $C_{-*}(\Omega_*\Lambda) \to \mathcal{A}_{C_{-*}(\Omega_*\Lambda)}(\Lambda)$. 
    \item[(9)] Asplund \cite{Asplund-CY} showed a strong smooth Calabi--Yau structure for a certain class of examples of singular Legendrians by showing that its dg algebra is the relative Ginzburg algebra.
\end{enumerate}
Meanwhile, there are also certain structures closely related to the Calabi--Yau structure. For instance, Kontsevich--Takeda--Vlassopoulos \cite{Kontsevich-Takeda-Vlassopoulos1} defined the notion of pre-Calabi--Yau structures (Seidel \cite{SeidelFukI} also defined the equivalent notion of boundary algebras for finite dimensional $\mathcal{A}_\infty$-algebras), and Calabi--Yau structures induce pre-Calabi--Yau structures \cite{Kontsevich-Takeda-Vlassopoulos2}.
\begin{enumerate}
    \item[(10)] For the Fukaya--Seidel categories of Lefschetz fibrations $\mathcal{F}(X, \pi)$, Seidel \cite{SeidelFukI} showed that there is a pre-Calabi--Yau structure as well, which we expect to be closely related to the strong proper relative Calabi--Yau structure.
\end{enumerate}
Finally, given a strong (relative) Calabi--Yau structure, there is a framed $E_2$-structure on the (relative) Hochschild cohomology of the categories by Brav--Nozenblyum \cite{BravRozenblyum}. Related structures have also been studied in symplectic geometry.
\begin{enumerate}
    \item[(11)] For wrapped Fukaya categories of nondegenerate Liouville manifolds $\mathcal{W}(X)$, Abouzaid--Groman--Varolgunes \cite{AbGroVarol} showed that there is a framed $E_2$-structure on the symplectic cohomology $SH^*(X)$, which is isomorphic to the Hochschild cohomology.
    \item[(12)] For Legendrian contact dg algebra $\mathcal{A}_{\Bbbk[\pi_1(\Lambda)]}(\Lambda)$, Ng \cite{Ng-Linfinity} showed that for $\Lambda \subseteq \bR^3$ there is an $L_\infty$-structure on the commutative version $\mathcal{A}^\text{comm}_{\Bbbk[\pi_1(\Lambda)]}(\Lambda)$, which we expect to be related to the framed $E_2$-structure and pre-Calabi--Yau structure on the Hochschild cohomology.
    \item[(13)] For Legendrian contact dg algebra $\mathcal{A}_{\Bbbk[\pi_1(\Lambda)]}(\Lambda)$, Casals--Gao--Ng--Shen--Weng \cite{CGNSW} will show that for a large class of $\Lambda \subseteq \bR^3$ there is a holomorphic symplectic structure on the augmentation variety $\Spec H_0(\mathcal{A}^\text{comm}_{\Bbbk[\pi_1(\Lambda)]}(\Lambda))$, which we expect to follow from the shifted Lagrangian intersection of the moduli spaces when $\Lambda$ is a Legendrian knot.
\end{enumerate}

Our result shows the existence of strong relative smooth Calabi--Yau structures, which induce strong relative proper Calabi--Yau structures. Our construction does not depend on the local properness and local smoothness and does not depend on computations on explicit models of arboreal Lagrangian skeleta as in \cite{Shende-Takeda}. Actually, we believe that our approach aligns closely to the conjectural approach to relative Calabi--Yau structures in partially wrapped Floer theory, where some technical ingredients can be found in \cite{AbGan}.

\subsection*{Acknowledgement}
We would like to thank Mohammed Abouzaid, Roger Casals, Zhenyi Chen, Georgios Dimitroglou Rizell, John Francis, Marc Hoyois, Sheel Ganatra, Jason Ma, Lenny Ng, Nick Rozenblyum, Joshua Sabloff, Vivek Shende, Alex Takeda, Dima Tamarkin, Harold Williams, Eric Zaslow and Bingyu Zhang for helpful discussions. CK was supported by NSF grant DMS-1928930 when in residence at the SLMath during Spring 2024.

\section{Duality and Continuous Adjunction} \label{sec:duality bimodule}

We recall some algebraic background needed in the paper. Discussions on dualizability can be found in \cite{Hoyois-Scherotzke-Sibilla,Gaitsgory-Rozenblyum,Brav-Dyckerhoff2}, discussions on Hochschild invariants of dualizable objects can be found in \cite{Hoyois-Scherotzke-Sibilla,AyalaFrancis,Brav-Dyckerhoff2}, and discussions on Calabi--Yau structures can be found in \cite{Brav-Dyckerhoff1,Brav-Dyckerhoff2,BravRozenblyum}.

\subsection{Dualizability and adjunction}

Let $\SC$ be the symmetric monoidal $2$-category. (For our purpose, we only need the case when $\SC = \PrLst(\cV)$ for a rigid symmetric monoidal category $\cV$.) 
We recall the notion of dualizable objects, right and left dualizable morphisms and adjunctions, following \cite{JFS,Hoyois-Scherotzke-Sibilla,Brav-Dyckerhoff2}, where many results and arguments are parallel to the $1$-categorical setting \cite[Section 4.6]{Lurie-HA}; see also \cite[Chapter 1, Section 4]{Gaitsgory-Rozenblyum}.

\begin{definition}[{\cite[Definition 4.6.2]{Lurie-HA} \cite[Section 7]{JFS} \cite[Section 2.2]{Hoyois-Scherotzke-Sibilla}}]\label{smdual}
An object $X$ in $\SC$ is dualizable if there exists $Y \in \SC$ and a unit and a counit
$\eta: 1 \rightarrow X \otimes Y$ and $\epsilon: Y \otimes X \rightarrow 1$
such that the pair $(\eta,\epsilon)$ satisfies the standard triangle equality that the following compositions are identities
$$X \xrightarrow{\Id_X \otimes \, \eta} X \otimes Y \otimes X \xrightarrow{\epsilon\, \otimes \, \Id_X} X,$$
$$Y \xrightarrow{\eta\, \otimes \, \Id_Y} Y \otimes X \otimes Y \xrightarrow{\Id_Y \otimes \, \epsilon} Y.$$
\end{definition}
We note that these conditions uniquely classifies $(Y,\eta,\epsilon)$, and we will use the notation $(X^\vee, \eta_X, \epsilon_X)$ to denote its dual when there is no need to specify it. 

\begin{proposition}[{\cite[Proposition 4.6.2.1]{Lurie-HA} \cite[Section 2.1]{Brav-Dyckerhoff2}}]\label{dual-gives-inHom}
When $X$ is dualizable, the functor $X^\vee \otimes (-)$ is both the right and left adjoint of $X \otimes (-)$ since $(X^\vee)^\vee =X$.
In particular, for any $Y \in \SC$, we have $\inHom(X,Y) = Y \otimes X^\vee$.
\end{proposition}
\begin{remark}[{\cite[Equation (2.1) \& (2.2)]{Brav-Dyckerhoff2}}]\label{rem:dual-functor}
When $X$ and $Y$ are both dualizable, we thus have isomorphisms
$$\Hom(1_\SC, X^\vee \otimes Y) = \Hom(X, Y) = \Hom(X \otimes Y^\vee, 1_\SC),$$
where the morphism $f: X \to Y$ corresponds to 
\begin{align*}
\varphi_f: 1_\SC \xrightarrow{\eta_X} X \otimes X^\vee \xrightarrow{f \otimes \Id_{X^\vee}} Y \otimes X^\vee, \;\;
\psi_f: Y^\vee \otimes X \xrightarrow{\Id_{Y^\vee} \otimes f} Y^\vee \otimes Y \xrightarrow{\epsilon_Y} 1_\SC.
\end{align*}
In particular, under the equivalence 
$$\Hom(1_\SC, X \otimes X^\vee) = \Hom(X, X) = \Hom(X^\vee \otimes X, 1_\SC),$$
$\Id_X$ always corresponds to $\eta_X$ under the first isomorphism and $\epsilon_X$ under the second isomorphism.
\end{remark}

\begin{definition}[{\cite[Section 2.1]{Brav-Dyckerhoff2}}]\label{lem:dual-of-morphism}
Let $X$ and $Y$ be dualizable objects and $f: X \rightarrow Y$ be a morphism. 
We denote by $f^\vee: Y^\vee \rightarrow X^\vee$ by its imagine under the equivalence
$$\Hom(X,Y) = \Hom(1_\SC, X^\vee \otimes Y) = \Hom(Y^\vee, X^\vee).$$
We note that $(f^\vee)^\vee = f$.
\end{definition}

\begin{remark}[{\cite[Equation (2.3)]{Brav-Dyckerhoff2}}]\label{rem:dual-product}
Let $X$ and $Y$ be dualizable. Then $X \otimes Y$ is also dualizable and $(X \otimes Y)^\vee = Y^\vee \otimes X^\vee$.
In particular, the object $X^\vee \otimes X$ is self-dual and  $\eta_X^\vee = \epsilon_X$ so $\epsilon_X^\vee = \eta_X$.
\end{remark}

Now we recall the notion of adjoint morphisms, which we will be always working with.

\begin{definition}[{\cite[Section 7]{JFS}}]
A $1$-morphism $f: X \rightarrow Y$ is \textit{right adjoinable} if there exists a $1$-morphism $f^r: Y \rightarrow X$ and $2$-morphisms
$$u: \id_X \rightarrow f^r \circ f, \, c: f \circ f^r \rightarrow \id_Y$$
satisfying a similar triangle identities for dualizable objects.
The definition of being \textit{left adjoinable} is similar. In these cases, we say that $f^r$ (resp.~$f^l$) is the right (resp.~left) adjoint of $f$.
\end{definition}

For example, if $g: Z \rightarrow X$, $h: Z \rightarrow Y$ are morphisms, then the unit and counit induces
$$\Hom_{\Hom(Z,Y)}(h, f \circ g) = \Hom_{\Hom(Z,X)}(f^l \circ h, g)$$
which generalizes the standard identity for adjoint functors between categories.

\begin{definition}\label{def:push-forward}
We use the notation $F_!: \inHom(X,X) \rightarrow \inHom(Y,Y)$ to mean the morphism which,
under $\inHom(X,X) = X^\vee \otimes X$, corresponds to $f^{r\vee} \otimes f: X^\vee \otimes X \rightarrow Y^\vee \otimes Y$.
\end{definition}


We mention a lemma that will be convenient for us later.

\begin{lemma}[{\cite[Lemma 2.1]{Brav-Dyckerhoff2}}]\label{swapping}
Let $X$ and $Y$ be dualizable. For any morphism $f: X \rightarrow Y$ and $g: Y \rightarrow X$, we have 
\begin{align*}
\epsilon_Y \circ (\Id_{Y^\vee} \otimes f) = \epsilon_X \circ (f^\vee \otimes \Id_X): &\, Y^\vee \otimes X \rightarrow 1_\SC,\\ 
(g \otimes \Id_{Y^\vee}) \circ \eta_Y = (\Id_X \otimes \, g^\vee) \circ \eta_X: &\, 1_\SC \rightarrow X \otimes Y^\vee
\end{align*}
When $f: X \rightleftharpoons Y: f^r$ is an adjunction, we have
\begin{align*}
(f \otimes f^{r\vee}) \circ (\alpha \otimes \Id_{X^\vee}) \circ \eta_X = ((f\circ \alpha \circ f^r) \otimes \Id_{Y^\vee}) \circ \eta_Y &: 1_\SC \to Y \otimes Y^\vee.\\
\epsilon_Y \circ (\Id_{Y^\vee} \otimes\, \alpha) \circ (f^{\vee} \otimes f^r) = \epsilon_X \circ (\Id_{X^\vee} \otimes\, (f\circ \alpha \circ f^r)) &: X^\vee \otimes X \to 1_\SC.
\end{align*}
Finally, a morphism $f: X \rightarrow Y$ is right adjoinable 
if and only if $f^\vee: Y^\vee \rightarrow X^\vee$ is left adjoinable, and in this case, $(f^\vee)^l = (f^r)^\vee$.
\end{lemma}

Then we recall the left and right dualizing morphism $f: X \to Y$ between dualizable objects. See \cite[Section 2.2]{Brav-Dyckerhoff2} and \cite[Proposition 4.6.4.4 \& 4.6.4.12]{Lurie-HA}.

\begin{definition}[{\cite[Section 2.2]{Brav-Dyckerhoff2}}]\label{def:dualizing-functor}
Let $X, Y \in \SC$ be dualizable, $f: X \rightarrow Y$ be a morphism and $\varphi_f: X \otimes Y^\vee \rightarrow 1_\SC$ its corresponding morphism under
$$\Hom(X,Y) = \Hom(X \otimes Y^\vee, 1_\SC).$$ 
If $\varphi_f$ is left dualizable, the left dualizing morphism of $f$, denoted by $f^! \in \Hom(X, Y)$, is the morphism corresponding to $\varphi_f^l$. 
Let $\psi_f: 1_\SC \to X^\vee \otimes Y$ the corresponding morphism of $f$ under 
$$\Hom(X,Y) = \Hom(1_\SC, X^\vee \otimes Y).$$ 
If $\psi_f$ is right dualizable, then the right dualizing morphism of $f$, denoted by $f^* \in \Hom(X, Y)$, is the morphism corresponding to $\psi_f^r$.
\end{definition}

\begin{definition}[{\cite[Section 2.2]{Brav-Dyckerhoff2}}]
We say that a dualizable object $X$ is smooth (resp.~proper) if $\Id_X: X \to X$ 
is left (resp.~right) dualizable. Equivalently, $X$ is smooth (resp.~proper) if $\epsilon_X: X \otimes X^\vee \to 1_\SC$ is left (resp.~right) adjoinable, or $\eta_X: 1_\SC \rightarrow X^\vee \otimes X$
is right (resp.~left) adjoinable.
\end{definition}

\begin{example}
When $X$ is smooth, the endomorphism $\Id_X^!: X \rightarrow X$ is given by the composition
$$X \xrightarrow{\Id_X \otimes \, \epsilon_X^l} X \otimes X^\vee \otimes X \xrightarrow{\epsilon_X \otimes \, \Id_X} X.$$
\end{example}

\begin{lemma} \label{rem:bimod-adjunction}
Assume $X$ is smooth and $f: X \rightarrow Y$ be right adjointable. Then 
$$(ff^r)^! = f \circ \Id_X^! \circ f^r.$$
\end{lemma}

\begin{proof}
It is a direct computation using Lemma \ref{swapping}.
\end{proof}

\subsection{Hochschild homology and circle action}
For this section, we assume that $(\SC, \otimes, 1)$ is a symmetric monoidal category that admits geometric realization. (We only need the case when $\SC = \PrLst(\cV)$, which clearly satisfies the condition.) We recall the construction of Hochschild homology and circle actions on Hochschild homology, following \cite{Hoyois-Scherotzke-Sibilla,AyalaFrancis}.

\begin{definition}[{\cite[Section 2.2]{Hoyois-Scherotzke-Sibilla}}]
For a dualizable object $X$ and an endomorphism $f$, 
we define its trace $\Tr(f) \in \End(1_\SC)$ to be the composition
$$1_\SC \xrightarrow{\eta_X} X \otimes X^\vee \xrightarrow{f \otimes \,\Id} X \otimes X^\vee = X^\vee \otimes X \xrightarrow{\epsilon_X} 1_\SC.$$
For $f = \Id_X$, we call $\Tr(\Id_X)$ the Hochschild homology of $X$ and denote it by $\HH_*(X)$.
\end{definition} 

Assume that $X$ is smooth. Then there is an equivalence
\begin{align*}
\Hom_{\End(X)}(\Id_X^!, \Id_X) &= \Hom_{\Hom(1_\SC, X \otimes X^\vee) }(\epsilon_X^l, \eta_X) 
= \Hom_{\Hom(1_\SC, X \otimes X^\vee) }(\epsilon_X^l \circ \Id_{1_\SC}, \eta_X) \\
&= \Hom_{\End(1_\sC)}(\Id_{1_\SC}, \epsilon_X \circ \eta) = \Hom_{\End(1_\sC)}\left(\Id_{1_\SC}, \HH_*(X) \right).
\end{align*} 
Similarly, assume that $X$ is proper. Then 
$$\Hom_{\End(X)}(\Id_X^*, \Id_X) = \Hom_{\End(1_\SC)}(\HH_*(X), \Id_{1_\SC}).$$

For an adjunction $f: X \rightleftharpoons Y: f^r$, there is a map $\HH_*(X) \rightarrow \HH_*(Y)$ which is given by the composition \cite[Lemma 4.1]{Brav-Dyckerhoff2}
$$\HH_*(X) =\Tr(\Id_X) \to \Tr(\Id_X f^rf) \xrightarrow{\sim} \Tr(f \Id_X f^r) \xrightarrow{\sim} \Tr(\Id_Y ff^r) \to \Tr(\Id_Y) = \HH_*(Y).$$
Therefore, Hochschild homology is functorial under right dualizable morphisms, and one can define the relative Hochschild homology as follows.

\begin{definition}[{\cite[Section 2.3.3]{BravRozenblyum}}]
Let $f: X \rightleftharpoons Y: f^r$ be an adjunction.
Then the relative Hochschild homology $\HH_*(Y,X)$ is defined to be the cofiber of $\HH_*(X) \to \HH_*(Y)$, i.e.,
$\HH_*(Y,X)$ fits in the fiber sequence
$$\HH_*(X) \rightarrow \HH_*(Y) \rightarrow \HH_*(Y,X).$$
Here we surpass the dependence of $f$ in the notation.
\end{definition}

As shown in \cite[Theorem 2.14]{Hoyois-Scherotzke-Sibilla}, when $f$ is an automorphism, 
$\Tr(f) \in \End_\SC(1)$ admits a canonical functorial $S^1$-action.
We recall the construction of $S^1$-action using factorization homology by Ayala--Francis \cite{AyalaFrancisTrace}, following earlier works of \cite{AyalaFrancis,AyalaMazel-GeeRozenblyum}\footnote{We would like to thank Nick Rozenblyum for pointing out the reference of the construction, and we would like to thank Nick Rozenblyum and John Francis for explaining to us these works.}.

Recall the combinatorial model of the cyclic category $\boldsymbol{\Lambda}$, whose objects are $n \in \bN$ and morphisms are generated by
\begin{align*}
    d_i: [n] \to [n+1], \;s_i: [n+1] \to [n], \;t_{n}: [n] \to [n]
\end{align*}
such that the face morphism $d_i$, the degeneration morphism $s_i$ and the cyclic rotation $t_n$ satisfy the relations in the cyclic category (with the relation $t_n^n = \id$) \cite[Definition 6.1.1]{LodayCyclic} \cite[Section 3.1]{Kaledin}. Also recall the para-cyclic category $\boldsymbol{\Delta}_\circlearrowleft$ (also denoted by $\boldsymbol{\Lambda}_\infty$), whose objects are $n \in \bN$ and morphisms are generated by
\begin{align*}
    d_i: [n] \to [n+1], \;s_i: [n+1] \to [n], \;t_{n}: [n] \to [n]
\end{align*}
such that the face morphism $d_i$, the degeneration morphism $s_i$ and the cyclic rotation $t_n$ satisfy the relations in the para-cyclic category (without the relation $t_n^n = \id$) \cite[Section 1]{GetzlerJones} \cite[Section 3.1]{Kaledin}.

\begin{definition}
    Let $\boldsymbol{\Lambda}$ be the cyclic category. Then a cyclic object with value in $\cV$ is an object in $\Fun_\simeq(\boldsymbol{\Lambda}, \cV)$ where all the morphisms in $\boldsymbol{\Lambda}$ are sent to equivalences. Similarly, a para-cyclic object with value in $\cV$ is an object in $\Fun_\simeq(\boldsymbol{\Delta}_\circlearrowleft, \cV)$.
\end{definition}

\begin{proposition}[{\cite[Lemma 4.8, 4.9 \& Proposition 4.10]{Kaledin}}]\label{prop:paracyclic}
    Let $\widetilde{\boldsymbol{\Lambda}}$ be the groupoid completion of $\boldsymbol{\Lambda}$ and $\widetilde{\boldsymbol{\Delta}}_\circlearrowleft$ be the groupoid completion of $\boldsymbol{\Delta}_\circlearrowleft$. Then there are $S^1$-equivariant fiber sequences
    $$S^1 \to \boldsymbol{\Delta}_\circlearrowleft \to \boldsymbol{\Lambda}, \;\; S^1 \to \widetilde{\boldsymbol{\Delta}}_\circlearrowleft \to \widetilde{\boldsymbol{\Lambda}}$$
    such that $(\boldsymbol{\Delta}_\circlearrowleft)_{S^1} = \boldsymbol{\Lambda}$ and $(\widetilde{\boldsymbol{\Delta}}_\circlearrowleft)_{S^1} = \widetilde{\boldsymbol{\Lambda}}$. In addition, $\widetilde{\boldsymbol{\Delta}}_\circlearrowleft = \{*\}$ and $\widetilde{\boldsymbol{\Lambda}} = BS^1$. Hence
    $$\Fun_\simeq(\boldsymbol{\Delta}_\circlearrowleft, \cV)^{S^1} \simeq \Fun_\simeq(\boldsymbol{\Lambda}, \cV) \simeq \Fun(\widetilde{\boldsymbol{\Lambda}}, \cV) \simeq \Fun(BS^1, \cV),$$
    an object in $\cV$ with an $S^1$-action is equivalent to a cyclic object and is also equivalent to an $S^1$-invariant para-cyclic object with value in $\cV$.
\end{proposition}

    We recall the topological models of the para-cyclic categories considered in Ayala--Mazel-Gee--Rozenblyum \cite{AyalaMazel-GeeRozenblyum} and Ayala--Francis \cite{AyalaFrancisTrace}.

\begin{remark}\label{rem:cyclic-combin-topol}
    Let $\cD(S^1)$ be the category of framed 1-disk refinements of $S^1$, whose objects are stratifications of $S^1$ by disks and morphisms are maps between stratified spaces \cite[Definition 1.18]{AyalaMazel-GeeRozenblyum}. By \cite[Observation 1.20]{AyalaMazel-GeeRozenblyum} or \cite[Observation 1.3 \& Corollary B.9]{AyalaFrancisTrace}, there exists an $S^1$-equivariant equivalence between $\infty$-categories
    $$\cD(S^1) \simeq \boldsymbol{\Delta}_\circlearrowleft, \; (D \to S^1) \mapsto \pi_0(\exp^{-1}(S^1 \setminus sk_0(D))).$$
\end{remark}
\begin{remark}\label{rem:disk-stratum-embed}
    Consider the categories $\cD(S^1)^{\triangleleft}$ and $\cD(S^1)^{\triangleright}$ defined by adjoining an initial or a final object to $\cD(S^1)$. Let $\operatorname{Disk}_{1/S^1}^\text{fr}$ be the category of framed embedded disks into $S^1$, whose objects are disjoint unions of framed embedded disks in $S^1$ and morphisms are framed embeddings \cite[Section 2.2]{AyalaFrancis}. By \cite[Corollary B.9]{AyalaFrancisTrace}, there is an $S^1$-equivariant equivalence
    $$\cD(S^1)^\triangleleft \simeq \operatorname{Disk}_{1/S^1}^\text{fr}, \;\; \{*\} \mapsto (\varnothing \hookrightarrow S^1),\;\; (D \to S^1) \mapsto (D \setminus sk_0(D) \hookrightarrow S^1).$$
\end{remark}

\begin{theorem}[Ayala--Francis {\cite[Theorem 2.6]{AyalaFrancisTrace}}]\label{thm:AF}
    Let $\SC$ be a symmetric monoidal category that admits geometric realizations. Let $X$ be a dualizable object in $\SC$. Then there is a canonical functorial $S^1$-action on $\HH_*(X) = \Tr(\Id_X)$ defined by a cyclic object $\Tr(\Id_X, -) : \boldsymbol{\Lambda} \to \End(1_\SC)$ with
    \begin{align*}
        \Tr(\Id_X, n): 1_\SC \xrightarrow{\eta_X^{\otimes n}} (X \otimes X^\vee)^{\otimes n} \xrightarrow{\tau_{2n}} (X^\vee \otimes X)^{\otimes n} \xrightarrow{\epsilon_X^{\otimes n}} 1_\SC,
    \end{align*}
    where $\tau_{2n}: (X \otimes X^\vee)^{\otimes n} \xrightarrow{\sim} (X^\vee \otimes X)^{\otimes n}$ is the cyclic permutation of order $2n$, such that the face maps $d_i$ and degeneration maps $s_i$ are induced by the triangle identities of the unit and counit, and $t_n$ is induced by the cyclic rotation on $(X \otimes X^\vee)^{\otimes n}$ of order $n$.
\end{theorem}
\begin{remark}\label{rem:cyclic-object-map}
    The face morphism $d_i: \Tr(\Id_X, n) \to \Tr(\Id_X, n+1)$ is induced by the identity 
    $$(\epsilon_X \otimes \epsilon_X) \circ (\Id_{X^\vee} \otimes \,\eta_X \circ \Id_{X}) \xrightarrow{\sim} \epsilon_X.$$
    The degeneration morphism $s_i: \Tr(\Id_X, n+1) \to \Tr(\Id_X, n)$ is induced by the identity 
    $$\eta_X \xrightarrow{\sim} (\Id_X \otimes \,\epsilon_X \circ \Id_{X^\vee}) \circ (\eta_X \circ \,\eta_X).$$
    Finally, the cyclic rotation is induced by $\tau_{2n}^2: (X \otimes X^\vee)^{\otimes n} \to (X \otimes X^\vee)^{\otimes n}$ composed with the units and precomposed with the counits. See for instance \cite[Section 3.3--3.5]{Ben-Zvi-Nadler-Trace}.
\end{remark}

\begin{proof}[Proof of Theorem \ref{thm:AF}]
    Consider the categories $\cD(S^1)$ and $\operatorname{Disk}_{1/S^1}^\text{fr}$. 
    Following Remark \ref{rem:disk-stratum-embed}, the factorization homology of $X$ is defined by \cite[Observation 2.3]{AyalaFrancisTrace}
    $$\HH_*(X) \coloneqq \colim \left( \cD(S^1)^{\triangleleft} \xrightarrow{\sim} \operatorname{Disk}_{1/S^1}^\text{fr} \xrightarrow{\inEnd(X)} \SC \right).$$
    Ayala--Francis \cite[Theorem 2.6]{AyalaFrancisTrace} showed that there exists an $S^1$-invariant trace morphism from the factorization homology to $1_\SC$, or equivalently, that the colimit of the factorization homology over $\cD(S^1)^\triangleleft$ factors through $\cD(S^1)^{\triangleleft \triangleright}$:
    $$\HH_*(X) \coloneqq \colim \left( \cD(S^1)^{\triangleleft} \xrightarrow{\inEnd(X)} \SC \right) = \colim \left( \cD(S^1)^{\triangleleft \triangleright} \xrightarrow{\inEnd(X)} \SC \right).$$
    In particular, there exists a functor $\Tr(\Id_X, -): \cD(S^1)^{\triangleleft \triangleright} \to \End(1_\SC)$ by evaluating objects in the suspension category of framed 1-disk refinements. More precisely, since $X$ is dualizable which implies that $\inEnd(X) = X^\vee \otimes X$, it follows from \cite[Equation (10) \& (11)]{AyalaFrancisTrace} that the evaluation at the disk stratification $D \to S^1$ is given by
    \begin{align*}
        \Tr(\Id_X, D): 1_\SC \xrightarrow{\eta_X^{\otimes sk_0(D)}} (X^\vee \otimes X)^{\otimes \pi_0(D)} \xrightarrow{\tau_{2n}} (X \otimes X^\vee)^{\otimes \pi_0(S^1 \setminus sk_0(D)} \xrightarrow{\epsilon_X^{\otimes \pi_0(S^1 \setminus sk_0(D))}} 1_\SC,
    \end{align*}
    where $\tau_{2n}: sk_0(D) \xrightarrow{\sim} \pi_0(S^1 \setminus sk_0(D))$ is induced by Poincar\'e duality exchanging the 0-dimensional and 1-dimensional strata in $S^1$. Equivalently, we have a desuspension functor $\Tr(\Id_X, -): \cD(S^1) \to \End(1_\SC)$. For a pair of disk refinements $D \to D' \to S^1$, we have natural morphisms $\Tr(\Id_X, D) \to \Tr(\Id_X, D')$ that are induced by the triangle identities of the unit and counits.
    
    Let $\boldsymbol{\Delta}_\circlearrowleft$ be the paracyclic category. By Remark \ref{rem:cyclic-combin-topol}, there exists an $S^1$-equivariant equivalence $\cD(S^1) \simeq \boldsymbol{\Delta}_\circlearrowleft$.
    Thus we get an $S^1$-invariant paracyclic object $\Tr(\Id_X, -) : \boldsymbol{\Delta}_\circlearrowleft \to \End(1_\SC)$ such that
    \begin{align*}
        \Tr(\Id_X, n): \cV \xrightarrow{\eta_X^{\otimes n}} (X^\vee \otimes X)^{\otimes n} \xrightarrow{\tau_{2n}} (X \otimes X^\vee)^{\otimes n} \xrightarrow{\epsilon_X^{\otimes n}} \cV,
    \end{align*}
    Since $(\boldsymbol{\Delta}_\circlearrowleft)_{S^1} = \boldsymbol{\Lambda}$ by Proposition \ref{prop:paracyclic}, we obtain a cyclic object $\Tr(\Id_X, -) : \boldsymbol{\Lambda} \to \End(1_\SC)$. Thus, we explained how the theorem follows from \cite[Theorem 2.6]{AyalaFrancisTrace}.
\end{proof}

\begin{definition}[{\cite[Section 2.3.3]{BravRozenblyum}}]
For a dualizable object $X$, we define the $S^1$-coinvariant and $S^1$-invariant of the Hochschild homology 
$${\HH_*(X)}_{S^1} \, \text{and} \ {\HH_*(X)}^{S^1}$$
to be the cyclic homology and negative cyclic homology of $X$.
For dualizable objects $X, Y$ and an adjunction $f: X \rightleftharpoons Y: f^r$, since the $S^1$-action is functorial, we similarly consider the $S^1$-coinvariant and $S^1$-invariant of the relative Hochschild homology
$$\HH_*(Y,X)_{S^1} \ \text{and} \ \HH_*(Y,X)^{S^1}.$$
They are defined as the relative cyclic homology and
the relative negative cyclic homology.
\end{definition}

\begin{remark}
There are several definitions of coherent $S^1$-actions on Hochschild homology in the literature. Using the cobordism hypothesis \cite{LurieCobordism}, all of them should be identical. However, some of the identifications can also be shown directly. For example, the construction of Ayala--Francis using traces we used \cite{AyalaFrancisTrace} is compatible with Ayala--Mazel-Gee--Rozenblyum using bar comstructions \cite{AyalaMazel-GeeRozenblyum}, as shown in \cite[Proposition B.12]{AyalaFrancisTrace}. The construction of Ayala--Mazel-Gee--Rozenblyum \cite{AyalaMazel-GeeRozenblyum} and Nikolaus--Scholze \cite{NikolausScholze} can probably be identified by writing down both constructions explicitly (see \cite[Section 0.3]{AyalaMazel-GeeRozenblyum}). The comparison between the $S^1$-action by Ayala--Mazel-Gee--Rozenblyum \cite{AyalaMazel-GeeRozenblyum} and the classical mixed complex construction by Connes and Tsygan is shown by Hoyois \cite{Hoyois}. 

A comparion between the trace and the classical bar construction of Hochschild homology, on the other hand, has also appeared in \cite[Proposition 4.24]{Hoyois-Scherotzke-Sibilla}, but they did not compare the $S^1$-actions. In fact, the $S^1$-action defined in \cite[Proposition 2.16]{Hoyois-Scherotzke-Sibilla} invokes the cobordism hypothesis.
\end{remark}

\subsection{Smooth and proper Calabi--Yau structures}
For this section, we assume that the $2$-category $\SC$ is enriched in stable categories.
That is, for any $X, Y \in \SC$, the Hom-category $\Hom(X,Y)$ is stable.
In \cite[Section 3]{Hoyois-Scherotzke-Sibilla}, such a $2$-category is called linear.
(We will only need the case when $\SC = \PrLst(\cV)$ for a rigid symmetric monoidal category $\cV$, which clearly satisfies the condition.) 

\begin{definition}[{\cite[Definition 3.2 \& 3.5]{Brav-Dyckerhoff1}}] 
    Let $X$ be dualizable object in $\SC$ that is enriched in stable categories. Then
    \begin{enumerate}
      \item When $X$ is smooth, a weak $d$-dimensional smooth Calabi--Yau structure on $X$ is a morphism 
      $$\phi: \Id_{1_\SC}[d] \rightarrow \HH_*(X)$$
      such that the induced map $\Id_X^![d] \rightarrow \Id_X$ is an isomorphism;
      \item When $X$ is proper, a weak $d$-dimensional proper Calabi--Yau structure on $X$ is a morphism 
      $$\phi^*: \HH_*(X) \rightarrow \Id_{1_\SC}[-d]$$
      such that the induced map $\Id_X[d] \rightarrow \Id_X^*$ is an isomorphism;
      \item A strong $d$-dimensional smooth (resp.~proper) Calabi--Yau structure on $X$ is a weak smooth (resp.~proper)
      Calabi-Yau structure on $X$ and a lifting 
      $$\widetilde{\phi}: \Id_{1_\SC}[d] \rightarrow \HH_*(X)^{S^1} 
      (\text{resp.} \ \widetilde{\phi}^*: \HH_*(X)_{S^1} \rightarrow \Id_{1_\SC}[-d])$$
      to the negative cyclic homology (resp.~cyclic homology) of $X$.
    \end{enumerate}
\end{definition}

Smooth (resp.~proper) Calabi--Yau structures induce isomorphisms between Hochschild homology (resp.~the dual of Hochschild homology) and Hochschild cohomology. We recall the construction of Hochschild cohomology \cite[Section 2.1]{BravRozenblyum}.

\begin{definition}
For a not-necessarily dualizable object $X$, we denote by
$$\HH^*(X) \coloneqq \Hom_{\End(X)}(\Id_X, \Id_X)$$
and recall it the Hochschild cohomology of $X$. 
\end{definition}

\begin{remark}\label{cohomology-v-smooth}
When $X$ is dualizable, we have 
$$\HH^*(X) = \Hom_{\Hom\left(1_\SC, \inHom(X,X) \right)}(\eta_X, \eta_X) = \Hom_{\Hom\left(\inHom(X,X), 1_\SC \right)}(\epsilon_X, \epsilon_X).$$
When assuming further that $X$ is smooth, we have, for example,
$$\HH^*(X) = \Hom_{\End(1_\SC)}(\Id_{1_\SC}, \epsilon_X \epsilon_X^l).$$
\end{remark}

\begin{remark}\label{rem:orientation-is-id}
In practice, it is often easier to first find an equivalence $\Phi_X: \Id_X^![d] \xrightarrow{\sim} \Id_X$. Per Remark \ref{cohomology-v-smooth}, when $X$ is smooth and there is an equivalence $\Phi_X: \Id_X^![d] \xrightarrow{\sim} \Id_X$, the Hochschild homology class $\phi_X \in \HH_*(X)$ corresponding to $\Phi_X$ is tautologically the image of $1 \in \HH^*(X)$ under the induced isomorphism $$\Hom(\epsilon_X, \epsilon_X) \xrightarrow{\sim} \Hom(\epsilon_X^l, \epsilon_X^l) \xrightarrow{\Phi} \Hom(\epsilon_X^l, \eta_X)[-d].$$
To shows that the corresponding class $\phi_X \in \HH_*(X)$ gives a strong Calabi-Yau structure, it is thus equivalent to show that the map
$$1_\SC \rightarrow \epsilon_X \epsilon_X^l \xrightarrow{\Phi} \epsilon_X \eta_X[-d]$$
is $S^1$-equivariant. See also \cite[Section 4.1.1]{BravRozenblyum}. 
\end{remark}

We then consider the relative version of Calabi-Yau structures introduced in \cite{Brav-Dyckerhoff1}.
Let $f: X \rightleftharpoons Y : f^r$ be an adjunction and assume $X$ that is smooth. 
By Lemma \ref{rem:bimod-adjunction}, we know that 
$$(ff^r)^! \simeq f \Id_X^! f^r.$$
Assume further that $Y$ is smooth. Because taking $(-)^!$ is functorial on left dualizable morphisms, 
from the counit $c: ff^r \rightarrow \Id_Y$ we can obtain $c^!: \Id_Y^! \rightarrow (ff^r)^!$.
Recall also that, by Lemma \ref{swapping}, the map $F_!: \Hom(X,X) \rightarrow \Hom(Y,Y)$ is given 
by $\alpha \mapsto f \circ \alpha \circ f^r$. Hence, for a morphism $\phi: \Id_X^! \rightarrow \Id_X$, 
we have a morphism 
$$(ff^r)^! = f \Id_X^! f^r \xrightarrow{F_! \phi} f \Id_X f^r = ff^r.$$
Consider a morphism $\Iduc[d] \to \HH_*(X, Y)$. This is equivalent to a commutative diagram
\[\xymatrix{
\Iduc[d-1] \ar[r] \ar[d] & \HH_*(X) \ar[d] \\
0 \ar[r] & \HH_*(Y),
}\]
which means that there is a map $\phi: \Id_X^! \to \Id_X[1-d]$ such that the following induced map 
$$\Id_Y^! \xrightarrow{c^!} (ff^r)^! \xrightarrow{F_!\phi} ff^r[1-d] \xrightarrow{c} \Id_Y[1-d]$$
is null homotopic. A similar discussion can be made for the proper case.

\begin{definition}[{\cite[Definition 4.7 \& 4.11]{Brav-Dyckerhoff1}}]
    Let $f: X \rightleftharpoons Y: f^r$ be an adjunction.
    \begin{enumerate}
      \item If $X$ and $Y$ are smooth, a weak $d$-dimensional smooth Calabi--Yau structure on $f$ is a morphism
      $$\phi: \Iduc [d] \rightarrow \HH_*(X, Y)$$
      such that the induced vertical maps
          \[\xymatrix{
          \Id_Y^! \ar[r]^-{c^!} \ar[d] & (ff^r)^! \ar[r] \ar[d] & \mathrm{Cofib}(c^!) \ar[d] \\
          \mathrm{Fib}(c)[1-d] \ar[r] & ff^r[1-d] \ar[r]^-{c[1-d]} & \Id_Y[1-d]
          }\]
      are equivalences in $\End(Y)$;
      \item If $X$ and $Y$ are proper, a weak $d$-dimensional proper Calabi--Yau structure on $f$ is a morphism
      $$\phi^*: \HH_*(Y, X) \rightarrow \Iduc[-d]$$
      such that the induced vertical maps
          \[\xymatrix{
          \Id_X[d-1] \ar[r]^{u\hspace{5pt}} \ar[d] & f^rf[d-1] \ar[r] \ar[d] & \mathrm{Cofib}(u) \ar[d] \\
          \mathrm{Fib}(u^*) \ar[r] & (f^rf)^* \ar[r]^{\hspace{10pt}u^*} & \Id_X^*
          }\]
          are equivalences in $\End(X)$;      
      \item a strong $d$-dimensional smooth (resp.~proper) Calabi--Yau structure on $f$ 
      is a weak $d$-dimensional Calabi--Yau structure 
      together with a lifting to the negative cyclic homology (resp.~the cyclic homology) 
      $$\widetilde{\phi}: \Iduc[d] \rightarrow \HH_*(Y, X)^{S^1} (\text{resp.} \ 
      \widetilde{\phi}^*: \HH_*(Y, X)_{S^1} \rightarrow \Iduc[-d]).$$
    \end{enumerate}
\end{definition}

To have a parallel discussion of Remark \ref{rem:orientation-is-id}, we consider the relative Hochschild cohomology. See \cite[Section 2.2]{BravRozenblyum} for a detailed discussion.

\begin{definition}[{\cite[Section 2.2.2]{BravRozenblyum}}]
Let $f: X \rightarrow Y$ be an object of the arrow 2-category $\Arr(\SC) \coloneqq \Fun([1], \SC)$. Then, we define the relative Hochschild cohomology to be 
$$\HH^*(X,Y) \coloneqq \Hom_{\End_{\Arr(\SC)}(f)}(\Id_f, \Id_f).$$
Similar to the situation of relative homology, we surpass the dependence on $f$. 
\end{definition}

Unwrapping the diagram \cite[Equation (2.3)]{BravRozenblyum}, we see that $\End_{\Arr(\SC)}(f) = \Hom(X,X) \times_{\Hom(X,Y)} \Hom(Y,Y)$ where the maps are given by
$$\Hom(X,X) \xrightarrow{f \circ (-)} \Hom(X,Y) \xleftarrow{(-)\circ f} \Hom(Y,Y).$$
Recall that objects in a limit is given by compatible families of objects in the diagram and there is a similar description for Homs. In particular, under this equivalence, $F$ corresponds to the triple $(\Id_X, \Id_Y, f \circ \Id_X = f = \id_Y \circ f)$, and this implies that
$$\HH^*(X,Y) = \HH^*(X) \times_{\End(f)} \HH^*(Y).$$
Here the maps are given by
$$\Hom(\Id_X, \Id_X) \xrightarrow{f \bigcirc (-)} \Hom(f, f) \xleftarrow{(-)\bigcirc f} \Hom(Y,Y)$$
where we use `$\bigcirc$' to denote horizontal composition of $1$-morphisms on $2$-morphisms.

Assume now that $f: X \rightleftharpoons Y: f^r$ is an adjunction between two dualizable objects. Under the adjunction $\Hom(f, f) = \Hom(ff^r, \Id_Y)$, the map $HH^*(Y) \rightarrow \End(f)$ is given by composing with the counit $\epsilon_f: ff^r \rightarrow \Id_Y$,
$$\Hom(\Id_Y, \Id_Y) \xrightarrow{(-) \circ \,\epsilon_f} \Hom(ff^r, \Id_Y).$$
A similar discussion shows that the map $\HH^*(X) \rightarrow \End(f)$ can be identified as
$$\Hom(\Id_X, \Id_X) \xrightarrow{\eta_f \circ \,(-)} \Hom(\id_X, f^r f).$$
Combing the two, we obtain the following lemma.

\begin{lemma}[{\cite[Equation (4.2)]{BravRozenblyum}}]
For $f: X \rightleftharpoons Y: f^r$ an adjunction, there is a fiber sequence
$$\Hom\left(\mathrm{Cofib}(\epsilon_f), \Id_Y \right) \rightarrow \HH^*(X,Y) \rightarrow \HH^*(X).$$
\end{lemma}

\begin{proof}
Notice that $\HH^*(X,Y)$ is the fiber of $\HH^*(Y) \oplus \HH^*(X) \xrightarrow{(+,-)} \End(f)$ where we modify the maps above by the corresponding signs. Then the statement immediately follows.
\end{proof}

\begin{remark}
In the setting of our paper, as well as many other geometric situation \cite[Section 4]{BravRozenblyum}, the weak relative Calabi-Yau structure, $\phi: \Iduc \rightarrow \HH_*(Y,X)$ gives a weak relative Calabi-Yau structure on $X$ when composing with $\HH_*(Y,X) \rightarrow \HH_*(X)$. In this case, per Remark \ref{rem:orientation-is-id}, there is the commuting diagram between fiber sequences
$$
\xymatrix{
    \HH^*(X,Y) \ar[r]^{ }  \ar[d]^{(1 \mapsto \Phi) } & \HH^*(X) \ar[r]^{ }  \ar[d]^{(1 \mapsto \Phi) }   & \Hom(\fib(\epsilon_f), \Id_Y) \ar[d]^{ } \\
    \HH_*(X,Y)[-1 - n] \ar[r]^{ } & \HH_*(X)[-n] \ar[r]^{} & \HH_*(Y)[-n].
    }
$$ 
The left vertical arrow is an isomorphism because $\phi$ is a weak relative Calabi-Yau structure. The extra assumption of inducing a Calabi-Yau structure on $X$ implies the middle vertical arrow is an isomorphism as well. Thus, the third arrow is also an isomorphism. Both Calabi--Yau classes are the image of the units in the Hochschild cohomology.
\end{remark}

\begin{remark}
    We can write down the natural morphism
    $$\Hom(\Id_{X}^!, \Id_{X}) \longrightarrow \Hom((ff^r)^!, ff^r)$$
    induced by pre-composition and post-composition as follows. Consider the adjunction in Remark \ref{rem:bimod-adjunction}. We know that $\epsilon_X^l \to \eta_X$ induces a natural morphism
    $$(\epsilon_Y \circ (\Id_{Y^\vee} \otimes ff^r))^l \xrightarrow{\sim} (f^\vee \otimes f^r)^l \circ \epsilon_X^l \to (f^{r\vee} \otimes f) \circ \eta_X \xrightarrow{\sim} (\Id_{Y^\vee} \otimes ff^r) \circ \eta_Y.$$
\end{remark}

\subsection{Dualizable and compactly generated categories}

In this section, we restrict ourselves to dualizable presentable stable categories and in particular compactly generated presentable categories. Most of the materials can be found in \cite[Section 4]{Hoyois-Scherotzke-Sibilla}.

Fix a idempotent complete rigid symmetric monoidal category $\cV_0$ and consider the compactly generated rigid symmetric monoidal category $\cV = \Ind(\cV_0)$. The symmetric monoidal category we will be working on is the category $\SC = \PrLst(\cV)$ consisting of $\cV$-linear presentable stable $\infty$-categories with morphisms being left adjoints. When confusion is unlikely, we simply denote it as $\PrLst$.

\begin{remark}
A functor $f: \sB \rightarrow \sA$ in $\PrLst$ is left adjoinable if it admits a further left adjoint and right adjoinable if its right adjoint admits a further right adjoint. 
By the adjoint functor theorem \cite[Corollary 5.5.2.9]{Lurie-HTT}, up to some set-theoretic issues, being a left (resp.~right) adjoint is equivalent to colimit preserving (resp.~limit preserving). Thus, $f: \sB \rightarrow \sA$ is left dualizable if it is also limit preserving and is right dualizable if if it has a colimit preserving right adjoint. 
\end{remark}

\begin{example}\label{ex:small-functor}
Functors between compactly generated categories have a more interesting description. Let $f: B \rightarrow A$ be an exact functor. 
The functor induced by ind-completion $f: \sB \rightarrow \sA$ 
is tautologically colimit preserving, and one can check that its right adjoint is given by pre-composition 
$$f^*: \sB = \Fun^{ex}(B, \cV) \to \Fun^{ex}(A, \cV) = \sA, \;\; M \mapsto M \circ f$$
which is also colimit preserving. Here the superscript `ex' means exact functors.
\end{example}

Dualizable objects in $\PrLst$ are completely classified by compactly assembled categories \cite[Proposition D.7.3.1]{Lurie-SAG}. Here, a category is compactly assembled \cite[Definition 21.1.2.1]{Lurie-SAG} if there exists a compactly generated category $\sB = \Ind(B)$ so that $\sA$ is a retract of $\sB$. 
In this paper, the categories we will be working with are all compactly generated.

\begin{example}\label{ex:yoneda}
Consider a compactly generated category $\sA \coloneq \Ind(A) = \Fun^{ex}(A^{op}, \cV)$. Then the dual of $\sA$ is given by $\sA^\vee \coloneqq \Ind(A^{op})$, where the counit $\epsilon_\sA \in \Fun^{L}(\sA^\vee \otimes \sA, \cV)$ is given by $(X, Y) \mapsto \Hom(X, Y)$ under the equivalence \cite[Proposition 5.3.6.2]{Lurie-HA}
$$\Fun^L(\sA^\vee \otimes \sA, \cV) = \Fun^{ex}(A^{op} \otimes A, \cV),$$
and the unit is given by the diagonal bimodule under the equivalence \cite[Proposition 5.3.6.2]{Lurie-HA}
$$\Fun^L(\cV, \sA \otimes \sA^\vee) = \Fun^{ex}(\cV_0, \Fun^{ex}(A \otimes A^{op}, \cV)).$$
Here the supersecript `ex' means exact functors and the `L' means colimit-preserving functors.
\end{example}

The notion of dualizable morphisms in $\PrLst$ is straightforward.
For any dualizable category $\sA$, we have the following equivalences
$$\Fun^L(\cV, \sA \otimes \sB^\vee) = \Fun^L(\sB, \sA) = \Fun^L(\sA^\vee \otimes \sB, \cV).$$
See Remark \ref{rem:dual-functor} and Definition \ref{def:dualizing-functor}. In particular, under the equivalence 
$$\Fun^L(\cV, \sA \otimes \sA^\vee) = \Fun^L(\sA, \sA) = \Fun^L(\sA^\vee \otimes \sA, \cV),$$
the identity functor $\Id_\sA$ always corresponds to $\eta_\sA$ under the first isomorphism and to $\epsilon_\sA$ under the second isomorphism. For compactly generated categories, we can see that colimit preserving functors are equivalent to bimodules.

\begin{example}
An exact functor $f: B \to A$ induces a pull-back functor $F^*: \Fun^{ex}(A \otimes A^{op}, \cV) \to \Fun^{ex}(B \otimes B^{op}, \cV)$ and a push-forward functor $F_!: \Fun^{ex}(B \otimes B^{op}, \cV) \to \Fun^{ex}(A \otimes A^{op}, \cV)$ by \cite[Section 2.1]{Brav-Dyckerhoff1}
$$F^*M(X, Y) = M(fX, fY),\;\; F_!M(X, Y) = M(f^*\cY_X, f^*\cY_Y).$$
where $\cY: A \hookrightarrow \sA$ is the Yoneda embedding. By Example \ref{ex:small-functor} and \ref{ex:yoneda}, the push-forward functor $F_!$ of bimodules corresponds to the push-forward functor of colimit preserving functors $F_!: \Fun^L(\sB, \sB) \to \Fun^L(\sA, \sA)$ in Definition \ref{def:push-forward}.
\end{example}

The notion of proper and smooth for compactly generated categories \cite[Section 2.3]{Brav-Dyckerhoff2} recovers the usual notion for small categories. See \cite{Kontsevich-Soibelman-Ainfty,Ganatra1}.

\begin{example}
    When $\cV$ is compactly generated by $1$, $\sA$ is smooth if and only if $\eta_\sA(1) \in \sA^\vee \otimes \sA$ is a compact object. When $\sA$ is also compactly generated by $A$, by Example \ref{ex:yoneda}, it is smooth if and only if the diagonal bimodule $A_\Delta: (X, Y) \to \Hom(X, Y)$ is compact. When $\sA$ is compactly generated by $A$, $\sA$ is proper if and only if $\epsilon_\sA(X, Y)$ is compact for any $X, Y \in A$, or equivalently, by Example \ref{ex:yoneda}, $\Hom(X, Y) \in \cV$ is compact for any $X, Y \in A$.
\end{example} 



Therefore, one can define (relative) Calabi--Yau structures accordingly. For compactly generated categories, left and right Calabi--Yau structures are in particular related in the following way. Recall that for a small idempotent complete category $A$, the category of proper modules is $\Prop A \coloneqq \Fun^{ex}(A^{op}, \cV_0)$.

\begin{theorem}[Brav--Dyckerhoff {\cite[Theorem 3.1]{Brav-Dyckerhoff1}}]\label{thm:left-right-cy}
    Let $\sA = \Ind(A)$ be a compactly generated smooth category with a left Calabi-Yau structure. Then $\Prop \sA \coloneqq \Ind\Prop A$ has a canonically induced right Calabi-Yau structure.
\end{theorem}

\begin{proposition}[Brav--Dyckerhoff {\cite[Proposition 2.3]{Brav-Dyckerhoff1}, \cite[Corollary 2.6]{Brav-Dyckerhoff2}}]\label{prop:leftdual-serre}
    Let $\sA$ be a compactly generated smooth category and $X \in \sA$ is left proper, i.e.~$\Hom(-, X)^\vee$ is a continuous functor with continuous adjoint, then there is a natural isomorphism
    $$\Hom(-, X)^\vee \xrightarrow{\sim} \Hom(\Id_\sA^!(-), X).$$
\end{proposition}

\begin{theorem}[{\cite[Theorem 11.5]{KellerWang}}]\label{thm:smooth-induce-proper}
    Let $\sA, \sB$ be compactly generated smooth categories and $f: \sB \rightleftharpoons \sA : f^r$ be a continuous adjunction equipped with a relative left Calabi--Yau structure. Then $f^r: \Prop \sA \rightleftharpoons \Prop \sB : f^{rr}$ carries a canonically induced right Calabi--Yau structure.
\end{theorem}

\section{Preliminary in Sheaf Theory} \label{sec:sheaves}

\subsection{Microsupport and microsheaves}
Let $\cV$ be a rigid symmetric monoidal category. Let $M$ be a smooth manifold and $\Sh(M)$ be the category of sheaves with coefficients in $\cV$. Following Kashiwara--Schapira \cite{KS}, for a sheaf $F \in \Sh(M)$, one can define a conic closed subset in the cotangent bundle $\ms(F) \subseteq T^*M$ called the singular support of $F$ and the corresponding closed subset in the cosphere bundle $\msif(F) \subseteq S^*M$ called the singular support at infinity of $F$.

For $\widehat X \subseteq T^*M$, we define $\Sh_{\widehat X}(M)$ to be the full subcategory of sheaves $F$ such that $\ms(F) \subseteq \widehat X$. For $X \subseteq S^*M$, we define $\Sh_X(M)$ to be the full subcategory of sheaves $F$ such that $\msif(F) \subseteq X$. When $X \subseteq S^*M$ is compact subset, we also define $\Sh_X(M)_0$ to be the full subcategory of compactly supported sheaves $F$ such that $\msif(F) \subseteq X$. More specifically, in this paper, we will only consider singular support conditions that arise from the following setting.

\begin{definition}
    Let $M$ be a connected smooth manifold and $\Lambda \subseteq S^*M$ be a compact subset. Define $M_c$ to be the union of $\pi(\Lambda)$ and all bounded regions in $M \setminus \pi(\Lambda)$. Define $\widehat\Lambda = M_c \cup \bR_+ \times \Lambda$ to be the conification of $\Lambda$ that only includes the compact subset in the zero section $M_c$.
\end{definition}

\begin{lemma}
    Let $M$ be a connected smooth manifold and $\Lambda \subseteq S^*M$ be a compact subset. Then $\Sh_\Lambda(M)_0 = \Sh_{\widehat\Lambda}(M)$, so the latter consists of compactly supported sheaves $F$ such that $\msif(F) \subseteq \Lambda$.
\end{lemma}

In this paper, we will consider the situation where $\Lambda \subseteq S^*M$ (resp.~$\widehat\Lambda \subseteq T^*M$) is a subanalytic isotropic subset. Sheaves microsupported in $\Lambda$, in this case, are constructible sheaves.

Let $\Lambda \subseteq S^*M$ be subanalytic isotropic subset. Then the inclusion functor
$\iota_{\widehat\Lambda*}: \Sh_{\widehat\Lambda}(M) \hookrightarrow \Sh(M)$
is limit and colimit preserving, and thus admits both left and right adjoint. In particular, the left adjoint functor is also colimit perserving. We will denote the left and right adjoint by 
$$\iota_{\widehat\Lambda}^*, \, \iota_{\widehat\Lambda}^!: \Sh(M) \twoheadrightarrow \Sh_{\widehat\Lambda}(M).$$

\begin{proposition}\label{prop:stopremoval}
Let $M$ be connected and $\Lambda$ be a subanalytic isotropic subset in $S^*M$.
The category $\Sh_{\widehat\Lambda}(M)$ is compactly generated.
Let $\Lambda \subseteq \Lambda^\prime$ be an inclusion of subanalytic isotropic subsets. Then the left adjoint $\iota_{\widehat\Lambda \widehat\Lambda^\prime}^*$ of the inclusion functor $\iota_{\widehat\Lambda \widehat\Lambda^\prime*} : \Sh_{\widehat\Lambda}(M) \hookrightarrow \Sh_{\widehat\Lambda^\prime}(M)$ sends compact objects to compact objects, i.e.,
$\Sh_{\widehat\Lambda^\prime}^c(M) \twoheadrightarrow \Sh_{\widehat\Lambda}(M)^c$.
\end{proposition}

Then we recall the definition of microsheaves following \cite{Gui,Nadler-pants,Nadler-Shende,Kuo-Li-spherical}. First, define the presheaf
\begin{align*}
\msh^\pre: (\Op_{T^* M}^{\Rp})^{op} &\longrightarrow \st \\
\Omega &\longmapsto \Sh(M)/ \Sh_{\Omega^c}(M)
\end{align*}
where we restrict our attention to conic open sets $\Op_{T^* M}^{\Rp}$,
and the target $\st$ is the (large) category of stable categories with morphisms being exact functors.
We denote by $\msh$ its sheafification and refer it as the sheaf of microsheaves.
Note that $\msh |_{0_M} = \Sh$ as sheaves of categories on $M$. One can show that in fact the Hom in microsheaves valued in $\st$ can be computed by $\mhom(-, -)$ introduced by Kashiwara--Schapira \cite[Section 4.4]{KS} using \cite[Theorem 6.1.2]{KS}.

We note that since $\msh$ is conic, $\msh |_{\dT^* M}$ descents naturally to a sheaf on $S^* M$, 
and we abuse the notation, denoting it by $\msh$ as well.
Similar statements for $\mhom$ and $\ms^\infty$ hold in this case.

\begin{definition}
Fix a subanalytic isotropic subset $\Lambda \subseteq S^*M$. 
Let $\msh_\Lambda$ denote the subsheaf of $\msh$ which consists of objects microsupported in $\Lambda$ or $\Lambda \times \RR_{>0}$.
\end{definition}

We note that this sheaf coincides with the sheafification of the following 
subpresheaf $\msh^\pre_\Lambda$ of $\msh^\pre$ (where $\ms_\Omega(F) := \ms(F) \cap \Omega$):

\begin{align*}
\msh^\pre_\Lambda: (\Op_{T^* M}^{\Rp})^{op} &\longrightarrow \st \\
\Omega &\longmapsto \{F \in \msh^\pre(\Omega) | \ms_\Omega(F) \subseteq \Lambda \}
\end{align*}

We note that $\msh^\pre_\Lambda$ in fact takes value in the category of compactly generated stable categories
whose morphisms are given by functor which admits both the left and the right adjoints,
and its sheafification in $\st$ and $\PrLcs$ coincide.
In other word, restriction maps $\msh_\Lambda(\Omega) \rightarrow \msh_\Lambda(\Omega^\prime)$ admit both left and right adjoints. 
In particular, we will refer to the restriction map associated to $\dT^* M \subseteq T^* M$ as
the microlocalization functor along $\Lambda$
$$m_\Lambda: \Sh_{\widehat\Lambda}(M) \rightarrow \msh_\Lambda(\Lambda)$$
and denote its left and right adjoint by $m_\Lambda^l$ and $m_\Lambda^r$.
Note that $\msh_\Lambda$ is a constructible sheaf on $S^*M$ or $\dT^*M$ supported on $\Lambda$ or ${\Lambda} \times \bR_{>0}$,
and we will use the same notation $\msh_\Lambda$ to denote the corresponding sheaf on $\Lambda$ or ${\Lambda} \times \bR_{>0}$.


We sometimes restrict ourselves to small categories. Consider a compact isotropic subset $\Lambda \subseteq S^*M$ (or a conic isotropic subset $\widehat\Lambda \subseteq T^*M$). Let $\msh^c_\Lambda(\Omega)$ be the full subcategory of compact objects, $\msh^b_\Lambda(\Omega) \subseteq \msh_\Lambda(\Omega)$ be the full subcategory of objects with perfect stalks, and $\msh^{pp}_\Lambda(\Omega) = \mathrm{Fun}^\text{ex}(\msh^c_\Lambda(\Omega)^{op}, \cV_0)$ be the category of proper objects.
    
The following theorem shows that $\msh^b_\Lambda(\Omega)$ is the equivalent to the subcategories of proper modules $\msh^{pp}_\Lambda(\Omega)$ in $\msh^c_\Lambda(\Omega)$.

\begin{theorem}[Nadler \cite{Nadler-pants}*{Theorem 3.21}, \cite{Ganatra-Pardon-Shende3}*{Corollary 4.23}]\label{thm:perfcompact}
    Let $\Lambda \subseteq S^{*}M$ (resp.~$\widehat\Lambda \subseteq T^*M$) be any compact subanalytic isotropic subset. Then the natural pairing $\Hom_{\msh_\Lambda(\Omega)}(-, -)$ defines an equivalence
    $$\msh^b_\Lambda(\Omega) \simeq \msh^{pp}_\Lambda(\Omega) = \mathrm{Fun}^\text{ex}(\msh^c_\Lambda(\Omega)^{op}, \cV_0).$$
    In particular, $\Sh^b_{\widehat\Lambda}(M) \simeq \Sh^{pp}_{\widehat\Lambda}(M) = \mathrm{Fun}^\text{ex}(\Sh^c_{\widehat\Lambda}(M)^{op}, \cV_0)$.
\end{theorem}

    Using the above theorem, for a subanalytic Legendrian $\Lambda \subseteq S^{*}M$ the category of proper modules
    $$\Sh^{pp}_{\widehat\Lambda}(M) = \msh^{pp}_{\widehat{\Lambda}}(T^*M),$$
    is a proper category, and moreover, by \cite[Lemma A.8]{Ganatra-Pardon-Shende3}, $\Sh_{\widehat\Lambda}^b(M) \subseteq \Sh_{\widehat\Lambda}^c(M)$.

\subsection{Contact isotopies and deformations}
We recall sheaf quantizations of contact isotopies of sheaves and deformations of sheaves.

First, the result of Guillermou--Kashiwara--Schapira \cite{GKS} constructs sheaf quantizations for any contact isotopy (homogeneous Hamiltonian isotopy).

\begin{theorem}[Guillermou--Kashiwara--Schapira \cite{GKS}]\label{thm: GKS}
    Let $\Phi: S^*M \times I \to S^*M$ be a contact isotopy on $S^*M$ defined by the contact Hamiltonian $H_t = \alpha(\partial_t\varphi_t)$ with the contact movie
    $$\Lambda_\Phi \coloneqq \left\{(x, -\xi, \varphi_t(x, \xi), t, -H_t(\varphi_t(x, \xi)) )\right\}.$$
    Then there is a sheaf $K(\Phi) \in \Sh(M \times M \times I)$ that defines a convolution functor $\Phi(-) \coloneqq K(\Phi) \circ (-)$ such that $\msnz(\Phi(F)) = \Lambda_\Phi \circ \msnz(F)$. The convolution induces equivalences 
    $$\varphi_t \coloneqq K(\Phi)_t \circ (-): \Sh(M) \xrightarrow{\sim} \Sh(M),$$
    such that $\msnz(\varphi_t(F)) = \varphi_t(\msnz(F))$. Moreover, the convolution also induces a morphism between sheaf of categories
    \begin{align*}
    \varphi_t \coloneqq K(\Phi)_t \circ (-): \msh_\Lambda \rightarrow \varphi_t^* \msh_{\varphi_t(\Lambda)}.
    \end{align*} 
\end{theorem}

Consider the case when $\Phi$ is a positive isotopy, i.e., the defining Hamiltonian $H_t = \alpha(\partial_t \varphi_t) \geq 0$.
In this case, \cite[Section 3]{Kuo-wrapped-sheaves} implies that there are continuation maps $K(\Phi)_s \rightarrow K(\Phi)_t, \; s \leq t$, and it induces continuation maps for any $F \in \Sh(M)$ 
$$c(F, s, t): \varphi_s(F) \rightarrow \varphi_t(F), \;\; s \leq t.$$
Moreover, the continuation map is canonically defined, i.e., it is independent of homotopies of the positive isotopies that preserves $\varphi_s$ and $\varphi_t$.

Then we recall the gapped condition for families of Legendrians and the gapped non-characteristic deformation lemma, which will be used later on. We fix any Reeb flow $T_t: S^*M \to T^*M$, defined by an autonomous positive Hamiltonian.

\begin{definition}[{\cite[Definition 2.10]{Nadler-Shende}}]
    Let $\Lambda_t, \Lambda_t'$, $t \in [0, 1]$, be families of subsets in the contact manifold $Y$ with Reeb flow $T_t: Y \to Y$. Then $\Lambda_t$ and $\Lambda_t'$ are gapped if there exists $\epsilon > 0$ such that for any $\epsilon' \in (-\epsilon, 0) \cup (0, \epsilon)$,
    $$\Lambda_t \cap T_{\epsilon'}(\Lambda_t') = \varnothing, \; t \in [0, 1].$$
    $\Lambda_t$ and $\Lambda_t'$ are positively gapped if there exists $\epsilon > 0$ such that for any $\epsilon' \in (0, \epsilon)$,
    $$\Lambda_t \cap T_{\epsilon'}(\Lambda_t') = \varnothing, \; t \in [0, 1].$$
\end{definition}

Let $\Lambda \subseteq S^*(M \times [0, 1])$ be subsets that are non-characteristic with respect to $M \times [0, 1] \to [0, 1]$. Then under the projection $\Pi: T^*(M \times [0, 1]) \to T^*M \times [0, 1]$, we can get
$$\Pi(\Lambda) \subseteq S^*M \times [0, 1].$$
In particular, we will write $\Lambda_t = \Pi(\Lambda) \cap S^*M \times \{t\}$ which defines a family of subsets in $S^*M$.

\begin{theorem}[Zhou {\cite[Proposition 4.1]{Zhou}}, Nadler--Shende {\cite[Theorem 5.1]{Nadler-Shende}}]\label{thm:full-faithful-nearby}
    Let $\Lambda, \Lambda' \subseteq S^*(M \times [0, 1])$ be subanalytic Legendrians that are non-characteristic with respect to $M \times [0, 1] \to [0, 1]$. Assume that the families of subanalytic Legendrians $\Lambda_t, \Lambda_t' \subseteq S^*M$ are gapped. Then for $F \in \Sh_{\Lambda}(M \times [0, 1])$ and $G \in \Sh_{\Lambda'}(M \times [0, 1])$,
    $$\Hom(F, G) = \Hom(F_t, G_t), \; t \in [0, 1],$$
    where $F_t$ and $G_t$ be the restrictions to $M \times \{t\}$.
\end{theorem}
\begin{remark}
    The above theorem can be viewed as a simplified version of \cite[Theorem 5.1]{Nadler-Shende}, where we have families of Legendrians on a noncompact interval, and one needs to estimate the singular support of the nearby cycles when taking the closure of the interval.
\end{remark}


We will also need the Reeb perturbation lemma which will lead to a more refined criterion for the gapped non-characteristic deformation lemma.

\begin{proposition}[{\cite[Theorem 4.3]{Kuo-Li-spherical}}]\label{prop:perturbation}
    Let $\Lambda, \Lambda' \subseteq S^*M$ be subanalytic Legendrians. Let $T_t: S^*M \to S^*M$ be a non-negative contact flow such that 
    $$\Lambda \cap T_\epsilon(\Lambda') = \varnothing$$
    for any small $\epsilon > 0$. Then for $F \in \Sh_{\Lambda}(M)$ and $G \in \Sh_{\Lambda'}(M)$,
    $$\Hom(F, G) = \Hom(F, T_\epsilon G).$$
\end{proposition}

\begin{proposition}\label{prop:gapped-fully-faithful}
    Let $\Lambda, \Lambda' \subseteq S^*(M \times [0, 1])$ be subanalytic Legendrians that are non-characteristic with respect to $M \times [0, 1] \to [0, 1]$. Assume that the families of subanalytic Legendrians $\Lambda_t, \Lambda_t' \subseteq S^*M$ are positively gapped. Then for $F \in \Sh_{\Lambda}(M \times [0, 1])$ and $G \in \Sh_{\Lambda'}(M \times [0, 1])$,
    $$\Hom(F, G) = \Hom(F_t, G_t), \; t \in [0, 1],$$
    where $F_t$ and $G_t$ be the restrictions to $M \times \{t\}$.
\end{proposition}
\begin{proof}
    Since $\Lambda_t, \Lambda_t' \subseteq S^*M$ are positively gapped, by Proposition \ref{prop:perturbation}, for $\epsilon > 0$ sufficiently small,
    $$\Hom(F, G) = \Hom(F, T_\epsilon G), \; \Hom(F_t, G_t) = \Hom(F_t, T_\epsilon G_t).$$
    Meanwhile, when $\Lambda_t, \Lambda_t' \subseteq S^*M$ are positively gapped, we know that for $\epsilon > 0$ sufficiently small, $\Lambda_t$ and $T_\epsilon(\Lambda_t') \subseteq S^*M$ are gapped. Therefore, by Theorem \ref{thm:full-faithful-nearby}, we know that
    $$\Hom(F, T_\epsilon G) = \Hom(F_t, T_\epsilon G_t).$$
    This then completes the proof.
\end{proof}

Finally, following Zhou \cite[Definition 0.8]{Zhou} and Nadler--Shende \cite[Definition 9.13]{Nadler-Shende}, we define subanalytic Legendrian deformations which induce equivalences through Legendrian movies\footnote{Similar terminologies are refered to as Legendrian isotopies in the title of \cite{Zhou}. Here we choose to avoid the word isotopy since this is typically not an isotopy of subsets inside the ambient manifold.}.

\begin{definition}
    Let $\Lambda \subseteq S^*(M \times [0, 1])$ be a subanalytic Legendrian, and $\Lambda_0, \Lambda_1 \subseteq S^*M$ be subanalytic Legendrians such that
    $$\Lambda \cap S^*(M \times I_\epsilon(j)) = \Lambda_j \times I_\epsilon(j), \; j = 0, 1,$$
    where $I_\epsilon(j) \subseteq [0, 1]$ is an open neighbourhood of $j$ with radius $\epsilon$. Then $\Lambda$ is called a Legendrian deformation between $\Lambda_0$ and $\Lambda_1$ if there is a contact isotopy such that
    $$S^*(M \times [0, 1]) \setminus \Lambda \cong S^*(M \times [0, 1]) \setminus \Lambda_j \times [0, 1], \; j = 0, 1,$$
    as open contact submanifolds with convex boundary.
\end{definition}
\begin{remark}
    Let $X_0, X_1 \subseteq S^*M$ are Weinstein hypersurfaces, and $W \subseteq S^*(M \times [0, 1])$ is a Weinstein hypersurface that is Liouville homotopic to $X_j \times [0, 1] \subseteq S^*(M \times [0, 1])$, such that
    $$W \cap S^*(M \times I_\epsilon(j)) = X_j \times T^*I_\epsilon(j).$$
    Then the Legendrian skeleton $\mathfrak{c}_W$ is a Legendrian deformation between $\mathfrak{c}_{X_0}$ and $\mathfrak{c}_{X_1}$. This is exactly the sufficiently Weinstein condition \cite[Definition 9.13]{Nadler-Shende}. In particular, when $(T^*M, X_0)$ and $(T^*M, X_1)$ are Liouville homotopic (instead of Weinstein homotopic), there always exists a Weinstein movie $(T^*(M \times [0, 1]), W)$ between them \cite[Proposition 2.42]{LazarevSylvanTanaka1}.
\end{remark}

\begin{theorem}[{\cite[Theorem 3.1]{Zhou}, \cite[Theorem 9.14]{Nadler-Shende}}]\label{thm:invariance-deformation}
    Let $\Lambda \subseteq S^*(M \times [0, 1])$ be a subanalytic Legendrian deformation between $\Lambda_0$ and $\Lambda_1 \subseteq S^*M$. Then there are natural equivalences
    \[\xymatrix{
    \msh_{\Lambda_0}(\Lambda_0) \ar[d] & \msh_\Lambda(\Lambda) \ar[l]_-{i_0^*}^-{\sim} \ar[r]^-{i_1^*}_-{\sim} \ar[d]  & \msh_{\Lambda_1}(\Lambda_1) \ar[d], \\
    \Sh_{\widehat\Lambda_0}(M) & \Sh_{\widehat\Lambda}(M \times [0, 1]) \ar[l]_-{i_0^*}^-{\sim} \ar[r]^-{i_1^*}_-{\sim} & \Sh_{\widehat\Lambda_1}(M).
    }\]
\end{theorem}

\subsection{Duality, doubling and fiber sequence} \label{sec:duality fiber sequence}

We recall the result on Sabloff--Serre duality and Sato--Sabloff fiber sequence obtained in \cite[Section 3.2]{LiEstimate} and \cite[Section 4.1]{Kuo-Li-spherical}. These are the key geometric results we need to show the relative Calabi--Yau property.

\begin{theorem}[Sato--Sabloff fiber sequence {\cite[Theorem 4.3]{Kuo-Li-spherical}}]\label{thm:sato-sabloff}
    Let $\Lambda, \Lambda' \subseteq S^*M$ be compact subanalytic Legendrians. Let $T_t: S^*M \to S^*M$ be a non-negative contact flow such that 
    $$\Lambda \cap T_\epsilon(\Lambda') = \varnothing$$
    for any small $\epsilon \neq 0$. Then for $F \in \Sh_{\Lambda}(M)$ and $G \in \Sh_{\Lambda'}(M)$ such that $\mathrm{supp}(F) \cap \mathrm{supp}(G)$ is compact in $M$, there is a fiber sequence
    $$\Hom(F, T_{-\epsilon}G) \to \Hom(F, G) \to \Gamma(\Lambda \cap \Lambda', \mhom(F, G)).$$
\end{theorem}

\begin{theorem}[Sabloff--Serre duality {\cite[Proposition 4.11]{Kuo-Li-spherical}}]\label{thm:sabloff-serre}
    Let $\Lambda, \Lambda' \subseteq S^*M$ be compact subanalytic Legendrians. Let $T_t: S^*M \to S^*M$ be a non-negative contact flow such that
    $$\Lambda \cap T_\epsilon(\Lambda') = \varnothing$$
    for any small $\epsilon \neq 0$. Then for $F \in \Sh_{\Lambda}(M)$ and $G \in \Sh_{\Lambda'}(M)$ such that $\mathrm{supp}(F) \cap \mathrm{supp}(G)$ is compact in $M$ and $F$ has perfect stalks, there is an isomorphism
    $$\Hom(F, T_{-\epsilon}G \otimes \omega_M) \simeq p_!(\SD{M}(F) \otimes G) \simeq \Hom(G, F)^\vee.$$
    In particular, when $\omega_M = 1_M[-n]$, $\Hom(F, T_{-\epsilon}G)[-n] \simeq p_!(\LD{M}(F) \otimes G)[-n] \simeq \Hom(G, F)^\vee$.
\end{theorem}
\begin{remark}\label{rem:compact-support}
    The assumption that $\mathrm{supp}(F) \cap \mathrm{supp}(G)$ is compact is essential. For example, consider $M = \bR$ and $\Lambda = \Lambda' = \{(0, 1)\} \subseteq S^*\bR$. Let $F = G = 1_{[0, +\infty)}$ whose supports are non-compact. Then it is straightforward to check that
    $$\Hom(F, T_{-\epsilon}G) = 0, \; \Hom(G, F) = 1_\cV.$$
\end{remark}

More generally, since microsheaves form a sheaf of categories, we can also consider homomorphisms of microsheaves on an open subset. This can also be computed using contact isotopies.

\begin{theorem}[relative Sato--Sabloff fiber sequence {\cite[Theorem 4.45]{Kuo-Li-spherical}}]\label{thm:sato-sabloff-rel}
    Let $\Lambda, \Lambda' \subseteq S^*M$ be a subanalytic Legendrians and $\Omega \subseteq \Lambda \cap \Lambda'$ be a pre-compact open subset. Let $\hat{T}_t: S^*M \to S^*M$ be a non-negative contact flow such that $\hat H^{-1}(0) \cap (\Lambda \cap \Lambda') = (\Lambda \cap \Lambda') \setminus \Omega$ and 
    $$\Lambda \cap \hat{T}_\epsilon(\Omega \cap \Lambda') = \varnothing$$
    for any small $\epsilon \neq 0$. Then for $F \in \Sh_{\Lambda}(M)$ and $G \in \Sh_{\Lambda'}(M)$ such that $\mathrm{supp}(F) \cap \mathrm{supp}(G)$ is compact in $M$, there is a fiber sequence
    $$\Hom(F, \hat T_{-\epsilon}G) \to \Hom(F, G) \to \Gamma(\Omega, \mhom(F, G)).$$
\end{theorem}

As an application, we use the Sato--Sabloff sequence to compute the endomorphism of the constant microsheaves supported on the conormal bundles. Note that the following computation also follows from \cite[Theorem 4.2.3]{KS}.

\begin{lemma}\label{lem:mhom-computation}
Let $Z \subseteq M$ be a closed submanifold. Then there is an isomorphism
$$\mhom(1_Z, 1_Z)|_{S^*M} = 1_{S^*_ZM}.$$
\end{lemma}

\begin{proof}
Consider the small positive pushoff by the geodesic flow on $S^*M$. Using the invariance of microsheaves under contact transformations Theorem \ref{thm: GKS}, it suffices to show the sheaf 
$$\mhom(1_{U(Z)}, 1_{U(Z)})  = 1_{S^*_{\partial U(Z), +} U(Z)},$$
where $U(Z)$ is a tubular neighbourhood of $Z$ and $S^*_{\partial U(Z), +} U(Z)$ is the outward unit conormal bundle of $\partial U(Z)$.
For any open subset $\Omega$, we consider a non-negative contact flow $\hat{T}_t: S^*M \to S^*M$ supported on $\overline{\Omega}$ with defining Hamiltonian $\hat{H}$. Applying Theorem \ref{thm:sato-sabloff-rel} 
, we can obtain the fiber sequence
$$ \Hom(1_{U(Z)}, \hat{T}_{-\epsilon} 1_{U(Z)}) \rightarrow \Hom(1_{U(Z)}, 1_{U(Z)}) \to \Gamma_c(\Omega, \mhom(1_{U(Z)}, 1_{U(Z)})).$$
Since $\pi: S^*_{\partial U(Z), +} U(Z) \to M$ is a smooth embedding whose image is $\partial U(Z)$, we know that for sufficiently small $\epsilon > 0$, $\pi: \hat{T}_{-\epsilon}(S^*_{\partial U(Z), +} U(Z)) \to M$ is also a smooth embedding whose image is $\partial \hat{U}_{-\epsilon}(Z)$, where $\hat{U}_{-\epsilon}(Z) \subseteq U(Z)$ is a smaller tubular neighbourhood and $U(Z) \setminus \hat{U}_{-\epsilon}(Z) \cong \pi(\Omega) \times [-\epsilon, 0)$. Therefore, we know that $\hat{T}_{-\epsilon} 1_{U(Z)} = 1_{\hat{U}_{-\epsilon}(Z)}$ and
\begin{align*}
\Gamma(\Omega, \mhom(1_{U(Z)}, 1_{U(Z)})) &= \Hom(1_{U(Z)}, \mathrm{Cofib}(\hat{T}_{-\epsilon} 1_{U(Z)} \to 1_{U(Z)})) \\
&= \Hom(1_{U(Z)}, 1_{U(Z) \setminus \hat{U}_{-\epsilon}(Z)}) = \Gamma(\Omega, 1_\Omega).
\end{align*}
For any $\Omega' \subseteq \Omega$, we can consider a non-negative contact flow $\hat{T}_t': S^*M \to S^*M$ supported on $\overline{\Omega}$ with defining Hamiltonian $\hat{H}' \leq \hat{H}$. Then it follows that for sufficiently small $\epsilon > 0$, $\pi: \hat{T}_{-\epsilon}'({S^*}_{\partial U(Z), +} U(Z)) \to M$ is also a smooth embedding whose image is $\partial \hat{U}'_{-\epsilon}(Z)$, where $\hat{U}_{-\epsilon}'(Z) \subseteq \hat{U}_{-\epsilon}(Z) \subseteq U(Z)$ is a smaller tubular neighbourhood and $U(Z) \setminus \hat{U}_{-\epsilon}'(Z) \cong \pi(\Omega') \times [-\epsilon, 0)$. Therefore, we have a commutative diagram 
\[\xymatrix{
\Gamma(\Omega, \mhom(1_{U(Z)}, 1_{U(Z)})) \ar[r]^-{\sim} \ar[d] & \Hom(1_{U(Z)}, 1_{U(Z) \setminus \hat{U}_{-\epsilon}(Z)}) \ar[r]^-{\sim} \ar[d] & \Gamma(\Omega, 1_{\Omega}) \ar[d] \\
\Gamma(\Omega', \mhom(1_{U(Z)}, 1_{U(Z)})) \ar[r]^-{\sim} & \Hom(1_{U(Z)}, 1_{U(Z) \setminus \hat{U}_{-\epsilon}'(Z)}) \ar[r]^-{\sim} & \Gamma(\Omega', 1_{\Omega'}).
}\]
Therefore, we can conclude that $\mhom(1_{U(Z)}, 1_{U(Z)})|_{S^*M} = 1_{S^*_{\partial U(Z), +} U(Z)}$.
\end{proof}

    Fix some Reeb flow $T_t: S^*M \to S^*M$. Consider $\msh_{\Lambda}(\Lambda)$ and $\Sh_{\widehat{\Lambda}_{\cup, \epsilon}}(M)$ where 
    $$\widehat{\Lambda}_{\cup, \epsilon} = (\Lambda \times \bR_+) \cup (T_\epsilon(\Lambda) \times \bR_+) \cup \bigcup\nolimits_{0 \leq s \leq \epsilon} \pi(T_s(\Lambda)).$$
    Consider the fully faithful embedding given by the doubling functor. We can give an explicit characterization of its essential image. See also \cite[Section 6.2]{Nadler-Shende}.

\begin{theorem}[{\cite[Theorem 4.1 \& Proposition 6.8]{Kuo-Li-spherical}}] \label{doubling}
    Let $\Lambda \subseteq S^*M$ be a compact subanalytic Legendrian. Then for $\epsilon > 0$ sufficiently small, for some Reeb flow $T_t: S^*M \to S^*M$, there is a fully faithful embedding
    $$w_\Lambda: \msh_{\Lambda}(\Lambda) \hookrightarrow \Sh_{\Lambda \cup T_\epsilon(\Lambda)}(M).$$
    Moreover, there is a recollement 
    $\msh_\Lambda(\Lambda) \hookrightarrow \Sh_{\Lambda \cup \Lambda_\epsilon}(M) \to \Sh_\Lambda (M).$ For $\epsilon > 0$ sufficiently small, the doubling functor gives an equivalence 
    $$w_\Lambda: \msh_{\Lambda}(\Lambda) \xrightarrow{\sim} \Sh_{\widehat\Lambda_{\cup,\epsilon}}(M).$$
\end{theorem}

\section{Relative Calabi--Yau Structure} \label{sec:weak-cy}

Let $\Lambda \subseteq S^*M$ be a compact subanalytic Legendrian and $\widehat\Lambda \subseteq T^*M$ be the corresponding precompact subanalytic Lagrangian. Our main result in this section is that, firstly, the categories $\msh_\Lambda(\Lambda)$ and $\Sh_{\widehat\Lambda}(M)$ are smooth (Propositions \ref{prop:weak-rel-cy-nondeg} and \ref{prop:weak-cy-nondeg}), and secondly, when $M$ is orientable, the continuous adjunction
$$m_\Lambda^l: \msh_\Lambda(\Lambda) \rightleftharpoons \Sh_{\widehat\Lambda}(M) : m_\Lambda$$
admits a weak smooth Calabi--Yau structure, and $\msh_\Lambda(\Lambda)$ itself admits a weak smooth Calabi--Yau structure (Theorem \ref{thm:weak-relative-cy}), using Sato--Sabloff fiber sequence Theorem \ref{thm:sato-sabloff} and Sabloff--Serre duality Theorem \ref{thm:sabloff-serre}.

\subsection{Duality and functors for microlocal sheaves} 
We consider the category $\sA = \Sh_{\widehat\Lambda}(M)$ of compactly supported sheaves on manifolds with isotropic singular support. First, we realize the duality pair of sheaves with isotropic singular support as follows:

\begin{theorem}[{\cite[Theorem 1.3]{Kuo-Li-duality}}]\label{thm:duality-sheaf}
Let $\Lambda \subseteq S^*M$ be a compact subanalytic Legendrian. Then the category $\sA = \Sh_{\widehat\Lambda}(M)$ is dualizable with $\sA^\vee = \Sh_{-\widehat\Lambda}(M)$, where the counit and unit are 
\begin{align*}
\epsilon_\sA &= p_! \Delta^* : \Sh_{-\widehat{\Lambda} \times \widehat{\Lambda}}(M \times M) \rightarrow \cV \\
\eta_\sA &= \iota_{\widehat{\Lambda} \times -\widehat{\Lambda}}^* \Delta_* p^*
: \cV \rightarrow \Sh_{\widehat{\Lambda} \times -\widehat{\Lambda}}(M \times M)
\end{align*}
\end{theorem}

\begin{theorem}[{\cite[Theorem 1.1]{Kuo-Li-duality}}]\label{thm:fourier-mukai}
Let $\Lambda \subseteq S^*M$, $\Sigma \subseteq S^*N$ be compact subanalytic Legendrians. Then there is a Fourier--Mukai isomorphism 
\begin{align*}
\Sh_{-\widehat{\Lambda} \times \widehat{\Sigma}}(M \times N) &\xrightarrow{\sim} \Fun^L(\Sh_{\widehat{\Lambda}}(M), \Sh_{\widehat{\Sigma}}(N)) \\
K & \mapsto \pi_{2 !}(K \otimes \pi_1^*(-)).
\end{align*}
\end{theorem}
\begin{remark}\label{rem:dual-fourier-mukai}
The Fourier--Mukai isomorphism intertwines with the isomorphisms $\sA^\vee \otimes \sA = \Fun^L(\sA, \sA) = \Fun^L(\sA \otimes \sA^\vee, \cV)$. Therefore, we also have the dual Fourier--Mukai isomorphism
\begin{align*}
\Sh_{-\widehat{\Lambda} \times \widehat{\Sigma}}(M \times N) &\xrightarrow{\sim} \Fun^L(\Sh_{\widehat{\Lambda} \times -\widehat\Sigma}(M \times N), \cV) \\
K & \mapsto \pi_{!}(K \otimes -).
\end{align*}
\end{remark}

Let $\sA = \Sh_{\widehat\Lambda}(M)$, $\sB = \msh_\Lambda(\Lambda)$ and write $f = m_{\Lambda}^l: \msh_\Lambda(\Lambda) \rightleftharpoons \Sh_{\widehat\Lambda}(M): m_{\Lambda} = f^r$. Using the Fourier--Mukai isomorphism, we have the following identifications.

\begin{lemma}\label{lem:dual-bimod}
    Let $\Lambda \subseteq S^{*}M$ be a subanalytic Legendrian. 
    Then under the Fourier--Mukai isomorphism in Theorem \ref{thm:fourier-mukai}, 
    $$\Id_{\Sh_{\widehat\Lambda}(M)} \simeq \iota_{\widehat{\Lambda} \times -\widehat{\Lambda}}^*1_\Delta, \;\; m_{\Lambda}^l m_{\Lambda} \simeq m_{\widehat\Lambda \times -\widehat\Lambda}^l m_{\widehat\Lambda \times -\widehat\Lambda} 1_\Delta.$$
\end{lemma}
\begin{proof}
    First, the identity $\Id_{\Sh_{\widehat\Lambda}(M)} \simeq \iota_{\widehat{\Lambda} \times -\widehat{\Lambda}}^*1_\Delta$ follows from Theorem \ref{thm:duality-sheaf}. 
    Then consider $m_{\Lambda}^l m_{\Lambda}$. Using Sato-Sabloff fiber sequence Theorem \ref{thm:sato-sabloff}, we have for some Reeb flow $T_t: S^*M \to S^*M$,
    \begin{align*}
         \Hom(m_{\Lambda}^l m_{\Lambda} F, G) &= \Hom(\mathrm{Fib}(T_{-\epsilon}1_\Delta \circ F \rightarrow T_\epsilon1_\Delta \circ F), G)\\
         &= \Hom(\mathrm{Fib}(T_{-\epsilon}1_\Delta \rightarrow T_\epsilon 1_\Delta), \sHom(\pi_1^*F, \pi_2^!G)) \\
         &= \Hom(m_{\widehat\Lambda \times -\widehat\Lambda}^l m_{\widehat\Lambda \times -\widehat\Lambda} 1_\Delta, \sHom(\pi_1^*F, \pi_2^!G)) \\
         &= \Hom( m_{\widehat\Lambda \times -\widehat\Lambda}^l m_{\widehat\Lambda \times -\widehat\Lambda} 1_\Delta  \circ F, G).
    \end{align*}
    Therefore, under the Fourier--Mukai isomorphism, we get that
    $$m_{\Lambda}^l m_{\Lambda} \simeq m_{\widehat\Lambda \times -\widehat\Lambda}^l m_{\widehat\Lambda \times -\widehat\Lambda} 1_\Delta.$$
    This thus completes the proof.
\end{proof}


\begin{lemma}\label{lem:dual-bimod-sheaf2}
    Let $\Lambda \subseteq S^{*}M$ be a compact subanalytic Legendrian. Then under the dual Fourier--Mukai isomorphism in Remark \ref{rem:dual-fourier-mukai}, for some Reeb flow $T_t: S^*M \to S^*M$,
    \begin{gather*}
        \Id_{\Sh_{\widehat\Lambda}(M)} \simeq \pi_!(1_\Delta \otimes -), \;\; m_{\Lambda}^l m_{\Lambda} \simeq \pi_!(\mathrm{Fib}(T_{-\epsilon}1_\Delta \to T_\epsilon 1_\Delta) \otimes -).
    \end{gather*} 
\end{lemma}
\begin{proof}
    We write $p: M \to \{*\}$ and $p_2: M \times M \to \{*\}$. Consider the dual Fourier--Mukai isomorphism $\Sh_{\widehat\Lambda \times -\widehat\Lambda}(M \times M) \simeq \Fun^L(\Sh_{-\widehat\Lambda \times \widehat\Lambda}(M \times M), \cV)$. For $\Id_{\Sh_{\widehat\Lambda}(M)} = \iota_{\widehat{\Lambda} \times -\widehat{\Lambda}}^*1_\Delta$, using Theorem \ref{thm:duality-sheaf}, we can immediately get
    $$\pi_{!}(\iota_{\widehat{\Lambda} \times -\widehat{\Lambda}}^*1_\Delta \otimes K) = \pi_{!}(1_\Delta \otimes K).$$
    Then consider $m_{\Lambda}^l m_{\Lambda} = m_{\widehat{\Lambda} \times -\widehat{\Lambda}}^l m_{\widehat\Lambda \times -\widehat\Lambda} 1_{\Delta}$. Using Sato--Sabloff fiber sequence, we have
    \begin{align*}
         \Hom &(\pi_! ( m_{\widehat{\Lambda} \times -\widehat{\Lambda}}^l m_{\widehat\Lambda \times -\widehat\Lambda} 1_{\Delta} \otimes K), V) \\
         &= \Hom(m_{\widehat{\Lambda} \times -\widehat{\Lambda}}^l m_{\widehat\Lambda \times -\widehat\Lambda} 1_{\Delta}, \sHom(K, \omega_{M \times M} \otimes V)) \\
         &= \Hom(\mathrm{Fib}(T_{-\epsilon}1_\Delta \to T_\epsilon 1_\Delta), \sHom(K, \omega_{M \times M} \otimes V)) \\
         &= \Hom(\pi_! (\mathrm{Fib}(T_{-\epsilon}1_\Delta \to T_\epsilon 1_\Delta) \otimes K), V).
    \end{align*}
    Thus, using the duality of $\Sh_{\widehat\Lambda \times -\widehat\Lambda}(M \times M)$, we get
    $$\pi_!(m_{\widehat{\Lambda} \times -\widehat{\Lambda}}^l m_{\widehat\Lambda \times -\widehat\Lambda} 1_\Delta \otimes K) \simeq \pi_!(\mathrm{Fib}(T_{-\epsilon}1_\Delta \to T_\epsilon 1_\Delta) \otimes K).$$
    This completes the proof of the lemma.
\end{proof}

    Then we consider the category $\sB = \msh_\Lambda(\Lambda)$ of microsheaves on isotropic subsets in $S^*M$ and study the dualizability result. Recall the isomorphism in Theorem \ref{doubling} that $\sB = \msh_\Lambda(\Lambda) = \Sh_{\widehat{\Lambda}_{\cup}}(M)$ where 
    $$\widehat\Lambda_{\cup} = \widehat{\Lambda}_{\cup, \epsilon} = (\Lambda \times \bR_+) \cup (\Lambda_\epsilon \times \bR_+) \cup \bigcup\nolimits_{0 \leq s \leq \epsilon} \pi(\Lambda_s),$$
    where we omit the subscript $\epsilon$ when there is no confusion. By Theorem \ref{thm:duality-sheaf}, we know that category $\sB$ is dualizable, and the Fourier--Mukai isomorphism holds.

\begin{proposition}[{\cite[Theorem 1.4 \& Corollary 4.21]{Kuo-Li-duality}}]\label{thm:duality-microsheaf}
    Let $\Lambda \subseteq S^*M$ be a compact subanalytic Legendrian. Then for $\epsilon > 0$ sufficiently small, the category $\sB = \msh_\Lambda(\Lambda) = \Sh_{\widehat\Lambda_{\cup}}(M)$ is dualizable with $\sB^\vee = \Sh_{-\widehat\Lambda_{\cup}}(M)$, where the counit and unit are 
    \begin{align*}
    \epsilon_{\sB} &= p_! \Delta^* : \Sh_{-\widehat{\Lambda}_{\cup} \times \widehat{\Lambda}_{\cup}}(M \times M) \rightarrow \cV \\
    \eta_{\sB} &= \iota_{\widehat{\Lambda}_{\cup} \times -\widehat{\Lambda}_{\cup}}^* \Delta_* p^*
    : \cV \rightarrow \Sh_{\widehat{\Lambda}_{\cup} \times -\widehat{\Lambda}_{\cup}}(M \times M).
    \end{align*}
\end{proposition}
    
    Then it is clear that from Lemma \ref{lem:dual-bimod} that we can identify the identity functor and left dualizing functor. Recall that the colimit preserving adjunction $m_{\Lambda,{dbl}}^l: \msh_\Lambda(\Lambda) \rightleftharpoons \Sh_{\widehat\Lambda_{\cup}}(M) : m_{\Lambda,{dbl}}$ is a pair of inverse equivalences. From Lemma \ref{lem:dual-bimod}, we can then get a different characterization of these functors.

\begin{lemma}\label{lem:dual-bimod-microsheaf}
    Let $\Lambda \subseteq S^{*}M$ be a compact subanalytic Legendrian and let $\sB = \msh_\Lambda(\Lambda) = \Sh_{\widehat\Lambda_{\cup}}(M)$. Then under the Fourier--Mukai isomorphism in Theorem \ref{thm:fourier-mukai},
    \begin{align*}
    \Id_{\msh_\Lambda(\Lambda)} \simeq \iota_{\widehat{\Lambda}_{\cup} \times -\widehat{\Lambda}_{\cup}}^* 1_\Delta = m_{\widehat\Lambda_{\cup} \times -\widehat\Lambda_{\cup}}^l m_{\widehat\Lambda_{\cup} \times -\widehat\Lambda_{\cup}} 1_\Delta. 
    \end{align*}
    Under the dual Fourier--Mukai isomorphism in Remark \ref{rem:dual-fourier-mukai},
    \begin{align*}
        \Id_{\msh_\Lambda(\Lambda)} \simeq \pi_!(1_\Delta \otimes -) = \pi_!(\mathrm{Fib}(T_{-\epsilon}1_\Delta \to T_\epsilon 1_\Delta) \otimes -).
    \end{align*} 
\end{lemma}

    Under the isomorphism $\sB = \msh_\Lambda(\Lambda) = \Sh_{\widehat{\Lambda}_{\cup}}(M)$, we can identify the adjunction of microlocalization $\sB = \msh_\Lambda(\Lambda) \rightleftharpoons \Sh_{\widehat\Lambda}(M) = \sA$ with an adjunction $ \sB = \Sh_{\widehat{\Lambda}_{\cup}}(M) \rightleftharpoons \Sh_{\widehat\Lambda}(M) = \sA$ as follows
    \[\xymatrix{
    \Sh_{\widehat\Lambda_{\cup}}(M) \ar[r]^{m_{\Lambda,{dbl}}}_{\sim} & \msh_\Lambda(\Lambda) \ar[r]^{m_{\Lambda,{dbl}}^l}_{\sim} \ar@{=}[d] & \Sh_{\widehat\Lambda_{\cup}}(M) \ar[d]^{m_{\Lambda}^l m_{\Lambda,{dbl}}} \\
    \Sh_{\widehat\Lambda}(M) \ar[r]^{m_{\Lambda}} \ar[u]^{m_{\Lambda,{dbl}}^l m_{\Lambda}} & \msh_\Lambda(\Lambda) \ar[r]^{m_{\Lambda}^l} & \Sh_{\widehat\Lambda}(M).
    }\]
    It follows from Theorem \ref{doubling} that we have
    $$\iota_{\Lambda}^* = m_{\Lambda}^l m_{\Lambda,{dbl}}: \Sh_{\widehat\Lambda_{\cup}}(M) \to  \Sh_{\widehat\Lambda}(M).$$
    Then, consider Fourier--Mukai isomorphism and the natural morphisms of sheaves on the product. We can identify the functor $f \otimes f^{r\vee} : \sB \otimes \sB^\vee \to \sA \otimes \sA^\vee$ under the isomorphism 
    $\sB \otimes \sB^\vee = \Sh_{\widehat{\Lambda}_{\cup} \times -\widehat{\Lambda}_{\cup}}(M \times M)$.

\begin{lemma}[{\cite[Proposition 3.10]{Kuo-Li-duality}}]\label{lem:bimod-adjoint-sheaf}
    Let $\Lambda \subseteq S^*M$ be a compact subanalytic Legendrian. There is a commutative diagram
    \[\xymatrix{
    \Sh_{\widehat\Lambda_{\cup}}(M) \otimes \Sh_{-\widehat\Lambda_{\cup}}(M) \ar[r]^{\sim} \ar[d]_{\iota_{\widehat\Lambda}^* \otimes \iota_{-\widehat\Lambda}^*} & \Sh_{\widehat{\Lambda}_{\cup} \times -\widehat{\Lambda}_{\cup}}(M \times M) \ar[d]^{\iota_{\widehat\Lambda \times -\widehat\Lambda}^*} \\
    \Sh_{\widehat\Lambda}(M) \otimes \Sh_{-\widehat\Lambda}(M) \ar[r]^{\sim} & \Sh_{\widehat\Lambda \times -\widehat\Lambda}(M \times M).
    }\]
\end{lemma}

Finally, given the description of the unit and counit for $\sA = \Sh_{\widehat\Lambda}(M)$ and $\sB = \msh_\Lambda(\Lambda) = \Sh_{\widehat\Lambda_{\cup}}(M)$, their Hochschild homologies $\HH_*(\sA)$ and $\HH_*(\sB)$ can be computed as follows.

\begin{lemma}
Let $\Lambda \subseteq S^{*}M$ be a compact subanalytic Legendrian. Then
\begin{gather*}
    \HH_*(\Sh_{\widehat\Lambda}(M)) = p_! \Delta^* \iota_{\widehat\Lambda \times -\widehat\Lambda}^* 1_\Delta, \\
    \HH_*(\msh_\Lambda(\Lambda)) = p_! \Delta^*  \iota_{\widehat\Lambda_{\cup} \times -\widehat\Lambda_{\cup}}^* 1_\Delta = p_! \Delta^*  m_{\widehat\Lambda_{\cup} \times -\widehat\Lambda_{\cup}}^l m_{\widehat\Lambda_{\cup} \times -\widehat\Lambda_{\cup}} 1_\Delta. 
\end{gather*}
\end{lemma}

\begin{remark}
    Using Lemma \ref{lem:bimod-adjoint-sheaf}, we know that the natural map $\HH_*(\msh_\Lambda(\Lambda)) \to \HH_*(\Sh_{\widehat\Lambda}(M))$ factors through the morphism induced by $\iota_{\widehat\Lambda \times -\widehat\Lambda}^*: m_{\widehat\Lambda_{\cup} \times -\widehat\Lambda_{\cup}}^l m_{\widehat\Lambda_{\cup} \times -\widehat\Lambda_{\cup}} 1_\Delta \to m_{\widehat\Lambda \times -\widehat\Lambda}^l m_{\widehat\Lambda \times -\widehat\Lambda} 1_\Delta$:
    $$p_! \Delta^*  m_{\widehat\Lambda_{\cup} \times -\widehat\Lambda_{\cup}}^l m_{\widehat\Lambda_{\cup} \times -\widehat\Lambda_{\cup}} 1_\Delta \to p_! \Delta^* m_{\widehat\Lambda \times -\widehat\Lambda}^l m_{\widehat\Lambda \times -\widehat\Lambda} 1_\Delta \to p_! \Delta^* \iota_{\widehat\Lambda \times -\widehat\Lambda}^* 1_\Delta.$$
\end{remark}

On the other hand, using the unit and counit of $\sA = \Sh_{\widehat\Lambda}(M)$ and $\sB = \msh_\Lambda(\Lambda)$, the Hochschild cohomologies $\HH^*(\sA)$ and $\HH^*(\sB)$ can also be computed as follows.

\begin{lemma}
    Let $\Lambda \subseteq S^{*}M$ be a compact subanalytic Legendrian. Then
    \begin{gather*}
        \HH^*(\Sh_{\widehat\Lambda}(M)) = \End( \iota_{\widehat\Lambda \times -\widehat\Lambda}^* 1_\Delta ),\\
        \HH^*(\msh_\Lambda(\Lambda)) = \End(  \iota_{\widehat\Lambda_{\cup} \times -\widehat\Lambda_{\cup}}^* 1_\Delta ) = \End ( m_{\widehat\Lambda_{\cup} \times -\widehat\Lambda_{\cup}}^l m_{\widehat\Lambda_{\cup} \times -\widehat\Lambda_{\cup}} 1_\Delta ).
    \end{gather*}
\end{lemma}

\subsection{Weak Calabi--Yau structure on local systems}
    We consider the special case of local systems on a closed orientable manifold $\sB = \Loc(M)$. The smooth Calabi--Yau structure on $\Loc(M)$ is constructed in \cite[Section 5.1.1]{Brav-Dyckerhoff1} and recovers the Poincar\'e duality.

    For a vector bundle $E$ over a manifold $M$, we define the Thom space $\mathrm{Th}(E)$ to be the sphere bundle over $M$ by compactifying by adding one point at infinity in each fiber. In particular, the Thom space of a smooth manifold $M$ is the Thom space of its tangent bundle $\mathrm{Th}(TM)$.

\begin{definition}[{\cite[Definition 4.3.4]{Kochman}}]
    A smooth manifold $M$ is orientable with respect to $\cV$ if $TM$ is orientable with respect to $\cV$, namely for the Thom space there exists
    $$1_\cV[n] \to \Gamma(\mathrm{Th}(TM);1_{TM})$$
    such that the restriction to each fiber $1_\cV[n] \to \Gamma(S^n; 1_{\bR^n})$ is an equivalence.
\end{definition}
\begin{remark}
    When $\cV = \Mod(\Bbbk)$ where $\Bbbk$ is a ring or more generally a commutative ring spectrum, an orientation on $M$ is simply a map $\Bbbk[n] \to H^*(\mathrm{Th}(TM), M; \Bbbk)$ such that the restriction to each fiber $\Bbbk[n] \to H^*(S^n, *; \Bbbk)$ is a generator. This agrees with the standard definition.
\end{remark}

 Here is an alternative description of the orientation. Let $M$ be a smooth manifold. Consider the composition of the tangent bundle map and the delooping of the $J$-homomorphism
$$M \xrightarrow{TM} BO \xrightarrow{BJ} B\SPH^\times.$$
The composition is equal to the Thom spectrum $\mathrm{Th}^\infty(TM)$ of $M$ \cite[Proposition 7.7 \& Corollary 7.9]{Volpe-six-operations}. The Thom spectrum defines an invertible sheaf on $M$ with coefficients in $\SPH$ \cite[Equation 7.3]{Volpe-six-operations}. 
For a general $\cV$, being stable implies that $\cV \in \Sp = \Mod(\SPH)$, there is a symmetric monoidal functor
\begin{align*}
 \Sp &\rightarrow \cV \\
 X &\mapsto X \otimes 1_\cV.
\end{align*}
In particular, it induces a map of ring spectral $\SPH = \Hom_{\Sp}(\SPH, \SPH) \rightarrow \Hom_{\cV}(1_\cV,1_\cV) \eqqcolon \End_\cV(1_\cV)$ and hence a group homomorphism $\SPH^\times \rightarrow \Aut_\cV(1_\cV)$.
We can show that an orientation on $M$ with respect to $\cV$ is a null homotopy of the composition
$$M \xrightarrow{TM} BO \xrightarrow{BJ} B\SPH^\times \rightarrow B\hspace{-1pt}\Aut_\cV(1_\cV),$$
which is equivalent to a generating section of the sheaf $\mathrm{Th}^\infty(TM) \otimes 1_\cV$.

\begin{proposition}\label{prop:poincare}
    Let $\cV$ be a rigid symmetric monoidal category and $M$ be a smooth manifold. 
    Then the local system classified by the the Thom spectrum $\mathrm{Th}^\infty(TM) \otimes 1_\cV$
    $$M \xrightarrow{TM} BO \xrightarrow{BJ} B\SPH^\times \rightarrow B\hspace{-1pt}\Aut_\cV(1_\cV)$$
    is represented by the dualizing sheaf $\omega_M \coloneqq p^! 1_\cV$ where $p: M \rightarrow \{*\}$ is the projection to a point, and an orientation on $M$ with respect to $\cV$ is equivalent to an isomorphism $\omega_M \simeq 1_M[-n]$ or equivalently a null homotopy of the composition $M \to B\hspace{-1pt}\Aut_\cV(1_\cV)$.
\end{proposition}
\begin{proof}
First, we know the equality $\mathrm{Th}^\infty(TM) \otimes 1_\cV = \omega_M$ by \cite[Theorem 7.11]{Volpe-six-operations} for the case of spectra and the general case follows from changing coefficient using \cite[Proposition 6.16]{Volpe-six-operations}. Moreover, for $\pi: \mathrm{Th}(TM) \to M$ and $j: M \hookrightarrow \mathrm{Th}(TM)$, we know the equality
$$\mathrm{Th}^\infty(TM) \otimes 1_\cV = \pi_*\mathrm{Cofib}(j_*j^*1_{\mathrm{Th}(TM)} \to 1_{\mathrm{Th}(TM)}) = \pi_*1_{TM}.$$
by \cite[Proposition 7.7]{Volpe-six-operations}. Therefore, an orientation on $M$ is equivalent to an isomorphism $\omega_M \simeq 1_M[n]$ is the special case of the fact that, for an invertible local system, a non-zero global section is equivalent to a trivialization of the invertible local system.
\end{proof}

\begin{proposition}
Let $M$ be a closed manifold of dimension $n$. 
Then an orientation of $M$ induces an $n$-dimensional weak smooth Calabi-Yau structure on $\Loc(M)$.
\end{proposition}

\begin{proof}
We know that under the dual Fourier--Mukai isomorphism, $\Id_{\Loc(M)} = p_!\Delta^*(-)$. By \cite[Proposition 5.4.13]{KS}, for any map $f: M \rightarrow N$ between manifolds and any $L \in \Loc(Y)$, 
$f^! L = f^* L \otimes \omega_f$. In particular, $\Delta^*(-) = \Delta^! (-) \otimes \omega_M$ on $\Loc(M \times M)$.
Thus 
$$\Id_{\Loc(M)}^! = \iota_{\varnothing}^* \Delta_* (\omega_M).$$
Last, we notice that, for any $L \in \Loc(M)$, ${p_1}_! \iota_\varnothing^*(\Delta_*L) = L$ where $p_1: M \times M \rightarrow M$ is the projection to the first component. 
Since we can check that, for any $L^\prime \in \Loc(M)$,
$$\Hom({p_1}_! \iota_\varnothing^*(\Delta_*L), L^\prime) = \Hom(\iota_\varnothing^*(\Delta_*L), p_1^! L^\prime) = \Hom( \Delta_*L, p_1^! L^\prime) = \Hom(L,L^\prime).$$
Again we use the fact that $p_1^! L^\prime$ is a local system.
This implies an equality $1_M[n] = \omega_M$ 
induces an equality $\Id_{\Loc(M)}^! = \Id_{\Loc(M)}[n]$.
Thus we conclude the proof by the last Proposition \ref{prop:poincare}. 
\end{proof}

\subsection{Weak Calabi--Yau structure on microlocalization}
    Consider the case when $\sA = \Sh_{\widehat\Lambda}(M)$, $\sB = \msh_{\Lambda}(\Lambda)$, and $f = m_\Lambda^l: \msh_{\Lambda}(\Lambda) \rightarrow \Sh_{\widehat\Lambda}(M) : m_\Lambda = f^r$. Recall from Lemma \ref{lem:dual-bimod} that 
    \begin{gather*}
        \Id_\sA = \iota_{\widehat\Lambda \times -\widehat\Lambda}^*1_\Delta, \;\; 
        ff^r = m_{\widehat\Lambda \times -\widehat\Lambda}^l m_{\widehat\Lambda \times -\widehat\Lambda} 1_\Delta.
    \end{gather*}
    Consider the fiber sequence in the non-degeneracy condition which is
    $$\mathrm{Fib}(c) \longrightarrow ff^r \longrightarrow \Id_\sA.$$
    The following proposition is a functorial version of the Sato--Sabloff sequence Theorem \ref{thm:sato-sabloff}, from which we can identify the fiber in the above sequence.

\begin{definition}\label{def:wrap-once}
    Let $\Lambda \subseteq S^*M$ be a compact subanalytic Legendrian and let $T_t: S^*M \to S^*M$ be a positive contact flow. Then the wrap once functor is defined by $S_\Lambda^+ = \iota_{\widehat\Lambda}^* \circ T_\epsilon: \Sh_{\widehat\Lambda}(M) \to \Sh_{\widehat\Lambda}(M)$.
\end{definition}
\begin{remark}
    When $M$ is a closed manifold, the functor $\iota_{\widehat\Lambda}^*: \Sh(M) \to \Sh_{\widehat\Lambda}(M)$ is given by the colimit of all positive contact isotopies supported away from $\Lambda \subseteq S^*M$ \cite[Proposition 1.2]{Kuo-wrapped-sheaves}. This justifies the terminology.
\end{remark}

\begin{proposition}[{\cite[Theorem 4.3]{Kuo-Li-spherical}}]\label{prop:fiber-sequence}
    Let $M$ be a manifold and $\Lambda \subseteq S^*M$ be a compact subanalytic Legendrian. Then there is a fiber sequence 
    $$m_{\Lambda}^l m_{\Lambda} \to \Id_{\Sh_{\widehat\Lambda}(M)} \to S_{\Lambda}^+.$$
\end{proposition}
\begin{proof}
    For any $K \in \Sh_{\widehat\Lambda \times -\widehat\Lambda}(M \times M)$, using Sato-Sabloff fiber sequence Theorem \ref{thm:sato-sabloff}, we know
    $$\Hom(T_\epsilon 1_\Delta, K) \rightarrow \Hom(1_\Delta, K) \rightarrow \Gamma(S^*(M \times M), \mhom( m_{\widehat\Lambda \times -\widehat\Lambda}1_\Delta, m_{\widehat\Lambda \times -\widehat\Lambda} K) ).$$
    Therefore, by adjunction we have the fiber sequence
    \begin{equation*}
    \Hom(\iota_{\widehat\Lambda \times -\widehat\Lambda}^*T_\epsilon 1_\Delta , K) \rightarrow \Hom(\iota_{\widehat\Lambda \times -\widehat\Lambda}^*1_\Delta, K) \rightarrow \Hom (m_{\widehat\Lambda \times -\widehat\Lambda}^l m_{\widehat\Lambda \times -\widehat\Lambda}1_\Delta, K).
    \end{equation*}
    Finally, one observes that $S_{\Lambda}^+ = \iota_{\widehat\Lambda}^* \circ T_\epsilon \simeq \iota_{\widehat\Lambda \times -\widehat\Lambda}^*T_\epsilon 1_\Delta$ under the Fourier--Mukai isomorphism. This completes the proof of the fiber sequence.
\end{proof}

    Then we need to show the isomorphisms with the corresponding inverse dualizing bimodules defined by left adjoints. The following proposition is a functorial version of the Sabloff--Serre duality Theorem \ref{thm:sabloff-serre}.

\begin{proposition}\label{prop:weak-rel-cy-nondeg}
    Let $M$ be a compact manifold of dimension $n$ and $\Lambda \subseteq S^*M$ be a compact subanalytic Legendrian. Then $\Sh_{\widehat\Lambda}(M)$ is a smooth category, and there is an isomorphism
    $$\Id^!_{\Sh_{\widehat\Lambda}(M)} \simeq S_{\Lambda}^+ \otimes \omega_M.$$
    In particular, when $M$ is oriented, then $\Id^!_{\Sh_{\widehat\Lambda}(M)} \simeq S_{\Lambda}^+[-n]$.
\end{proposition}
\begin{proof}
    Fix any $K \in \Sh_{\widehat\Lambda \times -\widehat\Lambda}(M \times M)$. Then, by Sabloff--Serre duality Theorem \ref{thm:sabloff-serre}, we have
    \begin{equation*}
        \Hom(\iota_{\widehat\Lambda \times -\widehat\Lambda}^*T_\epsilon (\Delta_*\omega_M), K) = \pi_!(1_\Delta \otimes K).
    \end{equation*}
    This shows that $\iota_{\widehat\Lambda \times -\widehat\Lambda}^*(\Delta_*\omega_M)$ is the left adjoint of $\pi_!(1_\Delta \otimes -)$. Hence $\Sh_{\widehat\Lambda}(M)$ is smooth and we have an isomorphism of functors $\Id^!_{\Sh_{\widehat\Lambda}(M)} \simeq S_{\Lambda}^+ \otimes \omega_M$.
\end{proof}
\begin{remark}
    Define $v: M \times M \to M \times M, \, (x, y) \mapsto (y, x)$. Consider the equivalence induced by the duality of compactly generated categories
    $$\VD{\widehat\Lambda \times -\widehat\Lambda}: \Sh_{\widehat\Lambda \times -\widehat\Lambda}^c(M \times M)^{op} \to \Sh_{-\widehat\Lambda \times \widehat\Lambda}^c(M \times M).$$
    Then we can show that $\Id_{\Sh_{\widehat\Lambda}(M)}^! \simeq v_*\VD{\widehat\Lambda \times -\widehat\Lambda}(\iota_{\widehat\Lambda \times -\widehat\Lambda}^* 1_\Delta)$, and the result is equivalent to
    $$\iota_{\widehat\Lambda \times -\widehat\Lambda}^*T_\epsilon (\Delta_*\omega_M) = v_*\VD{\widehat\Lambda \times -\widehat\Lambda}(\iota_{\widehat\Lambda \times -\widehat\Lambda}^* \Delta_*1_M).$$
\end{remark}

\begin{remark}
    Using Proposition \ref{prop:leftdual-serre}, we know that the above isomorphism recovers the Sabloff--Serre duality, namely when $G \in \Sh_{\widehat\Lambda}^b(M)$, we have
    $$\Hom(F, G)^\vee = \Hom(S_{\Lambda}^+F, G) = \Hom(F, S_{\Lambda}^- G).$$
    In other words, the proper Calabi--Yau structure on $m_{\Lambda}: \Sh_{\widehat\Lambda}^b(M) \rightleftharpoons \msh_\Lambda(\Lambda) : m_{\Lambda}^r$ is induced by the Sabloff--Serre duality. The compact support assumption of the sheaves is crucial. See Remark \ref{rem:compact-support} for a counterexample when $M = \bR$, $\Lambda = \{(0, 1)\} \subseteq S^*\bR$ and $F = G = 1_{[0, +\infty)}$.
\end{remark}


    From the first isomorphism, we can deduce the third isomorphism using duality
    $$(S_{\Lambda}^+)^! \simeq \Id_{\Sh_{\widehat\Lambda}(M)} \otimes \,\omega_M.$$
    The two isomorphisms force the map in the middle to be an isomorphism, namely
    $$(m_{\Lambda}^l m_{\Lambda})^! \simeq m_{\Lambda}^l m_{\Lambda} \otimes \omega_M[1].$$
    Indeed, we can prove the following stronger non-degeneracy result. Consider $\sB = \msh_{\Lambda}(\Lambda) = \Sh_{\widehat{\Lambda}_{\cup}}(M)$, where the equivalence by Theorem \ref{doubling} is induced by the adjunction of microlocalizations $m_{\Lambda,{dbl}}^l: \msh_\Lambda(\Lambda) \to \Sh_{\widehat{\Lambda}_{\cup}}(M): m_{\Lambda,{dbl}}$.

\begin{proposition}\label{prop:weak-cy-nondeg}
    Let $M$ be a connected manifold of dimension $n$ and $\Lambda \subseteq S^*M$ be a compact subanalytic Legendrian. Then $\msh_\Lambda(\Lambda)$ is a smooth category, and there is an isomorphism
    $$\Id_{\msh_\Lambda(\Lambda)}^! \simeq \Id_{\msh_\Lambda(\Lambda)} \otimes \,\omega_M[1].$$
    In particular, when $M$ is orientable, then $\Id_{\msh_\Lambda(\Lambda)}^! \simeq \Id_{\msh_\Lambda(\Lambda)}[1-n]$.
\end{proposition}
\begin{proof}
    Fix any $K \in \Sh_{\widehat\Lambda_{\cup} \times -\widehat\Lambda_{\cup}}(M \times M)$. Then, by Sabloff--Serre duality Theorem \ref{thm:sabloff-serre}, we have
    \begin{align*}
        \Hom &(m_{\widehat\Lambda_{\cup} \times -\widehat\Lambda_{\cup}}^l m_{\widehat\Lambda_{\cup} \times -\widehat\Lambda_{\cup}} (\Delta_*\omega_M)[1], K) = \Hom(\iota_{\widehat\Lambda_{\cup} \times -\widehat\Lambda_{\cup}}^*\mathrm{Fib}(T_{-\epsilon} \Delta_*\omega_M \to T_\epsilon \Delta_*\omega_M, K)[1-n] \\
        &= \Hom(\iota_{\widehat\Lambda_{\cup} \times -\widehat\Lambda_{\cup}}^*\mathrm{Cofib}(T_{-\epsilon} \Delta_*\omega_M \to T_\epsilon \Delta_*\omega_M), K)[-n]
        = \pi_!(\mathrm{Fib}(T_{-\epsilon} 1_\Delta \to T_\epsilon 1_\Delta) \otimes K).
    \end{align*}
    This shows that $m_{\widehat\Lambda_{\cup} \times -\widehat\Lambda_{\cup}}^l m_{\widehat\Lambda_{\cup} \times -\widehat\Lambda_{\cup}} (\Delta_*\omega_M)[1]$ is the left adjoint of $\pi_!(\mathrm{Fib}(T_{-\epsilon} 1_\Delta \to T_\epsilon 1_\Delta) \otimes -)$. Hence $\msh_\Lambda(\Lambda)$ is smooth and $\Id^!_{\msh_\Lambda(\Lambda)} \simeq \Id_{\msh_\Lambda(\Lambda)} \otimes \,\omega_M[1]$.
\end{proof}

\begin{remark}
    Define $v: M \times M \to M \times M, \, (x, y) \mapsto (y, x)$. Consider the equivalence induced by the duality of compactly generated categories
    $$\VD{\widehat\Lambda_\cup \times -\widehat\Lambda_\cup}: \Sh_{\widehat\Lambda_\cup \times -\widehat\Lambda_\cup}^c(M \times M)^{op} \to \Sh_{-\widehat\Lambda_\cup \times \widehat\Lambda_\cup}^c(M \times M).$$
    Then we have $\Id_{\msh_{\Lambda}(\Lambda)}^! \simeq v_*\VD{\widehat\Lambda_\cup \times -\widehat\Lambda_\cup}(m_{\widehat\Lambda_\cup \times -\widehat\Lambda_\cup}^l m_{\widehat\Lambda_\cup \times -\widehat\Lambda_\cup} 1_\Delta)$, and the result is equivalent to
    $$m_{\widehat\Lambda_\cup \times -\widehat\Lambda_\cup}^l m_{\widehat\Lambda_\cup \times -\widehat\Lambda_\cup}  (\Delta_*\omega_M)[1] = v_*\VD{\widehat\Lambda_\cup \times -\widehat\Lambda_\cup}(m_{\widehat\Lambda_\cup \times -\widehat\Lambda_\cup}^l m_{\widehat\Lambda_\cup \times -\widehat\Lambda_\cup} \Delta_*1_M).$$
\end{remark}

    We need to show that the corresponding isomorphisms come from the (relative) Hochschild homologies. Since the isomorphism in Proposition \ref{prop:weak-rel-cy-nondeg} corresponds a map $(m_{\Lambda}^l m_{\Lambda})^! \to m_{\Lambda}^l m_{\Lambda} \otimes \,\omega_M[1]$ while the isomorphism in Proposition \ref{prop:weak-cy-nondeg} corresponds to a map $\Id_{\msh_\Lambda(\Lambda)}^! \to \Id_{\msh_\Lambda(\Lambda)} \otimes \,\omega_M[1]$, we will need to understand the relation between the isomorphisms.

\begin{proposition}\label{prop:cy-compatible}
    Let $M$ be a connected manifold and $\Lambda \subseteq S^*M$ be a compact subanalytic Legendrian. 
    There is a commutative diagram between the induced isomorphisms in Propositions \ref{prop:weak-rel-cy-nondeg} and \ref{prop:weak-cy-nondeg}
    \[\xymatrix{
    (m_{\Lambda}^l m_{\Lambda})^! \ar[r]^-{\sim} \ar[d] & m_{\Lambda}^l m_{\Lambda} \otimes \,\omega_M[1]\\
    m_\Lambda^l \circ \Id_{\msh_\Lambda(\Lambda)}^! \circ m_\Lambda \ar[r]^-{\sim} & m_\Lambda^l \circ (\Id_{\msh_\Lambda(\Lambda)} \otimes \,\omega_M[1]) \circ m_\Lambda \ar[u].
    }\]
\end{proposition}
\begin{proof}
    Consider the morphism $\eta_{\msh_\Lambda(\Lambda)} \to (m_\Lambda^{\vee} \otimes m_\Lambda^l) \circ \eta_{\Sh_{\widehat\Lambda}(M)} \to (\Id_{\Sh_{\widehat\Lambda}(M)^\vee} \otimes m_\Lambda^lm_\Lambda) \circ \eta_{\Sh_{\widehat\Lambda}(M)}$. By the Fourier--Mukai isomorphism and Lemmas \ref{lem:dual-bimod}, \ref{lem:dual-bimod-microsheaf} and \ref{lem:bimod-adjoint-sheaf}, it can be written as
    \begin{align*}
        m_{\widehat\Lambda_{\cup} \times -\widehat\Lambda_{\cup}}^l m_{\widehat\Lambda_{\cup} \times - \widehat\Lambda_{\cup}} 1_\Delta & \to (\iota_{\widehat\Lambda}^* \otimes \iota_{-\widehat\Lambda}^*) m_{\widehat\Lambda_{\cup} \times -\widehat\Lambda_{\cup}}^l m_{\widehat\Lambda_{\cup} \times - \widehat\Lambda_{\cup}} 1_\Delta \\
        &\xrightarrow{\sim} \iota_{\widehat\Lambda \times -\widehat\Lambda}^* m_{\widehat\Lambda_{\cup} \times -\widehat\Lambda_{\cup}}^l m_{\widehat\Lambda_{\cup} \times - \widehat\Lambda_{\cup}} 1_\Delta \xrightarrow{\sim} m_{\widehat\Lambda \times -\widehat\Lambda}^l m_{\widehat\Lambda \times -\widehat\Lambda} 1_\Delta.
    \end{align*}
    Consider the morphism $\epsilon_{\msh_\Lambda(\Lambda)} \to \epsilon_{\msh_\Lambda(\Lambda)} \circ (m_\Lambda^{l\vee} \otimes m_\Lambda) \to \epsilon_{\Sh_{\widehat\Lambda}(M)} \circ (\Id_{\Sh_{\widehat\Lambda}(M)^\vee} \otimes m_\Lambda^lm_\Lambda)$. By the dual Fourier--Mukai isomorphism and Lemma \ref{lem:dual-bimod-sheaf2}, \ref{lem:dual-bimod-microsheaf}, and \ref{lem:bimod-adjoint-sheaf}, it can be written as
    \begin{align*}
        \pi_!(\mathrm{Fib}(T_{-\epsilon}1_\Delta \rightarrow T_\epsilon 1_\Delta) & \otimes (m_{\Lambda_{\cup,\epsilon}}^l m_\Lambda \otimes m_{-\Lambda_{\cup,\epsilon}}^l m_{-\Lambda})(-)) \\
        &\xrightarrow{\sim} \pi_!(\mathrm{Fib}(T_{-\epsilon}1_\Delta \rightarrow T_\epsilon 1_\Delta) \otimes -) \xrightarrow{\sim} \pi_!(\mathrm{Fib}(T_{-\epsilon}1_\Delta \rightarrow T_\epsilon 1_\Delta) \otimes -).
    \end{align*}
    Let $F, G \in \Sh_{\widehat\Lambda}(M)$ and $F_{dbl}, G_{dbl} \in \Sh_{\widehat\Lambda_{\cup}}(M)$ be defined by
    $$F_{dbl} = m_{\Lambda,{dbl}}^l m_{\Lambda} F, \;\; G_{dbl} = m_{-\Lambda,{dbl}}^l m_{-\Lambda} G.$$
    Consider the morphism $\Id_{\msh_\Lambda(\Lambda)}^! \to \Id_{\msh_\Lambda(\Lambda)} \otimes\, \omega_M[1]$ induced by Proposition \ref{prop:weak-cy-nondeg}, and the morphism $(m_\Lambda^l m_\Lambda)^! \to m_\Lambda^l m_\Lambda \otimes \,\omega_M[1]$ induced by Proposition \ref{prop:weak-rel-cy-nondeg}. Following Remark \ref{cohomology-v-smooth}, these can be realized as the vertical morphisms in the commutative diagram
    \[\xymatrix{
        \pi_!(\mathrm{Fib}(T_{-\epsilon}1_\Delta \rightarrow T_\epsilon 1_\Delta) \otimes (F_{dbl} \boxtimes G_{dbl}))[n-1]  \ar_{\ref{prop:weak-cy-nondeg}}[d] 
        & \pi_!(\mathrm{Fib}(T_{-\epsilon}1_\Delta \rightarrow T_\epsilon 1_\Delta) \otimes (F \boxtimes G))[n-1] \ar_-{\sim}[l]  \ar^{\ref{prop:weak-rel-cy-nondeg}}[d] \\
        \Hom(m_{\widehat\Lambda_{\cup} \times -\widehat\Lambda_{\cup}}^l m_{\widehat\Lambda_{\cup} \times - \widehat\Lambda_{\cup}} (\Delta_*\omega_M), F_{dbl} \boxtimes G_{dbl}) \ar^-{\sim}[r] & \Hom(m_{\widehat\Lambda \times -\widehat\Lambda}^l m_{\widehat\Lambda \times -\widehat\Lambda} (\Delta_*\omega_M), F \boxtimes G).
    }\]
    Therefore, the commutativity of the diagram in the proposition is equivalent to the the commutativity of the above diagram. Hence it is clear that the diagram commutes.
\end{proof}

    When $M$ is oriented, by Propositions \ref{prop:fiber-sequence}, \ref{prop:weak-rel-cy-nondeg}, \ref{prop:weak-cy-nondeg} and \ref{prop:cy-compatible}, we have a commutative diagram of morphisms of functors as follows
    \[\xymatrix{
    1_\cV[n-1] \ar_{\ref{prop:weak-cy-nondeg}}[d] \ar^{\sim}[r] & 1_\cV[n-1] \ar^-{\ref{prop:fiber-sequence}}[r] \ar^{\ref{prop:weak-rel-cy-nondeg}}[d] & 0 \ar^{\ref{prop:fiber-sequence}}[d] \\
    \Hom(\Id_{\msh_\Lambda(\Lambda)}^!,  \Id_{\msh_\Lambda(\Lambda)}) \ar[r] & \Hom((m_{\Lambda}^l m_{\Lambda})^!, m_{\Lambda}^l m_{\Lambda}) \ar[r] & \Hom(\Id_{\Sh_{\widehat\Lambda}(M)}^!, \Id_{\Sh_{\widehat\Lambda}(M)}).
    }\]
    Therefore, we can show the existence of a weak Calabi--Yau structure on the continuous adjunction pair coming from microlocalization.

\begin{theorem}\label{thm:weak-relative-cy}
    Let $M$ be a connected oriented manifold of dimension $n$ and $\Lambda \subseteq S^*M$ be a compact subanalytic Legendrian. Then the continuous adjunction
    $$m_{\Lambda}^l: \msh_\Lambda(\Lambda) \rightleftharpoons \Sh_{\widehat\Lambda}(M) : m_{\Lambda}$$
    admits an $n$-dimensional weak smooth relative Calabi--Yau structure, and $\msh_\Lambda(\Lambda)$ admits an $(n-1)$-dimensional weak Calabi--Yau structure.
\end{theorem}

\section{Circle Action and Calabi--Yau Structures}

Consider the weak Calabi--Yau structures on the category $\msh_\Lambda(\Lambda)$ the adjunction
$$m_\Lambda^l: \msh_\Lambda(\Lambda) \rightleftharpoons \Sh_{\widehat\Lambda}(M) : m_\Lambda.$$
In this section, we will upgrade them to strong (relative) Calabi--Yau structures by understanding the circle actions on the Hochschild homology of microlocal sheaf categories (Theorem \ref{thm:strong-cy} and \ref{thm:strong-rel-cy}). 

First, we build acceleration morphisms on the (relative) Hochschild homologies and cohomologies of the categories
\begin{gather*}
    \Gamma(\Lambda, 1_\Lambda) \to \HH^*(\msh_\Lambda(\Lambda)) \xrightarrow{\sim} \HH_*(\msh_\Lambda(\Lambda))[1-n], \\
    \Gamma(\widehat\Lambda, 1_{\widehat\Lambda}) \to \HH^*(\msh_\Lambda(\Lambda), \Sh_{\widehat\Lambda}(M)) \xrightarrow{\sim} \HH_*(\msh_\Lambda(\Lambda), \Sh_{\widehat\Lambda}(M))[-n],
\end{gather*}
and explain that the weak (relative) Calabi--Yau structures are exactly the elements induced by identities $1 \in \Gamma(\Lambda, 1_\Lambda)$ and $1 \in \Gamma(\widehat\Lambda, 1_{\widehat\Lambda})$ (Theorem \ref{thm:accelerate-hochschild}). Then we prove $S^1$-equivariancy of the acceleration morphisms (Propositions \ref{prop:accelerate-equivariant} and \ref{prop:accelerate-equivariant-rel}) using a sheaf theoretic model of the $S^1$-action on the constant loops (Propositions \ref{prop:cyclic-well-define} and \ref{prop:cyclic-well-define-rel}) and Hochschild homologies (Lemmas \ref{lem:cyclic-homology} and \ref{lem:cyclic-homology-rel}).

\subsection{Constant orbits in Hochschild invariants}
In this section, we consider the acceleration maps on the Hochschild homology of sheaves and microsheaves. This is a topological phenomenon analogous to the inclusion map of constant loops into the free loop space
$$H_*(M) \rightarrow H_*(L M) \xrightarrow{\sim} \HH_*(\Loc(M)).$$
In this section, we try to construct the constant orbits on the Hochschild homology and cohomology of sheaves and microsheaves, also known as the acceleration morphisms.

We will need to carefully choose nonnegative Hamiltonian pushoffs for the diagonal in order to write down a simple formula for the constant orbits on the Hochschild homology. We use the fact that subanalytic subsets admit Whitney refinements, and Whitney stratified spaces are deformation retracts of their tubular neighbourhoods with control data \cite[Section 6--7]{Mather}.

\begin{definition}\label{def:hamiltonian}
    Let $\Omega$ be a small neighbourhood of a subanalytic Legendrian $\Lambda \subseteq S^*M$ that deformation retracts onto $\Lambda$. Let $\rho: S^*M \to \bR$ be a smooth cut-off function such that
    $$\rho^{-1}(0) = \Lambda, \;\; \rho^{-1}(1) = S^*M \setminus \Omega,$$
    obtained by the smooth Urysohn's lemma. Then we write
    $$\check{T}_t, \; T_t: S^*M \to S^*M$$
    to be the flows defined by the non-negative Hamiltonian functions $\check{H}(x, \xi) = \rho(x, \xi)$ and $H(x, \xi) \equiv 1$. Abusing notations, we also write
    $$\check{T}_t, \; T_t: S^*(M \times M) \to S^*(M \times M)$$
    to be the contact Hamiltonian flows defined by $\check{H}(x, \xi, x', \xi') = \rho(x', \xi') \cdot |\xi'|/(|\xi|^2 + |\xi'|^2)^{1/2}$ and $H(x, \xi, x', \xi') \equiv |\xi'|/(|\xi|^2 + |\xi'|^2)^{1/2}$ on a neighbourhood of the diagonal $S^*_{\Delta}(M \times M)$.     

\end{definition}

From the construction, we know that under the non-negative contact flow $\check{T}_t: S^*(M \times M) \to S^*(M \times M)$, we have
$S^*_\Delta(M \times M) \cap \check{T}_\epsilon(S^*_\Delta(M \times M)) = \Delta_\Lambda.$
Under the non-negative contact flow $\hat{T}_t = T_t \circ \check{T}_{-t}: S^*(M \times M) \to S^*(M \times M)$, we have $S^*_\Delta(M \times M) \setminus \hat{T}_\epsilon(S^*_\Delta(M \times M)) = \Delta_{\Omega}$.

Then we construct the acceleration morphism on the Hochschild homology and cohomology of the category of microsheaves on $\Lambda \subseteq S^*M$. Here, we will assume that $M$ is oriented, though it should be fairly easy to write down the statements when $M$ is not orientable.

\begin{proposition}\label{prop:accelerate}
    Let $M$ be connected oriented and $\Lambda \subseteq S^*M$ be a compact subanalytic Legendrian. Then there is a commutative diagram where the horizontal morphisms are called acceleration morphisms
    \[\xymatrix{
    \mathrm{End}(\mathrm{Fib}(\check{T}_{\epsilon}1_{\Delta} \to T_\epsilon 1_{\Delta})) \ar[r] \ar[d]_{\rotatebox{90}{$\sim$}} & \mathrm{End}(\iota_{\widehat\Lambda_{\cup} \times -\widehat\Lambda_{\cup}}^* \mathrm{Fib}(\check{T}_{\epsilon}1_{\Delta} \to T_\epsilon 1_{\Delta})) \ar[d]^{\rotatebox{90}{$\sim$}} \\
    p_!\Delta^* \mathrm{Fib}(\check{T}_{\epsilon} 1_{\Delta} \to T_\epsilon 1_{\Delta})[1-n] \ar[r] & p_!\Delta^* \iota_{\widehat\Lambda_{\cup} \times -\widehat\Lambda_{\cup}}^* \mathrm{Fib}(\check{T}_{\epsilon}1_{\Delta} \to T_\epsilon 1_{\Delta}) [1-n],
    }\]
    where $\Gamma(\Lambda, 1_\Lambda) \simeq \mathrm{End}(\mathrm{Fib}(\check{T}_{\epsilon}1_{\Delta} \to T_\epsilon 1_{\Delta})) \simeq p_!\Delta^* \mathrm{Fib}(\check{T}_{\epsilon}1_{\Delta} \to T_\epsilon 1_{\Delta})[1-n]$. We denote the bottom acceleration map by
    $$\Tr_{cst}(\Id_{\msh_\Lambda(\Lambda)}) \to \Tr(\Id_{\msh_\Lambda(\Lambda)}).$$
\end{proposition}

\begin{proof}
    Consider the non-negative contact isotopy $\hat{T}_t = T_t \circ \check{T}_{-t}$ which is supported on the closure of $\Omega$. For the first isomorphism, using the perturbation lemma Proposition \ref{prop:perturbation}, Sato--Sabloff fiber sequence Theorem \ref{thm:sato-sabloff-rel} for the isotopy $\hat{T}_t: S^*(M \times M) \to S^*(M \times M)$ and Lemma \ref{lem:mhom-computation}, we know that
    \begin{align*}
    \Hom & (\mathrm{Fib}(\check{T}_{\epsilon}1_{\Delta} \to T_\epsilon 1_{\Delta}), \mathrm{Fib}(\check{T}_{\epsilon}1_{\Delta} \to T_\epsilon 1_{\Delta})) \\
    &\simeq \Hom({T}_\epsilon 1_{\Delta}, \mathrm{Fib}(\check{T}_{\epsilon}1_{\Delta} \to T_\epsilon 1_{\Delta}))[1] \simeq \Hom(1_{\Delta}, T_{-\epsilon}\mathrm{Fib}(\check{T}_{\epsilon}1_{\Delta} \to T_{\epsilon}1_{\Delta}))[1] \\
    &\simeq \Hom(1_{\Delta}, \mathrm{Fib}(\hat{T}_{-\epsilon}1_{\Delta} \to 1_{\Delta}))[1] \simeq \Gamma(\Delta_\Omega, \mhom(1_\Delta, 1_\Delta)) \simeq \Gamma(\Omega, 1_{\Omega}).
    \end{align*} 
    Since $\Lambda$ is a deformation retract of $\Omega$, we can complete the proof. For the second isomorphism, we know that for $\epsilon' < \epsilon$, a small negative pushoff ${T}_{-\epsilon'}$ of $\check{T}_\epsilon (S^*_\Delta(M \times M))$ is non-characteristic with respect to $\Delta: M \hookrightarrow M \times M$. Hence using the perturbation lemma Proposition \ref{prop:perturbation} for the isotopy $\hat{T}_t$ and the isotopy $T_{t'}: S^*M \to S^*M$, we have
    \begin{align*}
    \Hom & (\mathrm{Fib}(\check{T}_{\epsilon}1_{\Delta} \to T_\epsilon 1_{\Delta}), \mathrm{Fib}(\check{T}_{\epsilon}1_{\Delta} \to T_\epsilon 1_{\Delta})) \\
    & \simeq \Hom({T}_\epsilon1_{\Delta}, \mathrm{Fib}(\check{T}_{\epsilon}1_{\Delta} \to T_\epsilon 1_{\Delta}))[1] \simeq \Hom(1_{\Delta}, T_{-\epsilon}\mathrm{Fib}(\check{T}_{\epsilon}1_{\Delta} \to T_\epsilon 1_{\Delta}))[1] \\
    & \simeq \Hom(1_{\Delta}, {T}_{-\epsilon'} \mathrm{Fib}(\check{T}_{\epsilon}1_{\Delta} \to T_\epsilon 1_{\Delta}))[1] \simeq p_*\Delta^! ({T}_{-\epsilon'} \mathrm{Fib}(\check{T}_{\epsilon}1_{\Delta} \to T_\epsilon 1_{\Delta}))[1] \\
    & \simeq p_!\Delta^* ({T}_{-\epsilon'} \mathrm{Fib}(\check{T}_{\epsilon}1_{\Delta} \to T_\epsilon 1_{\Delta}))[1-n] \simeq p_!\Delta^* \mathrm{Fib}(\check{T}_{\epsilon}1_{\Delta} \to T_\epsilon 1_{\Delta})[1-n].
    \end{align*} 
    This shows the isomorphisms for the unwrapped terms.

    Then it suffices to show that the diagram commutes. The horizontal morphisms are induced by the positive wrapping 
    \begin{equation*}
        \mathrm{Fib}(\check{T}_{\epsilon}1_{\Delta} \to T_\epsilon 1_{\Delta})  \xrightarrow{\,\,\,} \iota_{\widehat\Lambda_{\cup} \times -\widehat\Lambda_{\cup}}^*\mathrm{Fib}(\check{T}_{\epsilon}1_{\Delta} \to T_\epsilon 1_{\Delta}) \xrightarrow{\sim} m_{\widehat\Lambda_{\cup} \times -\widehat\Lambda_{\cup}}^l m_{\widehat\Lambda_{\cup} \times -\widehat\Lambda_{\cup}} 1_{\Delta}.
    \end{equation*}
    Since the both vertical isomorphisms are both induced by adjunction and Poincar\'e duality Propositions \ref{prop:weak-rel-cy-nondeg} and \ref{prop:weak-cy-nondeg}, 
    we can conclude that the diagram commutes.
\end{proof}

Similarly, we are also able to construct the constant orbits on the Hochschild homology and cohomology of the sheaf category. When $M$ is noncompact, we consider the compact subset $M_c$ and the compact diagonal $\Delta_c$ of $M_c$, where $M_c$ is the union of all the bounded strata of $M \setminus \pi(\Lambda)$.

\begin{proposition}
    Let $M$ be a connected oriented manifold and $\Lambda \subseteq S^*M$ be a compact subanalytic Legendrian. Then there is an acceleration morphism for Hochschild homology
    $$p_!\Delta^*\check{T}_\epsilon 1_{\Delta_c}[-n] \longrightarrow p_!\Delta^* \iota_{\widehat\Lambda \times -\widehat\Lambda}^* \check{T}_\epsilon 1_{\Delta_c}[-n],$$
    and also an acceleration morphism for Hochschild cohomology
    $$\End(\check{T}_\epsilon 1_{\Delta_c}) \longrightarrow \End(\iota_{\widehat\Lambda \times -\widehat\Lambda}^* \check{T}_\epsilon 1_{\Delta_c}),$$
    where $\Gamma(\widehat{\Lambda}, 1_{\widehat\Lambda \setminus \Lambda}) \simeq p_!\Delta^*\check{T}_\epsilon 1_{\Delta_c}[-n]$ and $\Gamma(\widehat{\Lambda}, 1_{\widehat\Lambda}) \simeq \End(\check{T}_\epsilon 1_{\Delta_c})$. We denote the bottom acceleration map by
    $$\Tr_{cst}(\Id_{\Sh_{\widehat\Lambda}(M)}) \to \Tr(\Id_{\Sh_{\widehat\Lambda}(M)}).$$
\end{proposition}
\begin{proof}
    We only carry out the proof when $M$ is compact to keep the notations less complicated. For the Hochschild cohomology, the result is straightforward as $\End(\check{T}_\epsilon 1_{\Delta}) = \End(1_{\Delta})$. For the Hochschild homology, consider the fiber sequence
    $$p_!\Delta^*\check{T}_\epsilon 1_{\Delta}[-n] \to p_!\Delta^*(T_\epsilon 1_{\Delta})[-n] \to p_!\Delta^*\mathrm{Fib}(\check{T}_\epsilon 1_{\Delta} \to T_\epsilon 1_{\Delta})[1-n].$$
    Then it suffices to show that $\Gamma(\widehat\Lambda, 1_{\widehat\Lambda}) \simeq p_!\Delta^*\check{T}_\epsilon 1_{\Delta}[-n]$. First, we know that
    $$\Hom(1_{\Delta}, 1_{\Delta}) = \Gamma(\Delta, 1_\Delta) = \Gamma(\widehat\Lambda, 1_{\widehat\Lambda}).$$
    Then, notice that a small Reeb push-off $T_\epsilon(S^*_\Delta(M \times M))$ is non-characteristic with respect to $\Delta: M \hookrightarrow M \times M$. Hence using the perturbation lemma for the isotopy ${T}_t: S^*(M \times M) \to S^*(M \times M)$, 
    \begin{align*}
    \Hom(1_{\Delta}, 1_{\Delta}) \simeq \Hom(1_{\Delta}, T_\epsilon 1_{\Delta}) \simeq p_*\Delta^!T_\epsilon 1_{\Delta} \simeq p_!\Delta^*T_\epsilon 1_{\Delta}[-n].
    \end{align*}
    This shows the isomorphisms which are compatible with the isomorphisms in Proposition \ref{prop:accelerate}. This then constructs the acceleration morphism on Hochschild homology.
\end{proof}

\begin{figure}
    \centering
    \includegraphics[width=0.8\textwidth]{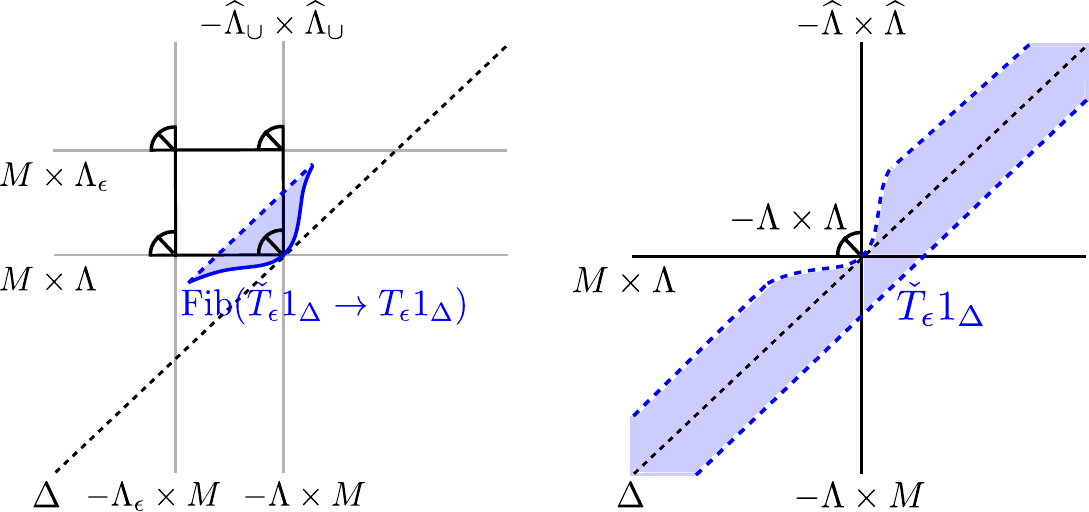}
    \caption{Blue regions are the supports of the sheaves $\check{T}_\epsilon 1_\Delta$ and $\mathrm{Fib}(\check{T}_\epsilon 1_\Delta \to T_\epsilon 1_\Delta)$ that appear in the construction of acceleration morphisms, and black arcs are the projections of $-\widehat\Lambda \times \widehat\Lambda$ and respectively $-\widehat\Lambda_\cup \times \widehat\Lambda_\cup$ onto $M \times M$.}
\end{figure}

Putting together the above propositions plus the definitions of relative Hochschild invariants, we have constructed the following acceleration morphisms on Hochschild invariants:

\begin{theorem}\label{thm:accelerate-hochschild}
    Let $M$ be connected oriented and $\Lambda \subseteq S^*M$ be a compact subanalytic Legendrian. Then for the microsheaves in $\msh_\Lambda(\Lambda)$, there is an acceleration morphism
    $$\Gamma(\Lambda, 1_\Lambda) \to \HH^*(\msh_\Lambda(\Lambda)) \xrightarrow{\sim} \HH_*(\msh_\Lambda(\Lambda))[1-n].$$
    For sheaves in $\Sh_{\widehat\Lambda}(M)$, there are acceleration morphisms
    $$\Gamma(\widehat\Lambda, 1_{\widehat\Lambda}) \to \HH^*(\Sh_{\widehat\Lambda}(M)), \; \Gamma(\widehat\Lambda, 1_{\widehat\Lambda \setminus \Lambda}) \to \HH_*(\Sh_{\widehat\Lambda}(M))[-n].$$
    These morphisms are functorial, and in particular we have acceleration morphisms
    $$\Gamma(\widehat\Lambda, 1_{\widehat\Lambda}) \to \HH^*(\msh_\Lambda(\Lambda), \Sh_{\widehat\Lambda}(M)) \xrightarrow{\sim} \HH_*(\msh_\Lambda(\Lambda), \Sh_{\widehat\Lambda}(M))[-n].$$
\end{theorem}
\begin{remark}
As explained in Remarks \ref{cohomology-v-smooth} and \ref{rem:orientation-is-id}, we know that under the acceleration morphism
$$\Gamma(\Lambda, 1_\Lambda) \to \HH^*(\msh_\Lambda(\Lambda)) \xrightarrow{\sim} \HH_*(\msh_\Lambda(\Lambda))[1-n],$$
the unit $1 \in \Gamma(\Lambda, 1_\Lambda)$ determines the weak Calabi--Yau class in
$\HH_*(\msh_\Lambda(\Lambda))$.
Moreover, under the acceleration morphism
$$\Gamma(\widehat\Lambda, 1_{\widehat\Lambda}) \to \HH^*(\msh_\Lambda(\Lambda), \Sh_{\widehat\Lambda}(M)) \xrightarrow{\sim} \HH_*(\msh_\Lambda(\Lambda), \Sh_{\widehat\Lambda}(M))[-n],$$
the unit $1 \in \Gamma(\widehat\Lambda, 1_{\widehat\Lambda})$, which is the unique lifting of $1 \in \Gamma(\Lambda, 1_\Lambda)$, determines the weak relative Calabi--Yau class in 
$\HH_*(\msh_\Lambda(\Lambda), \Sh_{\widehat\Lambda}(M))$. Thus, our goal is to show that the image of the units are $S^1$-invariant so as to obtain strong (relative) Calabi--Yau structures.
\end{remark}

\subsection{Circle action on Hochschild homology}
Now we give an explicit model for the $S^1$-action on the Hochschild homology $\HH_*(\Sh_{\widehat\Lambda}(M))$ and $\HH_*(\msh_\Lambda(\Lambda))$. Our construction is inspired by the $S^1$-action on the Hochschild homology of Tamarkin categories by Zhang \cite{ZhangS1} and the general $S^1$-action on Hochschild homology by Ayala--Francis \cite{AyalaFrancis}.

Consider the diagonal map $\Delta_n: M^n \to M^{2n}, \; (x_1, x_2, \dots, x_n) \mapsto (x_1, x_1, x_2, x_2, \dots, x_n, x_n)$, the twisted diagonal map
$\widetilde{\Delta}_n: M^n \to M^{2n}, \; (x_1, x_2, \dots, x_{n-1}, x_n) \to (x_n, x_1, x_1, \dots, x_{n-1}, x_{n-1}, x_n)$, and the projection map $p_n: M^n \to \{*\}$.

\begin{lemma}\label{lem:cyclic-homology}
Let $\Lambda \subseteq S^*M$ be a compact subanalytic Legendrian. The $S^1$-action on $\HH_*(\Sh_{\widehat\Lambda}(M)) = \Tr(\Id_{\Sh_{\widehat\Lambda}(M)})$ is defined by the cyclic object with
\begin{align*}
\Tr(\Id_{\Sh_{\widehat\Lambda}(M)}, n) = p_{n !}\widetilde{\Delta}_n^*(\iota_{-\widehat\Lambda \times \widehat\Lambda}^*1_\Delta)^{\boxtimes n}.
\end{align*}
where the face morphisms $d_i$ and degeneracy morphisms $s_i$ are induced by the identities of the unit and counit, and $t_n$ is the cyclic rotation on $(M \times M)^n$ of order $n$.
\end{lemma}
\begin{remark}
Following Remark \ref{rem:cyclic-object-map}, we can write down the morphisms in the cyclic object. Firstly, the identity $\Delta_2^*(- \boxtimes \iota_{-\widehat\Lambda \times \widehat\Lambda}^*1_\Delta \boxtimes -) \xrightarrow{\sim} \Delta_1^*(- \boxtimes -)$ induces the face morphism $d_i$ for $1 \leq i \leq n+1$:
$$d_i: p_{n+1 !}\widetilde{\Delta}_{n+1}^*(\iota_{-\widehat\Lambda \times \widehat\Lambda}^*1_\Delta)^{\boxtimes n+1} \xrightarrow{\sim} p_{n !}\widetilde{\Delta}_n^*(\iota_{-\widehat\Lambda \times \widehat\Lambda}^*1_\Delta)^{\boxtimes n}.$$
Secondly, the identity $\iota_{-\widehat\Lambda \times \widehat\Lambda}^*1_\Delta \xrightarrow{\sim} (\id \times \Delta_1 \times \id)^*(\iota_{-\widehat\Lambda \times \widehat\Lambda}^*1_\Delta)^{\boxtimes 2}$ induces the degeneration morphism $s_i$ for $1 \leq i \leq n$:
$$s_i: p_{n !}\widetilde{\Delta}_{n}^*(\iota_{-\widehat\Lambda \times \widehat\Lambda}^*1_\Delta)^{\boxtimes n} \xrightarrow{\sim} p_{n+1 !}\widetilde{\Delta}_{n+1}^*(\iota_{-\widehat\Lambda \times \widehat\Lambda}^*1_\Delta)^{\boxtimes n+1}.$$
Finally, the cyclic rotation on $(M \times M)^n$ of order $n$ induces the cyclic rotation morphism
$$t_n: p_{n !}\widetilde{\Delta}_n^*(\iota_{-\widehat\Lambda \times \widehat\Lambda}^*1_\Delta)^{\boxtimes n} \xrightarrow{\sim} p_{n !}\widetilde{\Delta}_n^*(\iota_{-\widehat\Lambda \times \widehat\Lambda}^*1_\Delta)^{\boxtimes n}.$$
\end{remark}

\begin{lemma}\label{lem:cyclic-homology-rel}
Let $\Lambda \subseteq S^*M$ be a compact subanalytic Legendrian. The $S^1$-action on $\HH_*(\msh_\Lambda(\Lambda)) = \Tr(\Id_{\msh_\Lambda(\Lambda)})$ is defined by the cyclic object with
\begin{align*}
\Tr(\Id_{\msh_\Lambda(\Lambda)}, n) = p_{n !}\widetilde{\Delta}_n^* (m_{\widehat\Lambda_{\cup} \times -\widehat\Lambda_{\cup}}^l m_{\widehat\Lambda_{\cup} \times -\widehat\Lambda_{\cup}} 1_\Delta)^{\boxtimes n},
\end{align*}
where the face morphisms $d_i$ and degeneracy morphisms $s_i$ are induced by the identities of the unit and counit, and $t_n$ is the cyclic rotation on $(M \times M)^n$ of order $n$.
\end{lemma}
\begin{remark}
Similarly, following Remark \ref{rem:cyclic-object-map}, we can write down the morphisms in the cyclic object. Firstly, the identity $\Delta_2^*(- \boxtimes m_{\widehat\Lambda_{\cup} \times -\widehat\Lambda_{\cup}}^l m_{\widehat\Lambda_{\cup} \times -\widehat\Lambda_{\cup}} 1_\Delta \boxtimes -) \xrightarrow{\sim} \Delta_1^*(- \boxtimes -)$ induces the face morphism $d_i$ for $1 \leq i \leq n+1$:
$$d_i: p_{n+1 !}\widetilde{\Delta}_{n+1}^*(m_{\widehat\Lambda_{\cup} \times -\widehat\Lambda_{\cup}}^l m_{\widehat\Lambda_{\cup} \times -\widehat\Lambda_{\cup}} 1_\Delta)^{\boxtimes n+1} \xrightarrow{\sim} p_{n !}\widetilde{\Delta}_n^*(m_{\widehat\Lambda_{\cup} \times -\widehat\Lambda_{\cup}}^l m_{\widehat\Lambda_{\cup} \times -\widehat\Lambda_{\cup}} 1_\Delta)^{\boxtimes n}.$$
Secondly, the identity $m_{\widehat\Lambda_{\cup} \times -\widehat\Lambda_{\cup}}^l m_{\widehat\Lambda_{\cup} \times -\widehat\Lambda_{\cup}} 1_\Delta \xrightarrow{\sim} (\id \times \Delta_1 \times \id)^*(m_{\widehat\Lambda_{\cup} \times -\widehat\Lambda_{\cup}}^l m_{\widehat\Lambda_{\cup} \times -\widehat\Lambda_{\cup}} 1_\Delta)^{\boxtimes 2}$ induces the degeneration morphism $s_i$ for $1 \leq i \leq n$:
$$s_i: p_{n !}\widetilde{\Delta}_{n}^*(m_{\widehat\Lambda_{\cup} \times -\widehat\Lambda_{\cup}}^l m_{\widehat\Lambda_{\cup} \times -\widehat\Lambda_{\cup}} 1_\Delta)^{\boxtimes n} \xrightarrow{\sim} p_{n+1 !}\widetilde{\Delta}_{n+1}^*(m_{\widehat\Lambda_{\cup} \times -\widehat\Lambda_{\cup}}^l m_{\widehat\Lambda_{\cup} \times -\widehat\Lambda_{\cup}} 1_\Delta)^{\boxtimes n+1}.$$
Finally, the cyclic rotation on $(M \times M)^n$ of order $n$ induces the cyclic rotation morphism
$$t_n: p_{n !}\widetilde{\Delta}_n^*(m_{\widehat\Lambda_{\cup} \times -\widehat\Lambda_{\cup}}^l m_{\widehat\Lambda_{\cup} \times -\widehat\Lambda_{\cup}} 1_\Delta)^{\boxtimes n} \xrightarrow{\sim} p_{n !}\widetilde{\Delta}_n^*(m_{\widehat\Lambda_{\cup} \times -\widehat\Lambda_{\cup}}^l m_{\widehat\Lambda_{\cup} \times -\widehat\Lambda_{\cup}} 1_\Delta)^{\boxtimes n}.$$
\end{remark}

\subsection{Circle action on the constant orbits}
Recall the acceleration morphism that we have constructed in Proposition \ref{prop:accelerate}
\begin{align*}
p_!\Delta^* \mathrm{Fib}(\check{T}_{\epsilon}1_\Delta \to T_\epsilon 1_\Delta) & \xrightarrow{\,\,\,} p_!\Delta^* \iota_{\widehat\Lambda_{\cup,\epsilon} \times -\widehat\Lambda_{\cup, \epsilon}}^* \mathrm{Fib}(\check{T}_{\epsilon}1_\Delta \to T_\epsilon 1_\Delta).
\end{align*}
We now explain the standard $S^1$-action on the left hand side. 

Recall the non-negative Hamiltonian flows $\check{T}_t, T_t: S^*(M \times M) \to S^*(M \times M)$ defined in Definition \ref{def:hamiltonian}. For $\sA = \Sh_{\widehat\Lambda}(M)$, using the compact subset $M_c$ and the compact diagonal $\Delta_c$ of $M_c$, we will consider a cyclic object such that
\begin{align*}
\Tr_{cst}(\Id_{\Sh_{\widehat\Lambda}(M)}, n) = p_{n !}\widetilde{\Delta}_n^*(\check{T}_{\epsilon/n}1_{\Delta_c})^{\boxtimes n}.
\end{align*}
Similarly, for $\sB = \msh_\Lambda(\Lambda)$, we will consider a cyclic object such that
\begin{align*}
\Tr_{cst}(\Id_{\msh_\Lambda(\Lambda)}, n) = p_{n !}\widetilde{\Delta}_n^*\mathrm{Fib}(\check{T}_{\epsilon/n}1_\Delta \to T_{\epsilon/n} 1_\Delta)^{\boxtimes n}.
\end{align*}
Here, when defining each term in the cyclic object, we intentionally take the Hamiltonian push-off for time $\epsilon/n$ to make sure that the microlocal behaviour of the $n$-th term in the cyclic object is captured by $n$-broken Reeb orbits, each of which has length at most $\epsilon/n$. Hence the microlocal behaviour of all terms in the cyclic object is captured by the Reeb orbits of total length at most $\epsilon$.

We check that the choice of Hamiltonian push-offs that we make are consistent so that the cyclic object is well defined.

\begin{proposition}\label{prop:cyclic-well-define}
    Let $\Lambda \subseteq S^*M$ be a compact subanalytic Legendrian. Then there is a cyclic object $\Tr_{cst}(\Id_{\msh_\Lambda(\Lambda)}, *)$ given by
    $$\Tr_{cst}(\Id_{\msh_\Lambda(\Lambda)}, n) = p_{n !}\widetilde{\Delta}_n^*\mathrm{Fib}(\check{T}_{\epsilon/n}1_\Delta \to T_{\epsilon/n} 1_\Delta)^{\boxtimes n},$$
    where the face morphisms $d_i$ are induced by $\Delta_2^*(- \boxtimes \mathrm{Fib}(\check{T}_{\epsilon/n}1_\Delta \to T_{\epsilon/n} 1_\Delta) \boxtimes -) \rightarrow \Delta_1^*(- \boxtimes -)$, and $t_n$ are induced by the cyclic rotation on $(M \times M)^n$ of order $n$.
\end{proposition}

\begin{proposition}\label{prop:cyclic-well-define-rel}
    Let $\Lambda \subseteq S^*M$ be a compact subanalytic Legendrian. Then there is a cyclic object $\Tr_{cst}(\Id_{\Sh_{\widehat\Lambda}(M)}, *)$ given by
    $$\Tr_{cst}(\Id_{\Sh_{\widehat\Lambda}(M)}, n) = p_{n !}\widetilde{\Delta}_n^*(\check{T}_{\epsilon/n}1_{\Delta_c})^{\boxtimes n},$$
    where the face morphisms $d_i$ are induced by $\Delta_2^*(- \boxtimes \check{T}_{\epsilon/n}1_{\Delta_c} \boxtimes -) \rightarrow \Delta_1^*(- \boxtimes -)$, and $t_n$ are induced by the cyclic rotation on $(M \times M)^n$ of order $n$.
\end{proposition}

\begin{remark}
    We conjecture that $\Tr_{cst}(\Id_{\Sh_{\widehat\Lambda}(M)})$ and $\Tr_{cst}(\Id_{\msh_\Lambda(\Lambda)})$ are actually the trace of some other smooth categories that admit functors into $\Sh_{\widehat\Lambda}(M)$ and $\msh_\Lambda(\Lambda)$. However, this will be the topics for future studies.
\end{remark}

Since the argument of the above propositions are the same, we will only write down the proof for the more complicated case, namely Proposition \ref{prop:cyclic-well-define}.

In order to show that these are cyclic objects, we need the following singular support estimation that shows the relation between the singular support of the twisted diagonal and broken Reeb trajectories. See for example \cite[Lemma 4.9]{Chiu} and \cite[Proposition 2.18]{Zhang}.

\begin{lemma}\label{lem:ss-broken-loop}
    Let $\varphi_t^j: S^*M \to S^*M$ be contact isotopies and $K_t^j \in \Sh(M \times M)$ be the sheaf quantizations of the contact isotopies such that $\msif(K_t^j) = \mathrm{Graph}(\varphi_t^j)$ for $1 \leq j \leq n$. Then
    \begin{align*}
        S^*_{\tilde{\Delta}_n} & (M \times M)^n \cap \msif(K_t^1 \boxtimes K_t^2 \boxtimes \dots \boxtimes K_t^n) \\
        = &\, \{(x_n, -\xi_n, x_1, \xi_1; x_1, -\xi_1, x_2, \xi_2; \dots; x_{n-1}, -\xi_{n-1}; x_n, \xi_n) \\
        & \mid (x_1, \xi_1) = \varphi_t^1(x_n, \xi_n), (x_{i+1}, \xi_{i+1}) = \varphi_t^{i+1}(x_i, \xi_i), 1 \leq i \leq n-1 \}.
    \end{align*}
\end{lemma}

We also need the following simple geometric lemma on small closed orbits of the Reeb flow, which will enable us to show that our definition above does not introduce any extra correction coming from the non-constant closed orbits of the Reeb flow.

\begin{lemma}\label{lem:non-const-orbit}
    Consider a complete metric on $S^*M$. For any $\delta > 0$, there exists a uniform lower bound $\epsilon > 0$ such that for any Hamiltonian flow $\varphi_t: S^*M \to S^*M$ whose contact vector field has norm at least $\delta > 0$, the closed orbits have period at least $\epsilon$.
\end{lemma}

Now we prove that $\Tr_{cst}(\Id_{{\msh_\Lambda(\Lambda)}})$ defines a cyclic object. First, we use the following lemma to show that the face morphisms $d_i,\,1 \leq i \leq n$, can be well defined.

\begin{lemma}\label{lem:face-map-define}
    For $\Lambda \subseteq S^*M$ be a compact subanalytic Legendrian, when $\epsilon > 0$ is sufficiently small, there are isomorphisms
    $$p_{n !} \widetilde{\Delta}_{n}^* (1_\Delta \boxtimes \mathrm{Fib}(\check{T}_{\epsilon/(n-1)} 1_\Delta \to T_{\epsilon/(n-1)} 1_\Delta)^{\boxtimes n-1}) \simeq p_{n !} \widetilde{\Delta}_{n}^* (1_\Delta \boxtimes \mathrm{Fib}(\check{T}_{\epsilon/n} 1_\Delta \to T_{\epsilon/n} 1_\Delta)^{\boxtimes n-1}).$$
\end{lemma}
\begin{proof}
    Using adjunction, we know that it suffices to prove the following isomorphisms:
    $$\Hom(1_\Delta \boxtimes \mathrm{Fib}(\check{T}_{\epsilon/(n-1)} 1_\Delta \to T_{\epsilon/(n-1)} 1_\Delta)^{\boxtimes n-1}, 1_{\tilde{\Delta}_n}) \simeq \Hom(1_\Delta \boxtimes \mathrm{Fib}(\check{T}_{\epsilon/n} 1_\Delta \to T_{\epsilon/n} 1_\Delta)^{\boxtimes n-1}, 1_{\tilde{\Delta}_n}).$$
    Consider the flow $\check{T}_{t/n}$ and $T_{t/n}$ for $0 < t \leq \epsilon$, and consider the intersection of the singular supports 
    $$S^*_{\tilde{\Delta}_n} (M \times M)^n \cap \msif(1_\Delta \boxtimes \mathrm{Fib}(\check{T}_{t/(n-1)} 1_\Delta \to T_{t/(n-1)} 1_\Delta)^{\boxtimes n-1}).$$
    We know that they are in bijection with the set of $n$-broken Reeb orbits
    $$(x_2, \xi_2) = \check{T}_{t/(n-1)} (x_1, \xi_1), \; (x_{i+1}, \xi_{i+1}) = T^*_{t/(n-1)}(x_i, \xi_i), \; (x_1, \xi_1) = (x_n, \xi_n),$$
    where $T^*_{t/(n-1)}$ means either $\check{T}_{t/(n-1)}$ or $T_{t/(n-1)}$. When $\epsilon > 0$ is sufficiently small, we know that the closed orbits of length at most $\epsilon$ for the concatenated non-negative Hamiltonian flow are all constant orbits by the concatenation of $\check{T}_{t/(n-1)}, \check{T}_{t/(n-1)}, \dots, \check{T}_{t/(n-1)}$, starting and ending on $\Lambda$. Therefore, for $0 < t \leq \epsilon$,
    $$S^*_{\tilde{\Delta}_n}(M \times M)^n \cap \msif(1_\Delta \boxtimes \mathrm{Fib}(\check{T}_{t/(n-1)} 1_\Delta \to T_{t/(n-1)} 1_\Delta)^{\boxtimes n-1}) = \widetilde{\Delta}_{\Lambda,n}.$$
    Take a sufficiently small Reeb push-off, we can also show that for any $0 < t \leq \epsilon$, the intersection is empty:
    $$T_{\epsilon} S^*_{\tilde{\Delta}_n}(M \times M)^n \cap \msif(1_\Delta \boxtimes \mathrm{Fib}(\check{T}_{t/(n-1)} 1_\Delta \to T_{t/(n-1)} 1_\Delta)^{\boxtimes n-1}) = \varnothing.$$
    Then the isomorphism follows from the positively gapped non-characteristic deformation Proposition \ref{prop:gapped-fully-faithful}.
\end{proof}

Then, we need the following lemmas to show that the face morphisms $d_i, \, 1 \leq i \leq n+1$, induce isomorphisms between different terms in the construction. By the relations in the cyclic category, this automatically implies that the degeneration morphisms $s_i, \, 1 \leq i \leq n$, also induce isomorphisms.

\begin{lemma}\label{lem:cyclic-define-isom}
    For $\Lambda \subseteq S^*M$ be a compact subanalytic Legendrian, there are isomorphisms
    $$p_{n !} \widetilde{\Delta}_{n}^* (1_\Delta \boxtimes \mathrm{Fib}(\check{T}_{\epsilon/n} 1_\Delta \to T_{\epsilon/n} 1_\Delta)^{\boxtimes n-1}) \simeq p_{n !}\widetilde{\Delta}_{n}^*(\check{T}_{\epsilon/n} 1_\Delta \boxtimes \mathrm{Fib}(\check{T}_{\epsilon/n} 1_\Delta \to T_{\epsilon/n} 1_\Delta)^{\boxtimes n-1}).$$
\end{lemma}
\begin{proof}
    Using adjunction, we know that it suffices to prove the following isomorphisms:
    $$\Hom(1_\Delta \boxtimes \mathrm{Fib}(\check{T}_{\epsilon/n} 1_\Delta \to T_{\epsilon/n} 1_\Delta)^{\boxtimes n-1}, 1_{\tilde{\Delta}_n}) \simeq \Hom(\check{T}_{\epsilon/n} 1_\Delta \boxtimes \mathrm{Fib}(\check{T}_{\epsilon/n} 1_\Delta \to T_{\epsilon/n} 1_\Delta)^{\boxtimes n-1}, 1_{\tilde{\Delta}_n}).$$
    Consider the flow $\check{T}_{t/n}$ for $0 \leq t \leq \epsilon$, and consider the intersection of the singular supports 
    $$S^*_{\tilde{\Delta}_n} (M \times M)^n \cap \msif(\check{T}_{t/n} 1_\Delta \boxtimes \mathrm{Fib}(\check{T}_{\epsilon/n} 1_\Delta \to T_{\epsilon/n} 1_\Delta)^{\boxtimes n-1}).$$
    We know that they are in bijection with the set of $n$-broken Reeb orbits
    $$(x_2, \xi_2) = T^*_{\epsilon/n} (x_1, \xi_1), \; (x_{i+1}, \xi_{i+1}) = T^*_{\epsilon/n}(x_i, \xi_i), \; (x_1, \xi_1) = \check{T}_{t/n}(x_n, \xi_n),$$
    where $T^*_{\epsilon/n}$ means either $\check{T}_{\epsilon/n}$ or $T_{\epsilon/n}$. When $\epsilon > 0$ is sufficiently small, we know that the closed orbits of length at most $\epsilon$ for the concatenated non-negative Hamiltonian flow are all constant orbits by the concatenation of $\check{T}_{t/n}, \check{T}_{\epsilon/n}, \dots, \check{T}_{\epsilon/n}$, starting and ending on $\Lambda$. Therefore, for $0 \leq t \leq \epsilon$,
    $$S^*_{\tilde{\Delta}_n}(M \times M)^n \cap \msif(\check{T}_{t/n} 1_\Delta \boxtimes \mathrm{Fib}(\check{T}_{\epsilon/n} 1_\Delta \to T_{\epsilon/n} 1_\Delta)^{\boxtimes n-1}) = \delta_{n*}\Lambda,$$
    where $\delta_n: S^*M \to S^*((M \times M) \times \dots \times (M \times M)), \, (x, \xi) \mapsto (x, \xi, x, \xi, \dots, x, \xi, x, \xi)$. Take a sufficiently small Reeb push-off, we can also show that for any $0 < t \leq \epsilon$, the intersection is empty:
    $$T_{\epsilon} S^*_{\tilde{\Delta}_n}(M \times M)^n \cap \msif(\check{T}_{t/n} 1_\Delta \boxtimes \mathrm{Fib}(\check{T}_{\epsilon/n} 1_\Delta \to T_{\epsilon/n} 1_\Delta)^{\boxtimes n-1}) = \varnothing.$$
    Then the isomorphism follows from positively gapped non-characteristic defomormation Proposition \ref{prop:gapped-fully-faithful}.
\end{proof}

\begin{lemma}\label{lem:cyclic-define-zero}
    For $\Lambda \subseteq S^*M$ be a compact subanalytic Legendrian, there are isomorphisms
    $$0 \simeq p_{n !}\widetilde{\Delta}_{n}^*({T}_{\epsilon/n} 1_\Delta \boxtimes \mathrm{Fib}(\check{T}_{\epsilon/n} 1_\Delta \to T_{\epsilon/n} 1_\Delta)^{\boxtimes n-1}).$$
\end{lemma}
\begin{proof}
    Using adjunction, we know that it suffices to prove the following isomorphisms:
    $$0 \simeq \Hom({T}_{\epsilon/n} 1_\Delta \boxtimes \mathrm{Fib}(\check{T}_{\epsilon/n} 1_\Delta \to T_{\epsilon/n} 1_\Delta)^{\boxtimes n-1}, 1_{\tilde{\Delta}_n}).$$
    Write $\hat{T}_{t/n} = T_{t/n} \circ \check{T}_{(\epsilon - t)/n}$ for $0 \leq t \leq \epsilon$. We will deform the sheaf $\mathrm{Fib}(\check{T}_{\epsilon/n} 1_\Delta \to T_{\epsilon/n} 1_\Delta)$ using this Hamiltonian flow. Consider the intersection of the singular supports 
    $$S^*_{\tilde{\Delta}_n} (M \times M)^n \cap \msif(T_{\epsilon/n} 1_\Delta \boxtimes \mathrm{Fib}(\check{T}_{\epsilon/n} 1_\Delta \to T_{\epsilon/n} 1_\Delta)^{\boxtimes n-2} \boxtimes \mathrm{Fib}(\hat{T}_t \circ \check{T}_{\epsilon/n} 1_\Delta \to T_{\epsilon/n} 1_\Delta)).$$
    We know that they are in bijection with the set of $n$-broken Reeb orbits
    $$(x_{i+1}, \xi_{i+1}) = T^*_{\epsilon/n}(x_i, \xi_i), \; (x_n, \xi_n) = \hat{T}_{t/n} \circ T^*_{\epsilon/n}(x_{n-1}, \xi_{n-1}), \; (x_1, \xi_1) = T_{\epsilon/n} (x_n, \xi_n),$$
    where $T^*_{\epsilon/n}$ means $\check{T}_{\epsilon/n}$ or $T_{\epsilon/n}$. When $\epsilon > 0$ is sufficiently small, we know that there exist no closed orbits of length at most $\epsilon$ for the concatenated positive Hamiltonian flow since the Hamiltonian $T_{\epsilon/n}$ is strictly positive and all the other Hamiltonians $T^*_{\epsilon/n}$ are non-negative. Therefore, for $0 \leq t \leq \epsilon$, we have
    $$S^*_{\tilde{\Delta}_n}(M \times M)^n \cap \msif(T_{\epsilon/n} 1_\Delta \boxtimes \mathrm{Fib}(\check{T}_{\epsilon/n} 1_\Delta \to T_{\epsilon/n} 1_\Delta)^{\boxtimes n-2} \boxtimes \mathrm{Fib}(\hat{T}_t \circ \check{T}_{\epsilon/n} 1_\Delta \to T_{\epsilon/n} 1_\Delta)) = \varnothing.$$
    Then using the gapped non-characteristic deformation Theorem \ref{thm:full-faithful-nearby}, we can conclude that
    \begin{align*}
        \Hom & ({T}_{\epsilon/n} 1_\Delta \boxtimes \mathrm{Fib}(\check{T}_{\epsilon/n} 1_\Delta \to T_{\epsilon/n} 1_\Delta)^{\boxtimes n-1}, 1_{\tilde{\Delta}_n}) \\
        & \simeq \Hom(T_{\epsilon/n} 1_\Delta \boxtimes \mathrm{Fib}(\check{T}_{\epsilon/n} 1_\Delta \to T_{\epsilon/n} 1_\Delta)^{\boxtimes n-2} \boxtimes \mathrm{Fib}(\hat{T}_{t/n} \circ \check{T}_{\epsilon/n} 1_\Delta \to T_{\epsilon/n} 1_\Delta), 1_{\tilde{\Delta}_n}) \\
        & \simeq \Hom(T_{\epsilon/n} 1_\Delta \boxtimes \mathrm{Fib}(\check{T}_{\epsilon/n} 1_\Delta \to T_{\epsilon/n} 1_\Delta)^{\boxtimes n-2} \boxtimes \mathrm{Fib}(\hat{T}_{\epsilon/n} \circ \check{T}_{\epsilon/n} 1_\Delta \to T_{\epsilon/n} 1_\Delta), 1_{\tilde{\Delta}_n}) \\
        & \simeq \Hom(T_{\epsilon/n} 1_\Delta \boxtimes \mathrm{Fib}(\check{T}_{\epsilon/n} 1_\Delta \to T_{\epsilon/n} 1_\Delta)^{\boxtimes n-2} \boxtimes \mathrm{Fib}(T_{\epsilon/n} 1_\Delta \to T_{\epsilon/n} 1_\Delta), 1_{\tilde{\Delta}_n}) \simeq 0.
    \end{align*}
    This completes the proof of the lemma.
\end{proof}

Using the above lemmas, we can finally show that the construction does induce a cyclic object $\Tr_{cst}(\Id_{\msh_\Lambda(\Lambda)}, *)$, characterizing an $S^1$-action on $\Tr_{cst}(\Id_{\msh_\Lambda(\Lambda)})$.

\begin{proof}[Proof of Proposition \ref{prop:cyclic-well-define}]
    First, the relations between the face morphisms, degeneration morphisms and cyclic permutation morphisms immediately follow. Next, notice that we can define the natural isomorphism induced by the cyclic permutation $t_n: (M \times M)^n \to (M \times M)^n$ for $n \in \bN$:
    $$t_n: p_{n !}\widetilde{\Delta}_n^*\mathrm{Fib}(\check{T}_{\epsilon/n}1_\Delta \to T_{\epsilon/n} 1_\Delta)^{\boxtimes n} \xrightarrow{\sim} p_{n !}\widetilde{\Delta}_n^*\mathrm{Fib}(\check{T}_{\epsilon/n}1_\Delta \to T_{\epsilon/n} 1_\Delta)^{\boxtimes n}.$$

    Then we construct the face morphisms $d_i$ for $i \in n$. Using the projection formula, the base change formula and Lemma \ref{lem:face-map-define}, we get
    \begin{align*}
        p_{n-1 !} \widetilde{\Delta}_{n-1}^* &\mathrm{Fib}(\check{T}_{\epsilon/(n-1)} 1_\Delta \to T_{\epsilon/(n-1)} 1_\Delta)^{\boxtimes n-1} \\
        & \simeq p_{n !} \widetilde{\Delta}_{n}^* (1_\Delta \boxtimes \mathrm{Fib}(\check{T}_{\epsilon/(n-1)} 1_\Delta \to T_{\epsilon/(n-1)} 1_\Delta)^{\boxtimes n-1}) \\ 
        & \simeq p_{n !} \widetilde{\Delta}_{n}^* (1_\Delta \boxtimes \mathrm{Fib}(\check{T}_{\epsilon/n} 1_\Delta \to T_{\epsilon/n} 1_\Delta)^{\boxtimes n-1}).
    \end{align*}
    Then the face morphism is induced by the following natural morphism
    $$d_i: p_{n !} \widetilde{\Delta}_{n}^* (1_\Delta \boxtimes \mathrm{Fib}(\check{T}_{\epsilon/n} 1_\Delta \to T_{\epsilon/n} 1_\Delta)^{\boxtimes n-1}) \xrightarrow{\sim} p_{n !} \widetilde{\Delta}_{n}^* \mathrm{Fib}(\check{T}_{\epsilon/n} 1_\Delta \to T_{\epsilon/n} 1_\Delta)^{\boxtimes n}.$$
    This is an isomorphism by Lemma \ref{lem:cyclic-define-isom} and \ref{lem:cyclic-define-zero}. By the relation in the cyclic category, we know that the degeneration morphisms have to be isomorphisms as well. Hence we obtain a cyclic object.
\end{proof}

Recall that the weak Calabi--Yau structure is induced by a morphism
$$1_\cV[n-1] \to \Tr_{cst}(\Id_{\msh_\Lambda(\Lambda)}) \to \Tr(\Id_{\msh_\Lambda(\Lambda)}).$$
In order to show that the weak Calabi--Yau structure is $S^1$-invariant, we will study the $S^1$-action on $\Tr_{cst}(\Id_{\msh_\Lambda(\Lambda)})$ and show that it is the trivial action.

\begin{proposition}\label{prop:trivial-action}
    For $\Lambda \subseteq S^*M$ be a compact subanalytic Legendrian, the standard $S^1$-action on 
    $$\Tr_{cst}(\Id_{\msh_\Lambda(\Lambda)}) = p_!\Delta^* \mathrm{Fib}(\check{T}_{\epsilon}1_\Delta \to T_\epsilon 1_\Delta)$$
    is the trivial $S^1$-action.
\end{proposition}

    Using adjunction, we can consider the cyclic permutation on the Hom space $\Hom(\mathrm{Fib}(\check{T}_{\epsilon/n}1_\Delta \to T_{\epsilon/n} 1_\Delta)^{\boxtimes n}, 1_{\tilde{\Delta}_n})$. We will show that the support of $\mhom(\mathrm{Fib}(\check{T}_{\epsilon/n}1_\Delta \to T_{\epsilon/n} 1_\Delta)^{\boxtimes n}, 1_{\tilde{\Delta}_n})$ is contained in the fixed point set of the cyclic permutation, where the action is trivial.

\begin{lemma}\label{lem:support-diagonal}
    For $\Lambda \subseteq S^*M$ be a compact subanalytic Legendrian, $\delta_n: T^*M \to T^*(M \times M) \times \dots \times T^*(M \times M), \, (x, \xi) \mapsto (x, \xi, x, -\xi, \dots, x, \xi, x, -\xi)$, the restriction functor induces an isomorphism
    \begin{align*}
        \Gamma (S^*(M \times M)^n &, \mhom(\mathrm{Fib}(\check{T}_{\epsilon/n}1_\Delta \to T_{\epsilon/n} 1_\Delta)^{\boxtimes n}, 1_{\tilde{\Delta}_n})) \\
        \simeq \Gamma(\delta_{n}(S^*M) &, \mhom(\mathrm{Fib}(\check{T}_{\epsilon/n}1_\Delta \to T_{\epsilon/n} 1_\Delta)^{\boxtimes n}, 1_{\tilde{\Delta}_n})).
    \end{align*}
\end{lemma}
\begin{proof}
    Since we know that $\mathrm{supp}(\mhom(\mathrm{Fib}(\check{T}_{\epsilon/n}1_\Delta \to T_{\epsilon/n} 1_\Delta)^{\boxtimes n}, 1_{\tilde{\Delta}_n}))$ is contained in the intersection of the singular supports for the sheaves, it suffices to show that the intersection
    $$S^*_{\tilde{\Delta}_n}(M \times M)^n \cap \msif(\mathrm{Fib}(\check{T}_{\epsilon/n}1_\Delta \to T_{\epsilon/n} 1_\Delta)^{\boxtimes n}) \subseteq \delta_{n}(S^*M).$$
    Recall that the intersection corresponds in bijection to the set of $n$-broken Reeb orbits of the form
    $$(x_2, \xi_2) = T^*_{\epsilon/n} (x_1, \xi_1), \; (x_{i+1}, \xi_{i+1}) = T^*_{\epsilon/n}(x_i, \xi_i), \; (x_1, \xi_1) = T^*_{\epsilon/n}(x_n, \xi_n),$$
    where $T^*_{\epsilon/n}$ means either $\check{T}_{\epsilon/n}$ or $T_{\epsilon/n}$. For $\epsilon > 0$ sufficiently small, the closed orbits of length at most $\epsilon$ for the concatenated non-negative Hamiltonian flow are all constant orbits by the concatenation of $\check{T}_{\epsilon/n}, \dots, \check{T}_{\epsilon/n}$. Therefore,
    \begin{equation*}
        S^*_{\tilde{\Delta}_n}(M \times M)^n \cap \msif(\mathrm{Fib}(\check{T}_{\epsilon/n} 1_\Delta \to T_{\epsilon/n} 1_\Delta)^{\boxtimes n}) = \delta_{n}(\Lambda) \subseteq \delta_{n}(S^*M). \qedhere
    \end{equation*}
\end{proof}

    Then we show that the Hom of microsheaves away from the zero section completely determines the Hom of the sheaves.

\begin{lemma}\label{lem:trivial-action-zerosection}
    For $\Lambda \subseteq S^*M$ be a compact subanalytic Legendrian, there is an isomorphism
    \begin{align*}
    \Hom & (\mathrm{Fib}(\check{T}_{\epsilon/n}1_\Delta \to T_{\epsilon/n} 1_\Delta)^{\boxtimes n}, 1_{\tilde{\Delta}_n})) \\
    \simeq \Gamma & (S^*(M \times M)^n,  \mhom(\mathrm{Fib}(\check{T}_{\epsilon/n}1_\Delta \to T_{\epsilon/n} 1_\Delta)^{\boxtimes n}, 1_{\tilde{\Delta}_n})).
    \end{align*}
\end{lemma}
\begin{proof}
    Using the Sato--Sabloff fiber sequence, it suffices to show that for $\epsilon' > 0$ sufficiently small, we have
    $$\Hom(T_{\epsilon'/n}\mathrm{Fib}(\check{T}_{\epsilon/n}1_\Delta \to T_{\epsilon/n} 1_\Delta)^{\boxtimes n}, 1_{\tilde{\Delta}_n})) \simeq 0.$$
    Similar to Lemma \ref{lem:cyclic-define-zero}, we write $\hat{T}_{t/n} = T_{t/n} \circ \check{T}_{(\epsilon - t)/n}$ for $0 \leq t \leq \epsilon$. We will deform the sheaf $\mathrm{Fib}(\check{T}_{\epsilon/n} 1_\Delta \to T_{\epsilon/n} 1_\Delta)$ using this Hamiltonian flow. Consider the intersection of the singular supports 
    $$S^*_{\tilde{\Delta}_n} (M \times M)^n \cap \msif(T_{\epsilon'/n} (\mathrm{Fib}(\check{T}_{\epsilon/n} 1_\Delta \to T_{\epsilon/n} 1_\Delta)^{\boxtimes n-1} \boxtimes \mathrm{Fib}(\hat{T}_{t/n} \circ \check{T}_{\epsilon/n} 1_\Delta \to T_{\epsilon/n} 1_\Delta))).$$
    We know that they are in bijection with the set of $n$-broken Reeb orbits
    $$(x_2, \xi_2) = T_{\epsilon'/n} \circ T^*_{\epsilon/n}(x_1, \xi_1), \; (x_{i+1}, \xi_{i+1}) = T_{\epsilon'/n} \circ T^*_{\epsilon/n}(x_i, \xi_i), \; (x_1, \xi_1) = T_{\epsilon'/n} \circ \hat{T}_{t/n} \circ T^*_{\epsilon/n}(x_n, \xi_n),$$
    where $T^*_{\epsilon/n}$ means either $\check{T}_{\epsilon/n}$ or $T_{\epsilon/n}$. When $\epsilon > 0$ is sufficiently small, we know that there exist no closed orbits of length at most $\epsilon + \epsilon'$ for the concatenated positive Hamiltonian flow since the Hamiltonian $T_{\epsilon'/n} \circ T^*_{\epsilon/n}$ are all strictly positive. Therefore, for $0 \leq t \leq \epsilon$, we have
    $$S^*_{\tilde{\Delta}_n}(M \times M)^n \cap \msif(T_{\epsilon'/n} (\mathrm{Fib}(\check{T}_{\epsilon/n} 1_\Delta \to T_{\epsilon/n} 1_\Delta)^{\boxtimes n-2} \boxtimes \mathrm{Fib}(\hat{T}_{t/n} \circ \check{T}_{\epsilon/n} 1_\Delta \to T_{\epsilon/n} 1_\Delta)) = \varnothing.$$
    Then using the gapped non-characteristic deformation Theorem \ref{thm:full-faithful-nearby}, we can conclude that
    \begin{align*}
        \Hom & (T_{\epsilon'/n}\mathrm{Fib}(\check{T}_{\epsilon/n}1_\Delta \to T_{\epsilon/n} 1_\Delta)^{\boxtimes n}, 1_{\tilde{\Delta}_n})) \\
        & \simeq \Hom  (T_{\epsilon'/n} (\mathrm{Fib}(\check{T}_{\epsilon/n}1_\Delta \to T_{\epsilon/n} 1_\Delta)^{\boxtimes n-1} \boxtimes \mathrm{Fib}(\hat{T}_{t/n} \circ \check{T}_{\epsilon/n} 1_\Delta \to T_{\epsilon/n} 1_\Delta)), 1_{\tilde{\Delta}_n})) \\
        & \simeq \Hom  (T_{\epsilon'/n} (\mathrm{Fib}(\check{T}_{\epsilon/n}1_\Delta \to T_{\epsilon/n} 1_\Delta)^{\boxtimes n-1} \boxtimes \mathrm{Fib}(T_{\epsilon/n} 1_\Delta \to T_{\epsilon/n} 1_\Delta)), 1_{\tilde{\Delta}_n}))\simeq 0.
    \end{align*}
    This completes the proof of the lemma.
\end{proof}

\begin{proof}[Proof of Proposition \ref{prop:trivial-action}]
    We prove that the isomorphism $t_n$ given by cyclic permutations
    $$t_n: p_{n !}\widetilde{\Delta}_n^*\mathrm{Fib}(\check{T}_{\epsilon/n}1_\Delta \to T_{\epsilon/n} 1_\Delta)^{\boxtimes n} \xrightarrow{\sim} p_{n !}\widetilde{\Delta}_n^*\mathrm{Fib}(\check{T}_{\epsilon/n}1_\Delta \to T_{\epsilon/n} 1_\Delta)^{\boxtimes n}$$
    is equivalent to the identity morphism. Then it will follow that the $S^1$-action on the cyclic object is trivial. By adjunction, it suffices to show the isomorphism is equivalent to the identity:
    $$t_n^\vee: \Hom(\mathrm{Fib}(\check{T}_{\epsilon/n}1_\Delta \to T_{\epsilon/n} 1_\Delta)^{\boxtimes n}, 1_{\tilde{\Delta}_n}) \xrightarrow{\sim} \Hom(\mathrm{Fib}(\check{T}_{\epsilon/n}1_\Delta \to T_{\epsilon/n} 1_\Delta)^{\boxtimes n}, 1_{\tilde{\Delta}_n}).$$

    Write $\delta_n: T^*M \to T^*(M \times M) \times \dots \times T^*(M \times M), \, (x, \xi) \mapsto (x, \xi, x, -\xi, \dots, x, \xi, x, -\xi)$. Consider the following composition of microlocalization and restrictions which are compatible with the cyclic rotation
    \begin{align*}
    \Hom &(\mathrm{Fib}(\check{T}_{\epsilon/n}1_\Delta \to T_{\epsilon/n} 1_\Delta)^{\boxtimes n}, 1_{\tilde{\Delta}_n}) \\ 
    & \simeq \Gamma(S^*(M \times M)^n, \mhom(\mathrm{Fib}(\check{T}_{\epsilon/n}1_\Delta \to T_{\epsilon/n} 1_\Delta)^{\boxtimes n}, 1_{\tilde{\Delta}_n}) \\
    & \simeq \Gamma(\delta_n(S^*M), \mhom(\mathrm{Fib}(\check{T}_{\epsilon/n}1_\Delta \to T_{\epsilon/n} 1_\Delta)^{\boxtimes n}, 1_{\tilde{\Delta}_n})).
    \end{align*}
    By Lemma \ref{lem:trivial-action-zerosection}, the first morphism by microlocalization is an isomorphism. By Lemma \ref{lem:support-diagonal}, the second morphism by restriction is also an isomorphism.
    
    Under the restriction to the image of $\delta_n$, the smooth map $t_n$ becomes the identity map, and the isomorphism $t_n^\vee$ becomes the identity morphism. Hence the microlocalization of this isomorphism will also be the identity. This thus completes the proof.
\end{proof}

Similar to Proposition \ref{prop:trivial-action}, since the null homotopy of the weak relative Calabi--Yau structure factors through the morphism
$$0 \to \Tr_{cst}(\Id_{\Sh_{\widehat\Lambda}(M)}) \to \Tr(\Id_{\Sh_{\widehat\Lambda}(M)}),$$
we will study that the $S^1$-action on $\Tr_{cst}(\Id_{\Sh_{\widehat\Lambda}(M)})$ and show that it is the trivial action. Recall that when $M$ is a noncompact manifold, we need to consider the compact subset $M_c$ and the compact diagonal $\Delta_c$ of $M_c$, where $M_c$ is the union of the bounded strata of $M \setminus \pi(\Lambda)$.

\begin{proposition}\label{prop:trivial-action-rel}
    For $\Lambda \subseteq S^*M$ be a compact subanalytic Legendrian, the standard $S^1$-action on 
    $$\Tr_{cst}(\Id_{\Sh_{\widehat\Lambda}(M)}) = p_!\Delta^* (\check{T}_{\epsilon}1_{\Delta_c})$$
    is the trivial $S^1$-action.
\end{proposition}

Again, using adjunction, we can consider the cyclic rotation on $\Hom(\check{T}_{\epsilon/n}1_{\Delta_c}^{\boxtimes n}, 1_{\tilde{\Delta}_n})$. We study the behaviour of the rotation on the zero section and away from the zero section separately.

First, we can show that the support of $\mhom(\check{T}_{\epsilon/n}1_{\Delta_c}^{\boxtimes n}, 1_{\tilde{\Delta}_n})$ is in the fixed point set of the cyclic permutation, similar to Lemma \ref{lem:support-diagonal}.

\begin{lemma}\label{lem:support-diagonal-rel}
    For $\Lambda \subseteq S^*M$ be a compact subanalytic Legendrian, $\delta_n: T^*M \to T^*(M \times M) \times \dots \times T^*(M \times M), \, (x, \xi) \mapsto (x, \xi, x, -\xi, \dots, x, \xi, x, -\xi)$, the restriction functor induces an isomorphism
    \begin{align*}
        \Gamma (S^*(M \times M)^n, \mhom(\check{T}_{\epsilon/n}1_\Delta^{\boxtimes n}, 1_{\tilde{\Delta}_n}))
        \simeq \Gamma(\delta_{n}(S^*M), \mhom(\check{T}_{\epsilon/n}1_\Delta, 1_{\tilde{\Delta}_n})).
    \end{align*}
\end{lemma}

For the effect on the zero section, we will need the following lemma to show that the support of the following sheaf is on the fixed point set of the rotation.

\begin{lemma}\label{lem:support-diagonal-zerosection}
    For $\Lambda \subseteq S^*M$ be a compact subanalytic Legendrian, $\delta_n: M \to (M \times M) \times \dots \times (M \times M), \, x \mapsto (x, x, \dots, x, x)$, the restriction functor induces an isomorphism
    \begin{align*}
        \Gamma((M \times M)^n, \sHom(1_{\Delta_c}^{\boxtimes n}, 1_{(M \times M)^n}) \otimes 1_{\tilde{\Delta}_n}) \simeq \Gamma(\delta_n(M), \sHom(1_{\Delta_c}^{\boxtimes n}, 1_{(M \times M)^n}) \otimes 1_{\tilde{\Delta}_n}).
    \end{align*}
\end{lemma}
\begin{proof}
    Consider the estimation of the support that
    $$\mathrm{supp}(\sHom(1_{\Delta_c}^{\boxtimes n}, 1_{(M \times M)^n}) \otimes 1_{\tilde{\Delta}_n}) \subseteq \mathrm{supp}(\sHom(1_{\Delta_c}^{\boxtimes n}, 1_{(M \times M)^n})) \cap \mathrm{supp}(1_{\tilde{\Delta}_n}) = \delta_n(M).$$
    Then the lemma immediately follows.
\end{proof}

\begin{proof}[Proof of Proposition \ref{prop:trivial-action-rel}]
    We prove that the isomorphism $t_n$ given by cyclic permutations
    $$t_n: p_{n !}\widetilde{\Delta}_n^*(\check{T}_{\epsilon/n}1_{\Delta_c} )^{\boxtimes n} \xrightarrow{\sim} p_{n !}\widetilde{\Delta}_n^*(\check{T}_{\epsilon/n}1_{\Delta_c} )^{\boxtimes n}$$
    is equivalent to the identity morphism. Then it will follow that the $S^1$-action on the cyclic object is trivial. By adjunction, it suffices to show the isomorphism is equivalent to the identity:
    $$t_n^\vee: \Hom(\check{T}_{\epsilon/n}1_{\Delta_c}^{\boxtimes n}, 1_{\tilde{\Delta}_n}) \xrightarrow{\sim} \Hom(\check{T}_{\epsilon/n}1_{\Delta_c}^{\boxtimes n}, 1_{\tilde{\Delta}_n}).$$
    Consider the Sato--Sabloff fiber sequence
    $$\Hom(T_{\epsilon/n}1_{\Delta_c}^{\boxtimes n}, 1_{\tilde{\Delta}_n}) \to \Hom(\check{T}_{\epsilon/n}1_{\Delta_c}^{\boxtimes n}, 1_{\tilde{\Delta}_n}) \to \Gamma(S^*(M \times M)^n, \mhom(\check{T}_{\epsilon/n}1_{\Delta_c}^{\boxtimes n}, 1_{\tilde{\Delta}_n})).$$
    It then suffices to show that the following isomorphisms are both equivalent to the identity:
    \begin{gather*}
        t_n^\vee: \Hom(T_{\epsilon/n}1_{\Delta_c}^{\boxtimes n}, 1_{\tilde{\Delta}_n}) \xrightarrow{\sim} \Hom(T_{\epsilon/n}1_{\Delta_c}^{\boxtimes n}, 1_{\tilde{\Delta}_n}), \\
        t_n^\vee: \Gamma (S^*(M \times M)^n, \mhom(\check{T}_{\epsilon/n}1_\Delta^{\boxtimes n}, 1_{\tilde{\Delta}_n})) \xrightarrow{\sim} \Gamma (S^*(M \times M)^n, \mhom(\check{T}_{\epsilon/n}1_\Delta^{\boxtimes n}, 1_{\tilde{\Delta}_n})).
    \end{gather*}
    For the second isomorphism, it follows from Lemma \ref{lem:support-diagonal-rel} that it is the identity. For the first isomorphism, we have isomorphisms that are compatible with the cylic rotation
    \begin{align*}
        \Hom(T_{\epsilon/n}1_{\Delta_c}^{\boxtimes n}, 1_{\tilde{\Delta}_n}) & \simeq \Gamma((M \times M)^n, \sHom(T_{\epsilon/n}1_{\Delta_c}^{\boxtimes n}, 1_{(M \times M)^n}) \otimes 1_{\tilde{\Delta}_n}) \\
        & \simeq \Gamma((M \times M)^n, \sHom(1_{\Delta_c}^{\boxtimes n}, 1_{(M \times M)^n}) \otimes 1_{\tilde{\Delta}_n}) \\
        & \simeq \Gamma(\delta_n(M), \sHom(1_{\Delta_c}^{\boxtimes n}, 1_{(M \times M)^n}) \otimes 1_{\tilde{\Delta}_n}).
    \end{align*}
    By the gapped non-characteristic deformation Proposition \ref{prop:gapped-fully-faithful}, we know that the first map is an isomorphism. By Lemma \ref{lem:trivial-action-zerosection}, we know that the second map is an isomorphism.

    Under the restriction to the image of $\delta_n$, the smooth map $t_n$ becomes the identity map, and the isomorphism $t_n^\vee$ becomes the identity morphism. Hence the microlocalization of this isomorphism will also be the identity. This thus completes the proof.
\end{proof}

\subsection{Strong Calabi--Yau structures}
In this section, we will show that with the standard $S^1$-action on both sides, the acceleration morphisms are $S^1$-equivariant morphisms, and thus get a canonical strong (relative) Calabi--Yau structure that lifts the weak (relative) Calabi--Yau structure.

\begin{proposition}\label{prop:accelerate-equivariant}
    For $\Lambda \subseteq S^*M$ be a compact subanalytic Legendrian, the acceleration morphism
    $\Tr_{cst}(\Id_{\msh_\Lambda(\Lambda)}) \to \Tr(\Id_{\msh_\Lambda(\Lambda)})$
    is $S^1$-equivariant.
\end{proposition}

\begin{proposition}\label{prop:accelerate-equivariant-rel}
    For $\Lambda \subseteq S^*M$ be a compact subanalytic Legendrian, the acceleration morphism
    $\Tr_{cst}(\Id_{\Sh_{\widehat\Lambda}(M)}) \to \Tr(\Id_{\Sh_{\widehat\Lambda}(M)})$
    is $S^1$-equivariant.
\end{proposition}

Again, as the arguments will be the same, we only write down the proof of the more complcated case, namely Proposition \ref{prop:accelerate-equivariant}.

\begin{proof}[Proof of Proposition \ref{prop:accelerate-equivariant}]
    We need to show the compatibility of the acceleration morphism with the degeneracy morphisms and the translation morphisms. For the degeneracy morphism, compatibility is equivalent to the commutative diagram
    \[\resizebox{\textwidth}{!}{
    \xymatrix@C=4mm{
    p_{n+1 !}\widetilde{\Delta}_{n+1}^*\mathrm{Fib}(\check{T}_{\epsilon/(n+1)}1_\Delta \to T_{\epsilon/(n+1)} 1_\Delta)^{\boxtimes n+1} \ar[r]^-{d_i} \ar[d] & p_{n !}\widetilde{\Delta}_n^*\mathrm{Fib}(\check{T}_{\epsilon/n}1_\Delta \to T_{\epsilon/n} 1_\Delta)^{\boxtimes n} \ar[d] \\
    p_{n+1 !}\widetilde{\Delta}_{n+1}^* \iota_{\widehat\Lambda_{\cup,\epsilon} \times -\widehat\Lambda_{\cup, \epsilon}}^* \mathrm{Fib}(\check{T}_{\epsilon/(n+1)}1_\Delta \to T_{\epsilon/(n+1)} 1_\Delta)^{\boxtimes n+1} \ar[r]^-{d_i} & p_{n !}\widetilde{\Delta}_n^* \iota_{\widehat\Lambda_{\cup,\epsilon} \times -\widehat\Lambda_{\cup, \epsilon}}^* \mathrm{Fib}(\check{T}_{\epsilon/n}1_\Delta \to T_{\epsilon/n} 1_\Delta)^{\boxtimes n},
    }}\]
    while for the translation morphism, compatibility is equivalent to the commutative diagram
    \[\xymatrix{
    p_{n !}\widetilde{\Delta}_{n}^*\mathrm{Fib}(\check{T}_{\epsilon/n}1_\Delta \to T_{\epsilon/n} 1_\Delta)^{\boxtimes n} \ar[r]^{t_n} \ar[d] & p_{n !}\widetilde{\Delta}_n^*\mathrm{Fib}(\check{T}_{\epsilon/n}1_\Delta \to T_{\epsilon/n} 1_\Delta)^{\boxtimes n} \ar[d] \\
    p_{n !}\widetilde{\Delta}_{n}^* \iota_{\widehat\Lambda_{\cup,\epsilon} \times -\widehat\Lambda_{\cup, \epsilon}}^* \mathrm{Fib}(\check{T}_{\epsilon/n}1_\Delta \to T_{\epsilon/n} 1_\Delta)^{\boxtimes n} \ar[r]^{t_n} & p_{n !}\widetilde{\Delta}_n^* \iota_{\widehat\Lambda_{\cup,\epsilon} \times -\widehat\Lambda_{\cup, \epsilon}}^* \mathrm{Fib}(\check{T}_{\epsilon/n}1_\Delta \to T_{\epsilon/n} 1_\Delta)^{\boxtimes n},
    }\]
    Both diagrams are obvious. This shows the $S^1$-equivariancy of the acceleration morphism.
\end{proof}

Using the above result, we can now show the existence of a strong absolute Calabi--Yau structure on microsheaves on Legendrians in the cophere bundle and a strong relative Calabi--Yau structure on the continuous adjunction pair from microlocalization.

\begin{theorem}\label{thm:strong-cy}
    Let $M$ be a connected oriented manifold of dimension $n$ and $\Lambda \subseteq S^*M$ be a compact subanalytic Legendrian. Then the category $\msh_\Lambda(\Lambda)$ admits a canonical $(n-1)$-dimensional strong smooth Calabi--Yau structure.
\end{theorem}
\begin{proof}
    Consider the weak smooth Calabi--Yau structure on $\msh_\Lambda(\Lambda)$ determined in Theorem \ref{thm:weak-relative-cy}. By Proposition \ref{prop:accelerate}, we know that the Calabi--Yau structure factors through the acceleration morphism
    $$1_\cV[n-1] \to \Gamma(\Lambda, 1_\Lambda)[n-1] \to  \HH_*(\msh_\Lambda(\Lambda)).$$
    By Proposition \ref{prop:accelerate-equivariant}, the acceleration morphism is $S^1$-equivariant. By Proposition \ref{prop:trivial-action}, the inclusion morphism to constant orbits with trivial action is $S^1$-equivariant. There is a canonical lift of the unit $1_\cV \in \Gamma(\Lambda, 1_\Lambda)$ to the unit
    $$1_\cV \in \Gamma(\Lambda, 1_\Lambda)^{S^1} \simeq \Gamma(\Lambda, 1_\Lambda) \otimes \Gamma(BS^1, 1_{BS^1}).$$
    Therefore, the weak Calabi--Yau structure has a canonical lift to the negative cyclic homology $\HH_*(\msh_\Lambda(\Lambda))^{S^1}$.
\end{proof}

\begin{theorem}\label{thm:strong-rel-cy}
    Let $M$ be a connected oriented manifold of dimension $n$ and $\Lambda \subseteq S^*M$ be a compact subanalytic Legendrian. Then the continuous adjunction
    $$m_\Lambda^l: \msh_\Lambda(\Lambda) \rightleftharpoons \Sh_{\widehat\Lambda}(M) : m_\Lambda$$
    admits a canonical $n$-dimensional strong smooth relative Calabi--Yau structure.
\end{theorem}
\begin{proof}
    Consider the weak smooth Calabi--Yau structure on $\msh_\Lambda(\Lambda) \rightleftharpoons \Sh_\Lambda(M)$ determined in Theorem \ref{thm:weak-relative-cy}. By Proposition \ref{prop:accelerate}, we know that the Calabi--Yau structure factors through the acceleration morphism
    \[\xymatrix{
    1_\cV[n-1] \ar[r] \ar[d] & \Gamma(\Lambda, 1_\Lambda)[n-1] \ar[r] \ar[d] & \HH_*(\msh_\Lambda(\Lambda)) \ar[d] \\
    0 \ar[r] & \Gamma(\widehat\Lambda, 1_{\widehat\Lambda \setminus \Lambda})[n] \ar[r] & \HH_*(\Sh_\Lambda(M)).
    }\]
    and the composition is null homotopic. The cofiber of the vertical morphisms defines the class
    $$1_\cV[n] \longrightarrow \Gamma(\widehat\Lambda, 1_{\widehat\Lambda})[n] \longrightarrow \HH_*(\msh_\Lambda(\Lambda), \Sh_{\widehat\Lambda}(M)).$$
    By Proposition \ref{prop:accelerate-equivariant}, the acceleration is $S^1$-equivariant. By Proposition \ref{prop:trivial-action}, the inclusion morphism to constant orbits with trivial action is $S^1$-equivariant. There is a canonical lift of the unit $1_\cV \in \Gamma(\Lambda, 1_{\Lambda})$ and the zero element $0 \in \Gamma(\widehat\Lambda, 1_{\widehat\Lambda \setminus \Lambda})$ to the unit and zero element
    $$1_\cV \in \Gamma(\widehat\Lambda, 1_{\widehat\Lambda})^{S^1} \simeq \Gamma(\widehat\Lambda, 1_{\widehat\Lambda}) \otimes \Gamma(BS^1, 1_{BS^1}), \;\; 0 \in \Gamma(\widehat\Lambda, 1_{\widehat\Lambda \setminus \Lambda})^{S^1} \simeq \Gamma(\widehat\Lambda, 1_{\widehat\Lambda \setminus \Lambda}) \otimes \Gamma(BS^1, 1_{BS^1}).$$
    Therefore, the weak relative Calabi--Yau structure admits a canonical lifting to the relative negative cyclic homology $\HH_*(\msh_\Lambda(\Lambda), \Sh_{\widehat\Lambda}(M))^{S^1}$.
\end{proof}

Finally, by Theorem \ref{thm:smooth-induce-proper}, we can then also conclude the following:

\begin{corollary}
    Let $M$ be an oriented manifold of dimension $n$ and $\Lambda \subseteq S^*M$ be a compact subanalytic Legendrian. Then the continuous functor
    $$m_\Lambda: \Sh_{\widehat\Lambda}^b(M) \to \msh_\Lambda^b(\Lambda)$$
    admits a canonical $n$-dimensional strong proper relative Calabi--Yau structure, and $\msh_\Lambda^b(\Lambda)$ admits a canonical $(n-1)$-dimensional strong proper Calabi--Yau structure.
\end{corollary}

\subsection{Canonical Calabi--Yau structures}
We can show naturality of the relative Calabi--Yau structure for deformations of subanalytic Legendrians at infinity.

\begin{theorem}\label{thm:canonical-calabi-yau}
    Let $\Lambda \subseteq S^*(M \times [0, 1])$ be a subanalytic Legendrian such that
    $$\Lambda \cap S^*(M \times [0, \epsilon)) = \Lambda_0 \times [0, \epsilon), \; \Lambda \cap S^*(M \times (1-\epsilon, 1]) = \Lambda_1 \times (1-\epsilon, 1].$$
    Then the Calabi--Yau structures are compatible under the restriction functors
    $$\HH_*(\msh_{\Lambda_0}(\Lambda_0))^{S^1} \leftarrow \HH_*(\msh_{\Lambda}(\Lambda))^{S^1} \rightarrow \HH_*(\msh_{\Lambda_1}(\Lambda_1))^{S^1},$$
    and the relative Calabi--Yau structures are compatible under the restriction functors
    $$\HH_*(\msh_{\Lambda_0}(\Lambda_0), \Sh_{\widehat\Lambda_0}(M))^{S^1} \leftarrow \HH_*(\msh_{\Lambda}(\Lambda), \Sh_{\widehat\Lambda}(M \times [0, 1]))^{S^1} \rightarrow \HH_*(\msh_{\Lambda_1}(\Lambda_1), \Sh_{\widehat\Lambda_1}(M))^{S^1}.$$
\end{theorem}

Let $\Delta: M \to M \times M$ and $\bar{\Delta}: M \times [0, 1] \to M \times M \times [0, 1] \times [0, 1]$ be the diagonal maps.
For $j = 0, 1$, we let $\Omega_j$ be a small neighbourhood of $\Lambda_j \subseteq S^*M$, and $\rho_j: S^*M \to \bR$ be a smooth cut-off function in Definition \ref{def:hamiltonian}. Similarly, we let $\Omega$ be a small neighbourhood of $\Lambda \subseteq S^*(M \times I)$, and $\rho: S^*(M \times [0, 1]) \to \bR$ be a smooth cut-off function in Definition \ref{def:hamiltonian}. Moreover, we require that $\Omega \cap S^*(M \times I_\epsilon(j)) = \Omega_j \times I_\epsilon(j)$, where $I_\epsilon(j)$ is the open neighbourhood of $j \in [0, 1]$ of radius $\epsilon$, and near the boundary of the interval, the cut-off functions satisfy
$$\rho(x, t, \xi, \tau) = \rho_j(x, \xi), \; t \in I_\epsilon(j).$$
Abusing notations, we will denote the corresponding Hamiltonian flow by $\check{T}_t: S^*M \to S^*M$ and $\check{T}_t: S^*(M \times [0, 1]) \to S^*(M \times [0, 1])$.

\begin{proof}[Proof of Theorem \ref{thm:canonical-calabi-yau}]
First, we identify the weak Calabi--Yau classes. We have natural morphisms of diagrams of Hochschild homologies
\[\xymatrix{
\HH_*(\msh_{\Lambda}(\Lambda)) \ar[d] \ar[r] & \HH_*(\Sh_{\widehat\Lambda}(M \times [0, 1])) \ar[d] \\
\HH_*(\msh_{\Lambda_0}(\Lambda_0)) \ar[r] & \HH_*(\Sh_{\widehat\Lambda_0}(M)).
}\]
Then, to show that the weak relative Calabi--Yau classes agree, it suffices to construct the following commutative diagram and identify the Calabi--Yau classes on the left hand side
\[\xymatrix{
\bar{p}_!\bar{\Delta}^* \mathrm{Fib}(\check{T}_{\epsilon}1_{\bar{\Delta}} \to T_\epsilon 1_{\bar{\Delta}}) \ar[r] \ar[d] & \bar{p}_!\bar{\Delta}^*\iota_{-\widehat{\Lambda}_{\cup} \times \widehat{\Lambda}_{\cup}}^* \mathrm{Fib}(\check{T}_{\epsilon}1_{\bar{\Delta}} \to T_\epsilon 1_{\bar{\Delta}}) \ar[r] \ar[d] & \bar{p}_!\bar{\Delta}^*\iota_{-\widehat{\Lambda} \times \widehat{\Lambda}}^* \check{T}_\epsilon 1_{\bar{\Delta}} \ar[d] \\
p_!\Delta^* \mathrm{Fib}(\check{T}_{\epsilon}1_\Delta \to T_\epsilon 1_\Delta) \ar[r] & p_!\Delta^*\iota_{-(\widehat{\Lambda}_0)_{\cup} \times (\widehat{\Lambda}_0)_{\cup}}^* \mathrm{Fib}(\check{T}_{\epsilon}1_\Delta \to T_\epsilon 1_\Delta) \ar[r] & p_!\Delta^*\iota_{-\widehat{\Lambda}_0 \times \widehat{\Lambda}_0}^* \check{T}_\epsilon 1_\Delta.
}\]
First, we construct the left vertical morphism. Since we have chosen the Hamiltonian function such that for $j = 0, 1$ and $t \in I_\epsilon(j)$, $\rho(x, t, \xi, \tau) = \rho_j(x, \xi)$, the Hamiltonian flows satisfy the conditions that
$$\check{T}_t(\Lambda \cap S^*(M \times I_\epsilon(j))) = \check{T}_t(\Lambda_j) \times I_\epsilon(j), \; T_t(\Lambda \cap S^*(M \times I_\epsilon(j))) = T_t(\Lambda_j) \times I_\epsilon(j).$$
Therefore, for the inclusion maps $i_{(j,j)}: M \times M \times \{(j, j)\} \hookrightarrow M \times M \times [0, 1] \times [0, 1]$, we have
$$\check{T}_\epsilon ( i_{(j,j)}^* 1_{\bar{\Delta}}) = i_{(j,j)}^* \check{T}_\epsilon 1_{\bar{\Delta}}, \; {T}_\epsilon ( i_{(j,j)}^* 1_{\bar{\Delta}}) = i_{(j,j)}^* {T}_\epsilon 1_{\bar{\Delta}}.$$
Thus, we can get natural morphisms
\begin{align*}
\bar{p}_!\bar{\Delta}^* \mathrm{Fib}(\check{T}_{\epsilon}1_{\bar{\Delta}} \to T_\epsilon 1_{\bar{\Delta}}) &\to p_!i_j^* \bar{\Delta}^* \mathrm{Fib}(\check{T}_{\epsilon}1_{\bar{\Delta}} \to T_\epsilon 1_{\bar{\Delta}}) \xrightarrow{\sim} p_! \Delta^* i_{(j,j)}^* \mathrm{Fib}(\check{T}_{\epsilon}1_{\bar{\Delta}} \to T_\epsilon 1_{\bar{\Delta}})
\\
&\xrightarrow{\sim} p_! \Delta^* \mathrm{Fib}(\check{T}_{\epsilon} (i_{(j,j)}^* 1_{\bar{\Delta}}) \to T_\epsilon (i_{(j,j)}^* 1_{\bar{\Delta}})) \xrightarrow{\sim} p_!{\Delta}^* \mathrm{Fib}(\check{T}_{\epsilon}1_{{\Delta}} \to T_\epsilon 1_{{\Delta}}),
\end{align*}
which, under the isomorphisms from Proposition \ref{prop:accelerate}, is compatible with the morphism
$$\Gamma(\Lambda, 1_\Lambda) \to \Gamma(\Lambda_j, 1_{\Lambda_j}).$$
Then, we also construct the vertical morphism in the middle and on the right between Hochschild homologies. Consider the commutative diagram
\[\xymatrix{
\Sh_{\widehat\Lambda \times -\widehat\Lambda}(M \times M \times [0, 1] \times [0, 1]) \ar[r]^-{\iota_{\widehat\Lambda \times -\widehat\Lambda}} \ar[d]_{i_{(j,j)}^*} & \Sh(M \times M \times [0, 1] \times [0, 1]) \ar[d]^{i_{(j,j)}^*} \\
\Sh_{\widehat\Lambda_j \times -\widehat\Lambda_j}(M \times M) \ar[r]^-{\iota_{\widehat\Lambda_j \times -\widehat\Lambda_j}} & \Sh(M \times M). 
}\]
Then, by adjunctions, we have the natural morphism 
$$\iota_{\widehat\Lambda_j \times -\widehat\Lambda_j}^* \circ i_{(j,j)}^* \to i_{(j,j)}^* \circ \iota_{\widehat\Lambda \times -\widehat\Lambda}^*.$$
Thus, we can get natural morphisms on Hochschild homologies
\begin{align*}
\bar{p}_!\bar{\Delta}^* \iota_{\widehat\Lambda \times -\widehat\Lambda}^* & \mathrm{Fib}(\check{T}_{\epsilon}1_{\bar{\Delta}} \to T_\epsilon 1_{\bar{\Delta}}) \to p_!i_j^* \bar{\Delta}^* \iota_{\widehat\Lambda \times -\widehat\Lambda}^* \mathrm{Fib}(\check{T}_{\epsilon}1_{\bar{\Delta}} \to T_\epsilon 1_{\bar{\Delta}}) \\
&\xrightarrow{\sim} p_! \Delta^* i_{(j,j)}^* \iota_{\widehat\Lambda \times -\widehat\Lambda}^* \mathrm{Fib}(\check{T}_{\epsilon}1_{\bar{\Delta}} \to T_\epsilon 1_{\bar{\Delta}}) \to p_! \Delta^* \iota_{\widehat\Lambda \times -\widehat\Lambda}^* i_{(j,j)}^*\mathrm{Fib}(\check{T}_{\epsilon} 1_{\bar{\Delta}} \to T_\epsilon 1_{\bar{\Delta}}) \\
&\xrightarrow{\sim} p_! \Delta^* \iota_{\widehat\Lambda \times -\widehat\Lambda}^* \mathrm{Fib}(\check{T}_{\epsilon} (i_{(j,j)}^* 1_{\bar{\Delta}}) \to T_\epsilon (i_{(j,j)}^*1_{\bar{\Delta}}) )
\xrightarrow{\sim} p_!{\Delta}^* \mathrm{Fib}(\check{T}_{\epsilon}1_{{\Delta}} \to T_\epsilon 1_{{\Delta}}),
\end{align*}
which is compatible with the morphisms on the constant orbits. Since the weak (relative) Calabi--Yau classes come from the identity elements on both sides, we know they are naturally identified.

Then, to identify the strong (relative) Calabi--Yau classes, we upgrade the above commutative diagram into an $S^1$-equivariant commutative diagram. The argument is exactly the same as above. Then consider the canonical lifting of the identity
$$1_\cV \in \Gamma(\Lambda, 1_\Lambda)^{S^1} = \Gamma(\Lambda, 1_\cV) \otimes \Gamma(BS^1, 1_{BS^1}), \; 1_\cV \in \Gamma(\Lambda_j, 1_{\Lambda_j})^{S^1} = \Gamma(\Lambda_j, 1_{\Lambda_j}) \otimes \Gamma(BS^1, 1_{BS^1}).$$
Since the (relative) Calabi--Yau classes come from the identity elements on both sides, they are naturally identified.
\end{proof}

\begin{remark}
    We did not require or claim any Calabi--Yau property on the continuous adjunction $m_\Lambda^l: \msh_\Lambda(\Lambda) \rightleftharpoons \Sh_{\widehat\Lambda}(M \times [0, 1]) : m_\Lambda$ in the above proof.
\end{remark}

When one or both of the restriction functors are equivalences, we can get a functor between the pairs of categories and directly compare the (relative) Calabi--Yau structures. The following result is a direct corollary of Theorems \ref{thm:canonical-calabi-yau} and \ref{thm:invariance-deformation}.

\begin{corollary}\label{cor:canonical-cy}
    Let $\Lambda \subseteq S^*(M \times [0, 1])$ be a subanalytic Legendrian deformation between $\Lambda_0$ and $\Lambda_1 \subseteq S^*M$. Then 
    the Calabi--Yau structures are compatible under the isomorphism
    $$\HH_*(\msh_{\Lambda_0}(\Lambda_0))^{S^1} \xrightarrow{\sim} \HH_*(\msh_{\Lambda_1}(\Lambda_1))^{S^1},$$
    and the relative Calabi--Yau structures are compatible under the restriction functors
    $$\HH_*(\msh_{\Lambda_0}(\Lambda_0), \Sh_{\widehat\Lambda_0}(M))^{S^1} \xrightarrow{\sim} \HH_*(\msh_{\Lambda_1}(\Lambda_1), \Sh_{\widehat\Lambda_1}(M))^{S^1}.$$
\end{corollary}

\begin{remark}\label{rem:shende-takeda}
    The relative Calabi--Yau structure have also been obtained by Shende--Takeda using arborealizations \cite{Shende-Takeda}. However, it seems hard to prove that the Calabi--Yau structure is canonical using the approach of arborealization. For instance, in the absolute case, they constructed the Calabi--Yau class as a global (co)section on the (co)sheafified Hochschild homology, which maps to the Hochschild homology of the global (co)section category \cite[Proposition 4.26 \& 4.27]{Shende-Takeda}:
    \[\xymatrix{
    1_\cV[n-1] \ar[r] \ar[d] & \Gamma(\Lambda, \mathscr{HH}(\msh_\Lambda)) \ar[r] \ar[d] & \HH_*(\msh_\Lambda(\Lambda)) \ar[d] \\
    0 \ar[r] & \Gamma(\Lambda, \mathscr{HH}(\Sh_{\widehat\Lambda})) \ar[r] & \HH_*(\Sh_{\widehat\Lambda}(M)).
    }\]
    We remark that the global (co)section of the (co)sheafified Hochschild homology is definitely not invariant under Liouville or Weinstein homotopies. In fact, we can consider a smooth Legendrian $\Lambda \subseteq S^*\bR^n$ with $\widehat\Lambda \subseteq T^*\bR^n$ such that $\HH_*(\msh_{\widehat\Lambda}(\widehat\Lambda))$ is not a compact object. We can deform the Lagrangian $\widehat\Lambda \subseteq T^*\bR^n$ into the cone singularity $\Lambda \times \bR_{\geq 0} \subseteq \bR^{2n}$. Then it is clear that 
    $$\Gamma(\Lambda \times \bR_{\geq 0}, \mathscr{HH}(\msh_{\Lambda \times \bR_{\geq 0}}^c)) = \HH_*(\msh_{\Lambda \times \bR_{\geq 0}}^c(\Lambda \times \bR_{\geq 0})) = \HH_*(\msh_{\widehat\Lambda}^c(\widehat\Lambda))$$ is not compact. However, for an compact arboreal singularity, as shown in \cite[Lemma 4.35 \& Theorem 6.7]{Shende-Takeda}, $\Gamma(\Lambda, \mathscr{HH}(\msh_\Lambda^c))$ is always a compact object.

    We also remark that the relation between the relative Calabi--Yau structure constructed by Shende--Takeda and the one constructed by us. For a Legendrian $\Lambda \subseteq S^*M$ with $\widehat\Lambda \subseteq T^*M$ such that the Legendrian only has immersed double points in the projection \cite[Section 7.3]{Shende-Takeda}, the Calabi--Yau structure they constructed is comes from a map
    \[\xymatrix{
    1_\cV[n-1] \ar[r] \ar[d] & \Gamma(\Lambda, \omega_\Lambda^\vee) \ar[r] \ar[d] & \HH_*(\msh_\Lambda(\Lambda)) \ar[d] \\
    0 \ar[r] & \Gamma(\widehat\Lambda, \omega_{\widehat\Lambda}^\vee) \ar[r] & \HH_*(\Sh_{\widehat\Lambda}(M)).
    }\]
    Here $\omega_\Lambda = p^!1_\cV$, $\omega_{\widehat\Lambda} = \widehat{p}^!1_\cV$ are the dualizing sheaves for $\Lambda$, $\widehat\Lambda$ where $p: \Lambda \to \{*\}$ and $\widehat p: \widehat\Lambda \to \{*\}$. When $M$ is orientable, then indeed we know that $\omega_M^\vee = 1_M[-n]$ and $\omega_\Lambda^\vee = 1_\Lambda[n-1]$. Taking global sections will give the terms
    $$\Gamma(\Lambda, \omega_\Lambda^\vee) = \Gamma(\Lambda, 1_\cV)[n-1], \; \Gamma(\widehat\Lambda, \omega_{\widehat\Lambda}^\vee) = \Gamma(\widehat\Lambda, 1_{\widehat\Lambda \setminus \Lambda})[n-1].$$
    Thus, checking the compatibility between the Calabi--Yau structure is reduced to checking the compatibility between the acceleration morphisms. We expect that this should follow from local computations if we use partition of unity to decompose the sheaf kernel used to construct the acceleration morphism, as in \cite[Section 4.4 \& 4.5]{Kuo-Li-spherical}. However, writing down the details could be nontrivial. 
\end{remark}



\section{Examples of Calabi--Yau structures}
We explain some examples of (relative) Calabi--Yau structures that fit into our setting. In particular, we recover other known cases of (relative) Calabi--Yau structures in the literature. We also explain some construction of symplectic structures and Lagrangian structures on moduli spaces, which are induced by (relative) Calabi--Yau structures of the corresponding categories.

\subsection{Calabi--Yau structures in smooth topology}
First, we recover the example of relative Calabi--Yau structure coming from smooth topology.

\begin{example}
    Let $(M, \partial M)$ be an oriented manifold with boundary. Let $\overline{M} = M \cup_{\partial M} \partial M \times [0, +\infty)$ be the completion. Consider the inward unit conormal bundle $S^*_{\partial M,-}\overline{M}$ that is diffeomorphic to $\partial M$. Then there is a relative Calabi--Yau structure on
    $$\msh_{S^*_{\partial M,-}\overline{M}}(S^*_{\partial M,-}\overline{M}) \rightleftharpoons \Sh_{S^*_{\partial M,-}\overline{M}}(\overline{M})_0.$$
    Since the projection $S^*_{\partial M,-}\overline{M} \to \overline{M}$ is a smooth embedding, we know that the Legendrian $S^*_{\partial M,-}\overline{M}$ is equipped with a canonical Maslov data. We have a commutative diagram
    \[\xymatrix{
    \msh_{S^*_{\partial M,-}\overline{M}}(S^*_{\partial M,-}\overline{M})  \ar[r] \ar[d]_{\rotatebox{90}{$\sim$}} & \Sh_{S^*_{\partial M,-}\overline{M}}(\overline{M})_0 \ar[d]^{\rotatebox{90}{$\sim$}} \\
    \Loc(\partial M) \ar[r] & \Loc(M).
    }\]
    Therefore, we recover the relative Calabi--Yau structure on $\Loc(\partial M) \rightleftharpoons \Loc(M)$ \cite[Theorem 5.7]{Brav-Dyckerhoff1}.
\end{example}

We consider the examples of relative Calabi--Yau structures that come from the link dg category by Yeung \cite{YeungComplete} and Berest--Eshmatov--Yeung \cite{BerestEshmatovYeung}. This is closely related to the construction of (relative) Calabi--Yau completions by Yeung \cite{YeungComplete}.

\begin{example}
    Let $(D^n, \partial D^n)$ be the $n$-dimensional disk with boundary and $x_1, \dots, x_k \in D^n$ be $k$ distinct points where $n \geq 2$. Then there is a relative Calabi--Yau structure on
    $$\msh_{S^*_{\partial D^n,-}\bR^n}(S^*_{\partial D^n,-}\bR^n) \oplus \bigoplus\nolimits_{i=1}^k \msh_{S^*_{x_i}\bR^n}(S^*_{x_i}\bR^n) \rightleftharpoons \Sh_{S^*_{\partial D^n,-}\bR^n \sqcup\, \bigsqcup_{i=1}^k S^*_{x_i}\bR^n}(\bR^n)_0.$$
    Since the projection $S^*_{\partial M,-}\overline{M} \to \overline{M}$ is a smooth embedding, and the projection $S^*_{x_i}M \to \overline{M}$ is a smooth embedding after a small Reeb pushoff, the Legendrian $S^*_{\partial D^n,-}\bR^n$ and $S^*_{x_i}\bR^n$ are equipped with canonical Maslov data. Therefore, we get a relative Calabi--Yau structure
    $$\Loc(S^{n-1}) \oplus \bigoplus\nolimits_{i=1}^k \Loc(S^*_{x_i}D^n) \rightleftharpoons \Sh_{\bigsqcup_{i=1}^k S^*_{x_i}\bR^n}(D^n).$$
    
    Let $\Vec{I}_k$ be the quiver with $(k+1)$ vertices and $k$ arrows from the first $k$ vertices to the last vertex. Yeung \cite{YeungComplete} constructed a relative Calabi--Yau structure on 
    $$\Bbbk \otimes \Pi_{n-1}(\sqcup_{i=1}^{k+1} *) \rightleftharpoons \Bbbk \otimes \Pi_n(\sqcup_{i=1}^{k+1} *, \Vec{I}_k).$$
    For $n \geq 3$, $\Pi_{n-1}(\sB)$ is the $(n-1)$-Calabi--Yau completion of $\sB$ and $\Pi_n(\sB, \sA)$ is the relative $n$-Calabi--Yau completion of $F: \sB \to \sA$; for $n = 2$, they are the localized Calabi--Yau completion and relative Calabi--Yau completion \cite[Definition 4.33]{YeungComplete}. Here, we abuse the notations (the localized version is denoted by $\Pi_{n-1}^\text{loc}(\sB)$ and $\Pi_n^\text{loc}(\sB, \sA)$ in \cite[Section 4.4]{YeungComplete}). We claim that there is a commutative diagram
    \[\xymatrix{
    \Loc(S^{n-1}) \oplus \bigoplus\nolimits_{i=1}^k \Loc(S^{n-1}) \ar[r] \ar[d]_{\rotatebox{90}{$\sim$}} &  \Sh_{\sqcup_{i=1}^k S^*_{x_i}D^n}(D^n) \ar[d]^{\rotatebox{90}{$\sim$}} \\
    \Bbbk \otimes \Pi_{n-1}(\sqcup_{i=1}^{k+1} *) \ar[r] & \Bbbk \otimes \Pi_n(\sqcup_{i=1}^{k+1} *, \Vec{I}_k).
    }\]
    Recall that for $n \geq 3$, $\Bbbk \otimes \Pi_{n-1}(*) \simeq \Bbbk[s]$, and for $n = 2$, $\Bbbk \otimes \Pi_{n-1}(*) \simeq \Bbbk[s, s^{-1}]$ where $s$ has degree $2-n$. Thus, we know that
    $$\Bbbk \otimes \Pi_{n-1}(*) \simeq C_{-*}(\Omega_*S^{n-1}) \simeq \Loc(S^{n-1}).$$
    On the other hand, objects in $\Bbbk \otimes \Pi_n(\sqcup_{i=1}^{k+1} *, \Vec{I}_k)$ consists of $F_1, \dots, F_{k+1} \in \Mod(\Bbbk)$, maps $\mu_i: F_i \to F_i[2-n]$, $s_i: F_i \to F_i[1-n]$, maps $f_i: F_{k+1} \to F_i$ of degree $0$ and maps $f_i^*: F_i \to F_{k+1}[2 - n]$ such that (the $-1$ term only appears when $n = 2$)
    $$ds_i = \mu_i - f_i f_i^* (- 1), \, 1 \leq i \leq k, \;\; ds_{k+1} = \mu_{k+1} - f_1^* f_1 - \dots - f_k^* f_k (- 1).$$
    For $F \in \Sh_{\bigsqcup_{i=1}^k S^*_{x_i}D^n}(D^n)$, we set $F_{k+1}$ to be the stalk away from $x_1, \dots, x_k$, $G_i$ to be the stalk at $x_i$, and $F_i = \mathrm{Cofib}(G_i \to F_{k+1})$ to be the microstalk on $S^*_{x_i}D^n$. Let $$\mu_i: F_i \to F_i[2 - n]$$
    to be the monodromy of the (microlocal) local system with stalk $F_i$ on $S^{n-1}$. Let $f_i: F_{k+1} \to F_i$ to be the natural map from the stalk to microstalk, and finally we set $$s_i: F_i \to F_i[1-n], \; f_i^*: F_i \to F_{k+1}[2 - n]$$
    to be the nap induced by the trivialization of the local system with stalk $G_i = \mathrm{Fib}(F_{k+1} \to F_i)$ on $S^{n-1}$ which extends to $D^n$. Hence Shende--Treumann--Williams \cite[Proposition 4.9]{Shende-Treumann-Williams}, Etg\"u--Lekili \cite[Theorem 1]{EtguLekiliPreprojective} and Karabas--Lee \cite[Theorem 4.2 \& 4.4]{KarabasLee} will imply that 
    $$\Bbbk \otimes \Pi_n(\sqcup_{i=1}^{k+1} *, \Vec{I}_k) \simeq \Sh_{\bigsqcup_{i=1}^n S^*_{x_i}D^n}(D^n).$$
    Thus, we recover the relative Calabi--Yau structure induced by the relative Calabi--Yau completion of the quiver with two vertices. Fixing the monodromy $\mu_i$ for $1 \leq i \leq k+1$ by taking a homotopy pushout (corresponding to attaching disks to $\partial D^n$ and $S_{x_i}^*D^n$):
    $$\Bbbk \otimes \left((\sqcup_{i=1}^{k+1} *) \otimes_{\Pi_n(\sqcup_{i=1}^{k+1} *)} \Pi_n(\sqcup_{i=1}^{k+1} *, \Vec{I}_k)\right)$$
    we obtain exactly exactly the derived multiplicative preprojective algebra \cite{BezrukavnikovKapranov,BozecCalaqueScherotzke}.
\end{example}

\begin{example}
    Let $(M, \partial M)$ be an oriented manifold with boundary and $N \subseteq M$ be a submanifold with trivialized unit conormal bundle $S^*_NM \cong N \times S^{r-1}$. Let $\overline{M} = M \cup_{\partial M} \partial M \times [0, +\infty)$ be the completion and consider the inward unit conormal bundle $S^*_{\partial M, -}\overline{M}$. Then there is a relative Calabi--Yau structure on 
    $$\msh_{S^*_{\partial M,-}\overline{M}}(S^*_{\partial M,-}\overline{M}) \oplus \msh_{S^*_NM}(S^*_NM) \rightleftharpoons \Sh_{S^*_{\partial M,-}\overline{M} \sqcup S^*_NM}(\overline{M})_0.$$
    Since the projection $S^*_{\partial M,-}\overline{M} \to \overline{M}$ is a smooth embedding, and the projection $S^*_N\overline{M} \to \overline{M}$ is a smooth embedding after a small Reeb pushoff, the Legendrian $S^*_{\partial M,-}\overline{M}$ and $S^*_NM$ are equipped with canonical Maslov data. Therefore, we get a relative Calabi--Yau structure
    $$\Loc(\partial M) \oplus \Loc(N \times S^{r-1}) \rightleftharpoons \Sh_{S^*_NM}(M).$$
    We claim that $\Sh_{S^*_NM}(M)$ is quasi-equivalent to modules over the perversely thickened link dg category $\sA(M, N)$, thus identifying our functors with the construction in Yeung \cite[Section 5]{YeungComplete} and Berest--Eshmatov--Yeung \cite{BerestEshmatovYeung}.

    Consider a tubular neighbourhood $N \times D^{r}$ of $N \subset M$. Then we can write down a homotopy push out formula
    $$\Sh_{S^*_NM}(M) = \operatorname{colim}\left(\Loc(M \backslash N) \leftarrow \Loc(N \times S^{r-1}) \to \Sh_{S^*_N(N \times D^{r})}(N \times D^{r}) \right).$$
    Let $\Vec{I}$ be the quiver with two vertices and one arrow. Consider the functor $C_{-*}(\Omega_*N) \oplus C_{-*}(\Omega_*N) \to C_{-*}(\Omega_*N) \otimes \Vec{I}$. Then the perversely thickened link dg category is defined as the following homotopy push out \cite[Definition 5.30]{YeungComplete}
    $$\sA(M, N) = \operatorname{colim}\left(C_{-*}(\Omega_*(M \backslash N)) \leftarrow C_{-*}(\Omega_*(N \times S^{r-1})) \to \Pi_n(C_{-*}(\Omega_*N)^{\oplus 2}, C_{-*}(\Omega_*N) \otimes \Vec{I}) \right).$$
    For $r \geq 3$, $\Pi_r(\sB, \sA)$ is the relative Calabi--Yau completion of $F: \sB \to \sA$; for $r = 2$, $\Pi_r(\sB, \sA)$ is the localized relative Calabi--Yau completion of $F: \sB \to \sA$ \cite[Definition 4.33]{YeungComplete}. Consider sheaves on a disk $D^r$ with microsupport in $S^*_0D^r$. Since
    $$\Sh_{S^*_0D^{r}}(D^{r}) \simeq \Bbbk \otimes \Pi_{r}(* \sqcup *, \Vec{I}).$$
    Then, by Kunneth formula, it suffices to show that there is an equivalence
    $$C_{-*}(\Omega_*N) \otimes \Pi_{r}(* \sqcup *, \Vec{I}) \simeq \Pi_n(C_{-*}(\Omega_*N)^{\oplus 2}, C_{-*}(\Omega_*N) \otimes \Vec{I}).$$
    This follows from direct computation using the explicit chain models in \cite[Section 3.4 \& 5.7]{YeungComplete} combined with the fact that $C_{-*}(\Omega_*N)$ is smooth Calabi--Yau of dimension $n-r$ \cite{Brav-Dyckerhoff1}.
\end{example}
\begin{remark}
    We compare our result with \cite[Theorem 5.89]{YeungComplete} that $\Mod^\textit{fd}(\sA(M, N)) \simeq \mathrm{Perv}_N(M)$. Our discussion above can easily recover that result by fixing the degree of the microstalk of the sheaves along $S^*_NM$, i.e.~the Moslov potential. Here, we are in the simple case where the projection $S^*_NM \to M$ is a smooth submanifold after a small Reeb pushoff, so fixing a perversity function is the same as fixing the Maslov potential; see also \cite[Section 2.2.1]{Jintilting} for some discussion between perversity and degrees of microstalks.
\end{remark}

\subsection{Symplectic structures on the moduli spaces}
For smooth dg categories, T\"oen--Vaqui\'e \cite{ToenVaquie} defined the derived moduli stacks of pseudo-perfect objects. Brav--Dyckerhoff \cite{Brav-Dyckerhoff2} showed that, a smooth $(n - 1)$-dimensional Calabi--Yau structure on $\sB$ induces a $(3 - n)$-shifted symplectic structure \cite{PTVV} on the derived moduli stack of pseudo-perfect objects $\bR\cM_\sB$ and a smooth relative $n$-dimensional Calabi--Yau structure on the adjunction $f: \sB \to \sA: f^r$ induces a Lagrangian structure \cite{PTVV} on the derived moduli stack of pseudo-perfect objects $\bR\cM_\sA \to \bR\cM_\sB$. In this subsection, we explain some implications when $n = 2$ or $3$.

Let $n = 2$. We discuss some examples of sheaves with microsupport on a Legendrian link, which is related to the moduli spaces of framed local systems on punctured surfaces, moduli spaces of Stokes data on punctured surfaces \cite[Section 3.3]{Shende-Treumann-Williams-Zaslow}, Richardson varieties and braid varieties \cite[Section 3]{CGGS}. See also related discussions in Shende--Takeda \cite[Section 8.2--8.4]{Shende-Takeda}.

\begin{example}
    Let $\Lambda \subseteq S^*\Sigma$ be a smooth Legendrian knot with vanishing Maslov class. Then 
    $$\Loc(\Lambda) \rightleftharpoons \Sh_\Lambda(\Sigma)_0$$
    has a smooth Calabi--Yau structure. Then we have a map between the derived moduli stacks of pseudo-perfect objects $\bR\cM_{\Sh_\Lambda(\Sigma)_0} \to \bR\cM_{\Loc(\Lambda)}$. Consider the locus of microlocal rank 1 sheaves $\bR\cM_1(\Sigma, \Lambda)$ and the locus of rank 1 local systems $\bR\cL oc_1(\Lambda) = [\bC^\times/\bC^\times]$ in the derived moduli stacks. Then $[\bC^\times/\bC^\times]$ admits a $1$-shifted symplectic structure and
    $$\bR\cM_1(\Sigma, \Lambda) \to [\bC^\times/\bC^\times]$$
    admits a shifted Lagrangian structure. Consider the map $[*/\bC^\times] \to [\bC^\times/\bC^\times]$ which fixes a trivialization of the local system on $\Lambda$. This is naturally equipped with a shifted Lagrangian structure \cite{Safronov}. Then the derived intersection
    $$\bR\cM_1(\Sigma, \Lambda) \times_{[\bC^\times/\bC^\times]} [*/\bC^\times]$$
    is equipped with a $0$-shifted symplectic structure \cite[Theorem 2.9]{PTVV}. Suppose that the derived stack $\bR\cM_1(\Sigma, \Lambda) = [\cM_1(\Sigma, \Lambda)/\bC^\times]$ where $\cM_1(\Sigma, \Lambda)$ is a smooth variety. Then
    $$\pi_0([\cM_1(\Sigma, \Lambda)/\bC^\times] \times_{[\bC^\times/\bC^\times]} [*/\bC^\times]) = \cM_1(\Sigma, \Lambda),$$
    and thus there is a symplectic form on the smooth variety $\cM_1(\Sigma, \Lambda)$.
\end{example}

Let $n = 3$. We discuss some examples of sheaves with microsupport on a Legendrian surface, which includes the case of conormal torus of smooth knots \cite{AganagicEkholmNgVafa} and Legendrian weaves \cite{Treumann-Zaslow,SchraderShenZaslow}. See also related discussions in Shende--Takeda \cite[Section 8.5]{Shende-Takeda}.

\begin{example}
    Let $\Lambda \subseteq S^*\Sigma$ be a smooth Legendrian surface of genus $g$ with vanishing Maslov class. Then 
    $$\Loc(\Lambda) \rightleftharpoons \Sh_\Lambda(\Sigma)_0$$
    has a smooth Calabi--Yau structure. We have a map between the derived moduli stacks of pseudo-perfect objects $\bR\cM_{\Sh_\Lambda(\Sigma)_0} \to \bR\cM_{\Loc(\Lambda)}$. Consider the locus of microlocal rank 1 sheaves $\bR\cM_1(\Sigma, \Lambda)$ and the locus of rank 1 local systems $\bR\cL oc_1(\Lambda)$ in the derived moduli stacks. Then $\bR\cL oc_1(\Lambda)$ admits a $0$-shifted symplectic structure and
    $$\bR\cM_1(\Sigma, \Lambda) \to \bR\cL oc_1(\Lambda)$$
    admits a shifted Lagrangian structure. We know that the 0-truncated moduli space of rank~1 local systems $t_0\bR\cL oc_1(\Lambda) = [(\bC^\times)^{2g}/\bC^\times]$. Suppose that $\bR\cM_1(\Sigma, \Lambda) = [\cM_1(\Sigma, \Lambda) /\bC^\times]$ where $\cM_1(\Sigma, \Lambda)$ is a smooth variety. Then
    $$\pi_0([\cM_1(\Sigma, \Lambda)/\bC^\times]) = \cM_1(\Sigma, \Lambda), \;\; \pi_0([(\bC^\times)^{2g}/\bC^\times]) = (\bC^\times)^{2g},$$
    and thus there is a Lagrangian structure on the map $\cM_1(\Sigma, \Lambda) \to (\bC^\times)^{2g}$.
\end{example}


\bibliographystyle{amsplain}
\bibliography{ref_KuoLi}

\end{document}